\documentclass[reqno]{amsart}
\usepackage{amsmath,amssymb,amsthm,amscd,amsfonts,mathrsfs,verbatim}
\usepackage[all]{xy}
\usepackage{tikz}
\usetikzlibrary{decorations.pathreplacing}

\numberwithin{equation}{section}

\newcommand{\Z}{{\mathbb Z}}                   
\newcommand{\R}{{\mathbb R}}                   
\newcommand{\Q}{{\mathbb Q}}                   
\newcommand{\C}{{\mathbb C}}                   
\renewcommand{\H}{{\mathbf H}}                   
\newcommand{\Mod}{{\mathcal M}}               
\renewcommand{\O}{{\mathcal O}}
\renewcommand{\L}{{\mathcal L}}
\newcommand{\CP}[1]{\mathbb{C}P^{#1}}      
\newcommand{\E}{{\mathcal E}}

\newcommand{\su}{{\mathfrak{su}}}

\newcommand{\p}{\mathfrak{p}}
\renewcommand{\d}{\operatorname{d}}

\DeclareMathOperator{\Ker}{Ker}

\DeclareMathOperator{\Gr}{Gr}
\DeclareMathOperator{\Gl}{Gl}

\DeclareMathOperator{\Dol}{Dol}
\DeclareMathOperator{\Betti}{B}
\DeclareMathOperator{\dR}{dR}
\DeclareMathOperator{\Hod}{Hod}
\DeclareMathOperator{\RH}{RH}

\DeclareMathOperator{\res}{res}

\DeclareMathOperator{\sgn}{sgn}
\DeclareMathOperator{\model}{limiting}
\DeclareMathOperator{\tr}{tr}

\DeclareMathOperator{\Ad}{Ad}
\DeclareMathOperator{\fid}{fid}

\DeclareMathOperator{\app}{app}

\DeclareMathOperator{\Sl}{Sl}

\DeclareMathOperator{\U}{U}

\newtheorem{theorem}{Theorem}
\newtheorem{prop}{Proposition}
\newtheorem{remark}{Remark}
\newtheorem{defn}[prop]{Definition}

\title[$P=W$ conjecture for the $5$-punctured sphere]
{$P=W$ conjecture in lowest degree for rank $2$ over the $5$-punctured sphere}

\author{Szil\'ard Szab\'o}
\address{Budapest University of Technology and Economics, 1111. Budapest,
Egry J\'ozsef utca 1. H \'ep\"ulet, Hungary, and 
R\'enyi Institute of Mathematics, 1053. Budapest, Re\'altanoda
utca 13-15. Hungary}
\email{szabosz@math.bme.hu, szabo.szilard@renyi.hu}

\begin{document}

\begin{abstract}
We use abelianization of Higgs bundles away from the ramification divisor 
and fiducial solutions to analyze the large scale behaviour of Fenchel--Nielsen 
co-ordinates on the moduli space of rank $2$ Higgs bundles on the Riemann 
sphere with $5$ punctures. We solve the related Hitchin WKB problem and prove 
the lowest degree weighted pieces of the $P=W$ conjecture in this case. 
\end{abstract}

\maketitle

\section{Introduction and statement of the main result}

In this paper we investigate the moduli space $\Mod_{\Dol}$ of Higgs bundles
on $\CP1$ with $5$ logarithmic points in rank $2$ and the corresponding character 
variety $\Mod_{\Betti}$, subject to specific choices of parameters. 
These spaces are complex varieties of dimension $4$. 
The first aim  of the paper is to give a complete answer 
(Propositions~\ref{prop:asymptotic_RH},~\ref{prop:monodromy_zeta2},~\ref{prop:monodromy_zeta3}) 
for these spaces to the Hitchin WKB problem raised in~\cite{KNPS}: \newline \noindent 
\textbf{Hitchin WKB problem} Consider a non-trivial $\C^{\times}$-orbit in the 
Hitchin base and a family of Higgs bundles lifting this orbit in the Dolbeault 
moduli space; determine then the asymptotic behaviour of the transport matrices 
of the associated family of representations in the character variety. 

The second, closely related goal is to use these results in order to obtain 
for these spaces one extremal graded piece of the so-called $P=W$ conjecture:  
\begin{theorem}\label{thm:main}
 Let $\Mod_{\Dol}$ and $\Mod_{\Betti}$ denote the Dolbeault moduli space and 
 character variety of $\CP1$ with $5$ logarithmic points in rank $2$. 
 Then, for every $0 \leq k \leq 4$ the regular singular Riemann--Hilbert 
 correspondence and the non-abelian Hodge correspondence
 induce an isomorphism 
 \begin{equation*}
      \Gr_P^{-k-2} H^k (\Mod_{\Dol}, \Q ) \cong \Gr^W_{2k} H^k (\Mod_{\Betti}, \Q ). 
 \end{equation*}
\end{theorem}
The weights appearing in the theorem represent the lowest (respectively, highest)  
possibly non-trivial weights of $P$ (respectively, $W$). 
For our weight conventions, see Sections~\ref{subsec:weight_filtration} 
and~\ref{subsec:perverse_filtration}. 
Notice that since $\Mod_{\Betti}$ is a smooth $4$-dimensional affine variety, 
by the Andreotti--Frankel theorem~\cite{Andreotti_Frankel} the only degrees 
where it may have non-trivial cohomology are $0 \leq k\leq 4$. 

Even though this paper contains a detailed study of just one special case of the $P=W$ 
conjecture in lowest weight, many of our technical results are valid for an 
arbitrary number $n\geq 5$ of parabolic points. 
We expect that it will be possible to treat more general cases along the same lines. 

Let us now give some motivational background for this study. 
The $P=W$ conjecture is a major open problem in non-abelian Hodge theory,
formulated by M.~de~Cataldo, T.~Hausel and L.~Migliorini~\cite{HdCM}
as a correspondence between the (decreasing) perverse Leray filtration $P$ 
induced by the Hitchin map on the cohomology of a Dolbeault moduli space and 
the (increasing) weight filtration $W$ of Deligne's mixed Hodge structure on 
the cohomology of the associated character variety (Betti space).
Recent years have brought intense activity toward establishing various cases
of the $P=W$ conjecture.
The starting point was the paper of de~Cataldo, Hausel and Migliorini~\cite{HdCM} 
where it was proven in rank $2$ over compact curves. 
As an important further contribution, M.~de~Cataldo, D.~Maulik and J.~Shen~\cite{dCMS}
established it for curves of genus $2$.
C.~Felisetti and M.~Mauri~\cite{FM} proved it for character varieties admitting a
symplectic resolution, i.e. in genus $1$ and arbitrary rank, and in genus $2$ and rank $2$.
The author has established the conjecture for complex $2$-dimensional
moduli spaces of rank $2$ Higgs bundles with irregular singularities over $\CP1$
corresponding to the Painlev\'e cases~\cite{Sz_PW}.
J.~Shen and Z.~Zhang~\cite{Shen_Zhang} proved it for five infinite families of
moduli spaces of parabolic Higgs bundles over $\CP1$.

The $P=W$ conjecture has been generalized in various interesting contexts.
A.~Harder showed a similar statement for elliptic Lefschetz fibrations using
methods coming from toric surfaces~\cite[Theorem~4.5]{Harder}.
Z.~Zhang~\cite{Zhang_cluster} found a related phenomenon for the weight filtration of
certain $2$-dimensional cluster varieties and the perverse filtration of elliptic
fibrations with constrained singular fibers.
A motivation for the $P=W$ conjecture  was the so-called curious hard Lefschetz
conjecture of T.~Hausel, E.~Letellier and F.~Rodriguez-Villegas~\cite{HLRV}, 
that has been confirmed by A.~Mellit~\cite{Mellit}. 

Among the various generalizations and analogues of the $P=W$ conjecture is an 
intriguing geometric counterpart formulated by L.~Katzarkov, A.~Noll, P.~Pandit and C.~Simpson~\cite[Conjecture 1.1]{KNPS} and C.~Simpson~\cite[Conjecture 11.1]{Sim} 
that is quite relevant to our approach; this version is now called Geometric 
$P=W$ conjecture. 
Roughly speaking, the Geometric $P=W$ conjecture asserts the existence of a certain 
homotopy commutative diagram involving the Riemann--Hilbert map, non-abelian 
Hodge correspondence, the Hitchin map and the natural map from the character 
variety to the topological realization of its dual boundary complex. 
An immediate consequence of validity of this conjecture is that
the homotopy type of the topological space of the dual boundary complex of the
character variety is that of a sphere of given dimension, therefore finding
this homotopy type is a first consistency check of the conjecture. 
The Geometric $P=W$ conjecture has also attracted considerable attention in recent times.
A.~Komyo~\cite{Kom} used an explicit geometric description to prove that the
homotopy type of the dual boundary complex of the character variety for $\CP1$ with
$5$ logarithmic points and group $\Gl(2,\C)$ (that is, the Betti space we will 
deal with in this paper) is that of the $3$-sphere. 
C.~Simpson~\cite{Sim} generalized Komyo's result to the case of arbitrarily
many logarithmic points on $\CP1$, in rank $2$, by proving that the
homotopy type of the dual boundary complex is that of $S^{2n-7}$;
for this purpose, he introduced Fenchel--Nielsen type co-ordinates that
will be widely used in this paper. 
T.~Mochizuki~\cite{Moc} solved the closely related Hitchin WKB problem for 
non-critical paths. 
M.~Mauri, E.~Mazzon and M.~Stevenson~\cite[Theorem~6.0.1]{MMS} used Berkovich
space techniques to show that the dual boundary complex of a log-Calabi--Yau
compactification of the $\Gl(n,\C)$ character variety of a $2$-torus is
homeomorphic to $S^{2n-1}$.
L.~Katzarkov, A.~Harder and V.~Przyjalkowski have formulated a version of the
cohomological $P=W$ conjecture for log-Calabi--Yau manifolds and their mirror
symmetric pairs, and in~\cite[Section~4]{KHP} discussed a geometric version thereof.
The author established the Geometric $P=W$ conjecture in the
Painlev\'e cases in~\cite{Sz_PW} 
via asymptotic
abelianization of solutions of Hitchin's equations. In joint work with
A.~N\'emethi~\cite{N-Sz}, the author gave a second proof for the same cases
using different techniques, namely plumbing calculus.
As far as the author is aware, up to date these latter articles are the only 
ones in which the full assertion of the Geometric $P=W$ conjecture has been
confirmed, rather than just its implication on the homotopy type of the
dual boundary complex.
It is remarkable that the geometrical understanding
of the moduli spaces developed in~\cite[Section~6]{N-Sz} is quite reminiscent
to the description of the weight filtration in terms of dual torus fibrations
appearing in~\cite[Section~4]{KHP} (up to the difference that the latter paper
deals with the case of a smooth elliptic anti-canonical divisor rather than a singular one).

Previously, F.~Loray and M.~Saito~\cite{LS} studied the algebraic geometric structure 
of the moduli space that we consider, endowed with its de Rham complex structure. 
R.~Donagi and T.~Pantev~\cite{DP} investigated Hecke transforms on this space and 
proved the Geometric Langlands correspondence for it. 
The paths that we consider in
Propositions~\ref{prop:asymptotic_RH},~\ref{prop:monodromy_zeta2},~\ref{prop:monodromy_zeta3}
are homologically non-trivial loops that do not satisfy the non-critical 
condition, therefore our results do not directly follow from previous study 
of T.~Mochizuki~\cite{Moc} (though we make use of methods of that paper). 

We will achieve our goals by refining the approach pioneered in our 
previous paper~\cite{Sz_PW}. 
Namely, using asymptotic abelianization we reduce the study to the classical 
abelian Hodge theory and Riemann--Hilbert correspondence treated in detail 
for instance in~\cite{Goldman_Xia}. 
Specifically, we will make use of technical results of T.~Mochizuki~\cite{Moc},
L.~Fredrickson, R.~Mazzeo, J.~Swoboda and H.~Weiss~\cite{FMSW} and
R.~Mazzeo, J.~Swoboda, H.~Weiss and ~F.~Witt~\cite{MSWW} describing the
large-scale behaviour of solutions of Hitchin's equations, combined with
C.~Simpson's Fenchel--Nielsen type co-ordinates of the character variety~\cite{Sim}.
The studies in~\cite{FMSW} and~\cite{MSWW} were inspired by physical 
considerations pertinent to the WKB-analysis of Hitchin's equations
given by D.~Gaiotto, G.~Moore and A.~Neitzke~\cite{GMN}, where the authors stated 
a conjecture about the large scale Riemannian structure of the Hodge moduli spaces. 
In a certain sense, our work therefore points out a connection between two 
seemingly unrelated circles of ideas: the $P=W$ conjecture on the algebraic 
topology of the Hodge moduli spaces on the one hand, and the Gaiotto--Moore--Neitzke 
conjecture on their Riemannian geometry on the other hand. 
This fits nicely into the broader picture of topology and Riemannian geometry having 
influence on one another, the bridge between them being built by geometric analysis. 
In particular, it will become clear from our proofs that integrals of the canonical 
Liouville $1$-form along some paths in the spectral cover govern the behaviour of 
the non-abelian Hodge and the Riemann--Hilbert correspondences; the same kind of 
integrals show up in~\cite{GMN}. 

\noindent {\bf Acknowledgements:} The author would like to thank T.~Hausel, 
M.~Mauri, R.~Mazzeo, T.~Mochizuki, A.~N\'emethi and C.~Simpson for useful discussions.
 During the preparation of this manuscript, the author was supported by the \emph{Lend\"ulet} Low Dimensional Topology grant of the Hungarian Academy of 
 Sciences and by the grants K120697 and KKP126683 of NKFIH.

\section{Basic notions and preparatory results}\label{sec:notations}

\subsection{Moduli spaces of tame harmonic bundles}

Consider $X = \CP1$ with co-ordinates $z$ and $w = z^{-1}$, endowed with the standard Riemannian metric. 
We denote by $\O$ and $K$ the sheaves of holomorphic functions and holomorphic $1$-forms respectively on $\CP1$. 
We set 
\begin{equation}\label{eq:divisor}
 t_0 = -\frac 1k, \quad t_1 = 0, \quad t_2 = 1, \quad t_3 = -1, \quad t_4 = \frac 1k
\end{equation}
for some $0 < k < 1$. 
These choices will not be used until Section~\ref{sec:Geometry_period}, 
so the results of all preceding sections are valid for any quintuple of 
distinct points in $\CP1$. 
We consider the simple effective divisor  
$$
  D = t_0 + t_1 + t_2 + t_3 + t_4 
$$
and set 
$$
  L = K (D). 
$$
By an abuse of notation, we will also denote by $D$ the support set of $D$. 
Finally, we fix a point $x_0 \in \CP1 \setminus D$. 
Much of the following discussion has a straightforward generalization to simple effective divisors of higher length too. 

For $0\leq j \leq 4$ we fix 
\begin{equation}\label{eq:Dolbeault_weights}
   \alpha_j^- = \frac 14, \quad  \alpha_j^+ = \frac 34  
\end{equation}
that will serve as parabolic weights in the Dolbeault complex structure. Notice that 
$$
  \sum_{j=0}^4 ( \alpha_j^- +  \alpha_j^+ ) = 5. 
$$
We will write 
$$
  \vec{\alpha} = (\alpha_j^-, \alpha_j^+ )_{j=0}^4 . 
$$

The basic object of our study will be a certain Hodge moduli space $\Mod_{\Hod}$ of 
\emph{tame harmonic bundles}~\cite{Sim_Hodge} of rank $2$ and 
parabolic degree $0$ on $\CP1$ with parabolic structure at $D$. 
We will describe this moduli space from two perspectives called the Dolbeault and the de Rham moduli spaces. 
Consider a smooth vector bundle $V$ of rank $2$ and degree $-5$ over $\CP1$. 
Then, the equations defining harmonic bundles are 
\emph{Hitchin's equations}~\cite{Hit} 
\begin{align}
 \bar{\partial}_{\E} \theta & = 0 \label{eq:complex_Hitchin} \\
 F_h + [\theta , \theta^{\dagger} ] & = 0 \label{eq:real_Hitchin}
\end{align}
for a $(0,1)$-connection $\bar{\partial}_{\E}$ on $V$, a Hermitian metric $h$ on $V$ and a section $\theta$ of $End(V)\otimes \Omega^{1,0}_{\CP1}$ over $\CP1 \setminus D$, 
where $F_h$ is the curvature of the Chern connection $\nabla_{h}$ associated to $(\bar{\partial}_{\E}, h)$ and $\theta^{\dagger}$ is the section of $End(V)\otimes \Omega^{0,1}_{\CP1}$ 
obtained by taking the adjoint of the endomorphism part of $\theta$ with respect to $h$ and the complex conjugate of its form-part. 
The reason of the terminology ``harmonic bundle'' is the fact that the equations imply that the map $h$ is harmonic from (the universal cover of) 
the Riemann surface to the symmetric space $\mbox{Gl}_2(\C )/ \mbox{U}(2)$. 
The behaviour of $\theta$ and $h$ is assumed to satisfy the so-called \emph{tameness} condition at each $t_j$, namely $h$ should admit a lift along any ray to $t_j$ which grows 
at most polynomially in Euclidean distance. 

Hitchin's equations are presented above from the Dolbeault point of view. Let us first describe the boundary behaviour of the data from this perspective. 
Let us denote by $\E$ the holomorphic vector bundle $(V, \bar{\partial}_{\E} )$ on $\CP1 \setminus D$. 
It turns out that there exists an extension of the holomorphic bundle $\E$ over $D$ such that the Higgs field has at most logarithmic poles at $D$. 
A \emph{parabolic structure} on $\E$ at $D$ is by definition a filtration 
\begin{equation}\label{eq:Dolbeault_parabolic_structure}
   0 \subset \ell_j \subset \E|_{t_j}
\end{equation}
of the fiber of $\E$ at every $t_j \in D$ that is stabilized by $\mbox{res}_{t_j} \theta$. 
We assume that the Higgs field $\theta$ is {\it strongly parabolic}, 
meaning that the action of $\mbox{res}_{t_j} \theta$ both on $\ell_j$ and on 
$\E|_{t_j}/ \ell_j$ is trivial. 
Then, in the Dolbeault complex structure $\Mod_{\Hod}$ parameterizes $\vec{\alpha}$-stable parabolic Higgs bundles with Higgs field having at most logarithmic poles at $D$ such that 
the eigenvalues of the residue of the associated Higgs field at $t_j$ vanish and the parabolic weights of the underlying holomorphic vector bundle 
in the Dolbeault picture at $t_j$ are equal to $\alpha_j^{\pm}$. The latter assumption on parabolic weights encodes a certain growth behaviour of the evaluation of the metric $h$ 
on elements of a local holomorphic trivialization. The moduli space of such logarithmic parabolic Higgs bundles is known to be a $\C$-analytic manifold  
$$
  \Mod_{\Dol} (\vec{0}, \vec{\alpha})
$$
called \emph{Dolbeault moduli space}, whose underlying smooth manifold is $\Mod_{\Hod}$. 

Let us now turn to the de Rham point of view. 
It is known that if $(\bar{\partial}_{\E}, h, \theta)$ is a tame harmonic bundle then the connection 
$$
  \nabla = \nabla_h + \theta + \theta^{\dagger}
$$
is integrable, and the underlying holomorphic vector bundle admits an extension over $D$ with respect to which $\nabla^{1,0}$ has regular singularities. 
The associated de Rham moduli space parameterizes $\vec{\beta}$-stable parabolic integrable connections on $V$ with regular singularities near the punctures $t_j$, 
with eigenvalues of its residue given by 
\begin{equation}\label{eq:dR_eigenvalues}
  \mu_j^{\pm} =  \alpha_j^{\pm} 
\end{equation}
and parabolic weights given by 
\begin{equation}\label{eq:dR_weights}
 \beta_j^{\pm} = \alpha_j^{\pm} . 
\end{equation}
Again, a parabolic structure on the underlying holomorphic vector bundle at $D$ is defined as a flag of its fiber over $t_j \in D$ that is stabilized by 
$\mbox{res}_{t_j} \nabla^{1,0}$ and such that its action on the first graded piece of the filtration be $\mu_j^-$. 
The \emph{de Rham moduli space} of such parabolic integrable connections with regular singularities will be denoted by 
$$
  \Mod_{\dR} (\vec{\mu}, \vec{\beta}); 
$$
it is a $\C$-analytic manifold  with underlying smooth manifold $\Mod_{\Hod}$. 

It follows from the above discussion that there exists a canonical diffeomorphism 
\begin{equation}\label{eq:NAHC}
 \psi\colon \Mod_{\Dol} (\vec{0}, \vec{\alpha}) \to \Mod_{\dR} (\vec{\mu}, \vec{\beta})
\end{equation}
called \emph{non-abelian Hodge correspondence}.

\subsection{Character variety, Riemann--Hilbert correspondence, 
dual boundary complex}\label{ssec:RH}
We will need a third point of view of harmonic bundles, called Betti side. 
The \emph{Betti moduli space} (or character variety) $\Mod_{\Betti}(\vec{c}, \vec{0})$ parameterizes filtered local systems on $\CP1 \setminus D$ with prescribed conjugacy 
class of its monodromy around every $t_j$ and growth order of parallel sections on rays emanating from the punctures, up to simultaneous conjugation by elements of $\mbox{PGl}(2,\C)$. 
We will now describe the value and role of parameters $\vec{c}$. 
Namely, the monodromy transformation of an integrable connection $\nabla$ in $\Mod_{\dR} (\vec{\mu}, \vec{\beta})$ along a positively oriented simple loop in $\CP1$ separating 
$t_j$ from the other parabolic points has eigenvalues 
\begin{equation}\label{eq:Betti_eigenvalues}
 c_j^{\pm} = \exp (-2\pi \sqrt{-1} \mu_j^{\pm}) = \exp (-2\pi \sqrt{-1} \alpha_j^{\pm} ) = \pm \sqrt{-1} 
\end{equation}
and all weights of the associated filtration equal to $0$.

It is known that the map 
\begin{equation}\label{eq:RH}
   \RH\colon \Mod_{\dR} (\vec{\mu}, \vec{\beta}) \to \Mod_{\Betti}(\vec{c}, \vec{0})
\end{equation}
mapping any integrable connection to its (filtered) local system of vector spaces is a $\C$-analytic isomorphism, called \emph{Riemann--Hilbert map}. 

It is known that $\Mod_{\Betti}(\vec{c}, \vec{0})$ is an affine algebraic variety, which is smooth for generic choices of the parameters. 
We will denote by $\overline{\Mod}_{\Betti}(\vec{c}, \vec{0})$ a smooth compactification by a simple normal crossing divisor $D_{\Betti}$. 
Such a compactification exists by Nagata's compactification theorem~\cite{Nag} 
combined with Hironaka's theorem on the existence of resolutions of
singularities in characteristic $0$~\cite{Hir}. 
\begin{defn}
  The \emph{dual complex} of $D_{\Betti}$ is the simplicial complex $\mathbb{D} D_{\Betti}$  whose vertices are in bijection with irreducible components of $D_{\Betti}$, 
  and whose $k$-faces are formed by $(k+1)$-tuples of vertices such that the intersection of the corresponding components is non-empty. We will denote the 
  $k$-skeleton of $\mathbb{D}D_{\Betti}$ by $\mathbb{D}_k D_{\Betti}$, and the topological realization of $\mathbb{D} D_{\Betti}$ by $|\mathbb{D} D_{\Betti}|$. 
\end{defn}
We will require $D_{\Betti}$ to be a very simple normal crossing divisor, meaning that any such non-empty intersection of components is connected. 
The above procedure may be applied to any quasi-projective smooth variety $X$, and an important result due to Danilov~\cite{Dan} states that the homotopy type 
of the simplicial complex is independent of the chosen compactification. 
We will apply it to $\Mod_{\Betti}(\vec{c}, \vec{0})$, and we will call the resulting simplicial complex its \emph{dual boundary complex}, 
denoted by $\mathbb{D} \partial \Mod_{\Betti}(\vec{c}, \vec{0})$. 
A.~Komyo~\cite{Kom} showed that for character varieties of rank $2$ 
representations with $k = 5$ parabolic points the homotopy type of the dual 
boundary complex is that of the sphere $S^3$. 
C.~Simpson~\cite{Sim} generalized this result to character varieties of 
the complement of $k\geq 5$ parabolic points, by showing that for 
$X = \Mod_{\Betti}(\vec{c}, \vec{0})$ the dual boundary complex is homotopy 
equivalent to the sphere $S^{2k-7}$. 

\subsection{Topological description of the weights in mixed Hodge structure}
\label{subsec:weight_filtration}

Another closely related consequence of the fact that 
$\Mod_{\Betti}(\vec{c}, \vec{0})$ is a smooth affine algebraic variety is that 
its cohomology spaces carry a mixed Hodge structure (MHS), defined by 
P.~Deligne~\cite{DelHodge2}. Let us recall the topological characterization of 
the weights in MHS, following~\cite[Section~6.5]{EN}. 
In what follows, we will often drop $\vec{c}, \vec{0}$ from the notation and 
write $\Mod_{\Betti}$ for the character variety. 

In this section we adopt the point of view 
of~\cite{EN} and consider homology groups rather than cohomology; 
application of the standard duality operation is implicitly meant whenever 
we compare a homology group with a cohomology group. 
This involves switching the signs of the degrees of the weight filtration. 
Let $\overline{\Mod}_{\Betti}$ be a smooth compactification of $\Mod_{\Betti}$ by 
a simple normal crossing divisor $D_{\Betti}$. 
We spell out the general construction of the mixed Hodge structure of 
$X\setminus Y$ given in~\cite{EN} for $X = \overline{\Mod}_{\Betti}$ 
and $Y = D_{\Betti}$. 

The filtration is the abutment of the spectral sequence associated to a double 
complex $A_{**}$ endowed with a filtration $W$. 
For any $p\geq 1$ we denote by $\tilde{D}^p$ the disjoint union of the $p$-fold 
intersections of the irreducible components of $D_{\Betti}$, and set 
$\tilde{D}^0 = \overline{\Mod}_{\Betti}$. 
We denote by $C_t^{\pitchfork} (\tilde{D}^s)$ the free abelian group generated by 
dimensionally transverse $t$-cycles in $\tilde{D}^s$, i.e. cycles for the 
$0$-perversity function. 
We let 
$$
   A_{s,t} = C_t^{\pitchfork} (\tilde{D}^{-s}), 
$$
where $s\leq 0, t\geq 0$. The filtration $W$ is defined by 
$$
    W_{s} = \bigoplus_{p\leq s}  A_{p,t}. 
$$
There exists a well-defined intersection morphism 
$$
    \cap\colon C_t^{\pitchfork} (\tilde{D}^s) \to C_{t-2}^{\pitchfork} (\tilde{D}^{s+1}) 
$$
compatible with $W$, turning $A_{**}$ into a filtered double complex. 
It is shown in~\cite[Theorem~1.5]{EN} that the associated spectral sequence 
$E_{st}^r$ degenerates at page $r = 2$ and abuts to the filtration 
$$
    E_{st}^{\infty } \otimes \Q = \Gr^W_{-t} H_{s+t} (\Mod_{\Betti}, \Q ). 
$$
The filtration $W$ on the right-hand side is then equal to Deligne's weight filtration. 

The topological representatives of  $\Gr^W_{-2k} H_{k}$ corresponding to the 
choices $t = 2k$ and $s = -k$ are generated by classes of the following form 
(for the similar cases $k=1$ and $k=2$ over surfaces see~\cite[Example~6.9]{EN}). 
Take a generic point $Q$ in the $k$-fold intersections of the divisors 
$$
    Q \in \tilde{D}^k \setminus \tilde{D}^{k+1} .
$$
Let the corresponding divisor components be denoted without loss of generality 
$Y_1, \ldots , Y_k$. 
The preimage $\Pi^{-1} (Q)$ of $Q$ in the normal bundle of $Y_1 \cap \cdots \cap Y_k$ 
in $X\setminus Y = \Mod_{\Betti}$ deformation retracts onto a $k$-dimensional real torus 
(the boundary of a tubular neighbourhood of $Y_1\cap \cdots \cap Y_k$). 
If one considers all $k$-tuple intersections of divisor components, then the classes 
of these tori generate $\Gr^W_{-2k} H_{k}$, and the dual cohomology classes 
generate $\Gr^W_{2k} H^k$.

\subsection{Hitchin maps and bases}

\begin{prop}\label{prop:Hitchin_base}
\begin{enumerate}
 \item For every strongly parabolic $\vec{\alpha}$-stable Higgs bundle $(\E , \theta )$ with logarithmic singularities at $D$ we have 
 $$
    \mbox{tr} (\theta ) \equiv 0. 
 $$
 \item There exists a linear subspace 
$$
  \mathcal{B} = \{ q : \quad q(t_j ) = 0 \;\;\mbox{for all}\;\; 0 \leq j \leq 4 \} \subset H^0 (\CP1, L^{\otimes 2}) \cong \C^7 
$$
of dimension $2$ over $\C$ such that an $\vec{\alpha}$-stable Higgs bundle $(\E , \theta )$ with logarithmic singularities at $D$ is strongly parabolic if and only if 
$$
   \det (\theta ) \in  \mathcal{B}. 
$$
\end{enumerate}
\end{prop}

\begin{proof}
We have 
\begin{align*}
   \mbox{tr} (\theta ) & \in H^0 (\CP1 , L) \cong \C^4 \\
   \det (\theta ) & \in H^0 (\CP1 , L^{\otimes 2}) \cong \C^7. 
\end{align*}
The requirement on the eigenvalues of the residues of $\theta$ together imposes $5$ linear  relations on $\mbox{tr} (\theta )$; 
however, one of these conditions expresses that the sum of the eigenvalues is 0, and is therefore redundant. So, $\mbox{tr} (\theta )$ is uniquely determined as $0 \in H^0 (\CP1 , L)$, proving the first assertion.  
On the other hand, the same requirements impose $5$ independent linear relations on $\det (\theta )$. The second assertion follows. 
\end{proof}

From now on, we will often let $\Mod_{\Dol}$ stand for 
$\Mod_{\Dol} (\vec{0}, \vec{\alpha})$. 
It follows from the Proposition that we have a well-defined map 
\begin{align}\label{eq:Hitchin_map}
 H\colon \Mod_{\Dol} & \to \mathcal{B} \\
 (\E, \theta ) & \mapsto - \det (\theta ) \notag
\end{align}
called the \emph{Hitchin map}. (The negative sign will simplify our later formulas.) The target space $\mathcal{B}$ of $h$ is called the \emph{Hitchin base}. 
The generic element of $\mathcal{B}$ will be denoted as 
$$
    q \in H^0 (\CP1 , L^{\otimes 2}). 
$$
We fix the isomorphism $\O (3 ) \cong  L$ given on the affine open subset $w\neq 0$ by 
\begin{equation}\label{eq:isomorphism_L}
   s (z,w) =  \sum_{i=0}^3 s_i z^{3-i} w^i  \mapsto S(z) = s(z,1) 
   \frac{\mbox{d}z}{\prod_{j=0}^4(z-t_j)} .
\end{equation}
Under this isomorphism, the value $s(t_j,1)$ for $0\leq j \leq 4$ 
is equal to the some non-zero multiple (only depending on the divisor $D$ and $j$) 
of the residue $\mbox{res}_{t_j} (S)$. 
The isomorphism~\eqref{eq:isomorphism_L} induces the isomorphism $\O (6 ) \cong  L^{\otimes 2}$ given by 
\begin{equation}\label{eq:isomorphism_L2}
   q(z,w) = \sum_{i=0}^6 q_i z^{6-i} w^i \mapsto Q(z) = q(z,1)   \frac{\mbox{d}z^{\otimes 2}}{\prod_{j=0}^4(z-t_j)^2}.
\end{equation}

\subsection{Spectral curve, Jacobian variety}\label{subsec:spectral_curve}

We consider the total space $\mbox{Tot} ( L )$ of $L$ with the natural projection  
$$
  p_L \colon \mbox{Tot} ( L ) \to \CP1 . 
$$
We denote by $\zeta$ the canonical section of $p_L^* L$. For any 
$$
    q \in H^0 (\CP1 , L^{\otimes 2})
$$
we denote for simplicity $p_L^* q$ by $q$. 

We endow $\mathcal{B}$ with a scalar product (we will be more precise in ~\eqref{eq:norm}), 
pick $R> 0$ and let $S^3_R$ denote the sphere of radius $R$ in $\mathcal{B}\cong \C^2$.
For $q\in S^3_1$ we write $\zeta_{\pm} (Rq,z)$ for the roots of 
$$
  \zeta^2 - R q = 0, 
$$
specifically 
\begin{equation}\label{eq:tilde_zeta}
   \zeta_{\pm} (Rq,z) = \pm \sqrt{R q (z,1)}. 
\end{equation}
We denote by 
\begin{equation}\label{eq:nilpotent_spectral_curve}
   X_{R q} = \{ ([z:w], \pm \sqrt{R q (z,w)} ) \} \subset \mbox{Tot} ( L )
\end{equation}
the Riemann surface of the bivalued function $\zeta_{\pm} (Rq,z)$. 
For a generic choice of $q$ this curve is smooth and of genus 
$$
  g ( X_{R q} ) = 2. 
$$
For generic $q\in S^3_1$, the fiber $H^{-1} (q)$ is smooth, and known to be isomorphic 
to an abelian variety of dimension $2$ over $\C$, namely (a torsor over) the Jacobian 
$\operatorname{Jac}(X_{q})$ of $X_{q}$: 
\begin{equation}\label{eq:generic_Hitchin_fiber}
 H^{-1} (q) \cong \operatorname{Jac}(X_{q}) =  H^{0,1} (X_q) / \Lambda_q
\end{equation}
for the period lattice $\Lambda_q \subset H^{0,1} (X_q) \cong \C^2$ of $X_{R q}$. 
Recall that 
$$
    \Lambda_q = \operatorname{Im} \left(p^{0,1} \circ \iota \right)
$$
where the map 
$$
    \iota \colon  H^1 (X_q, 2 \pi \sqrt{-1} \Z ) \to H^1  (X_q, \C ) 
$$
is induced by the coefficient inclusion $2 \pi \sqrt{-1} \Z \to \C$ and the map 
\begin{equation}\label{eq:projection}
     p^{0,1} \colon H^1  (X_q, \C ) \to H^{0,1} (X_q)
\end{equation}
is projection of harmonic forms to their antiholomorphic part. 
Then, for given $\mu_1, \mu_2 \in H^{0,1} (X_q)$ the relation 
$$
    \mu_1 - \mu_2 \in \Lambda_q
$$
is equivalent to the following condition: for every $1$-cycle $A$ on $X_q$ with
coefficients in $\Z$ we have 
$$
    \int_A ( \mu_1 - \mu_2 ) \in 2 \pi \sqrt{-1} \Z. 
$$

The abelian version of the Hodge correspondence $\psi$ of~\eqref{eq:NAHC} on 
$X_q$ states that any class in $H^{0,1} (X_q)$ may be represented by an 
anti-holomorphic form, i.e. $\mu \in \Omega^{0,1} (X_q)$ 
satisfying $\partial \mu = 0$, and that then the $\U(1)$-connection on the 
trivial line bundle defined by the connection form 
$$
    B = \mu - \bar{\mu} \in \Omega^1 (X_q)
$$ 
is flat, see~\cite[Proposition~4.1.5]{Goldman_Xia}. 
With this notation, fixing any basis $A_1, A_2, B_1, B_2 \in H_1 (X_q, \Z)$, 
the abelian version of $\RH\circ\psi$ (where $\RH$ is the Riemann--Hilbert 
correspondence~\eqref{eq:RH}) is then the diffeomorphism between the Jacobian 
and the $4$-torus given by 
\begin{align*}
    \operatorname{Jac}(X_{q}) & \to T^4 = (S^1)^4 \\
    \mu & \mapsto \left( e^{\int_{A_1} B}, e^{\int_{A_2} B}, 
    e^{\int_{B_1} B}, e^{\int_{B_2} B} \right). 
\end{align*}

\subsection{Ramification of spectral curve}\label{subsec:ramification}

Clearly, $X_{R q}$ is ramified over $D$. 
Let $\tilde{D}$ denote the corresponding branch 
divisor, so $\tilde{D}$ consists of the preimages of the points of $D$ on 
$X_{R q}$, all counted with multiplicity $1$. 
We again let 
\begin{equation}\label{eq:Z_tilde}
  Z_{\pm} (Rq, z)
\end{equation}
stand for the bivalued meromorphic differentials corresponding to~\eqref{eq:tilde_zeta} over the chart $z$ under the isomorphism~\eqref{eq:isomorphism_L}. 
In concrete terms, we have 
\begin{equation}\label{eq:Z_tilde2}
 Z_{\pm} (Rq, z) = \pm \sqrt{R q (z,1)} \frac{\mbox{d}z}{\prod_{j=0}^4(z-t_j)}
\end{equation}
We set
$$
    \Delta_q = \{ z \in \C \colon \quad q(z) = 0 \}. 
$$
Regardless of the value of $R>0$, $\Delta_q$ is the ramification divisor of the projection map 
\begin{equation}\label{eq:projection_X}
   p_{R q} : X_{R q} \to \CP1 
\end{equation}
induced by $p_L$. 
$\Delta_q$ contains the points of $D$ by Proposition~\ref{prop:Hitchin_base}, and is of cardinality $6$ because $\mbox{deg} (L^{\otimes 2}) = 6$. 
It follows that it is of the form 
\begin{equation}\label{eq:discriminant2}
   \Delta_q = \{ t_0 , t_1 , t_2 , t_3 , t_4 , t(q) \} 
\end{equation}
for some $t(q)\in\CP1$. In case $t(q) = t_{j_0}$ for some $0 \leq j_0 \leq 4$, we assign multiplicity $2$ to $t_{j_0}$ in $\Delta_q$. 
On the other hand, for any fixed $t\in\CP1$ we denote by $\Delta_t$ the set of 
$q \in S^3_1$ such that $t \in \Delta_q$. 
We denote by $\tilde{\Delta}_q$ the corresponding ramification points on $X_q$, 
and similarly by $\tilde{t}(q), \tilde{D}$ the lifts of $t(q)$ and of the 
divisor $D$, respectively. 

\begin{prop}\label{prop:Hopf}
 For any fixed $t\in\CP1$, the set $\Delta_t$ is diffeomorphic to $S^1$, and the map 
 \begin{align*}
    t\colon S^3_1 & \to \CP1 \\
    q & \mapsto t(q) 
 \end{align*}
 defined by~\eqref{eq:discriminant2} is the Hopf fibration. 
\end{prop}

\begin{proof}
 The section $q$ is a homogeneous polynomial of degree $6$, vanishing at the points of $D$ by Proposition~\ref{prop:Hitchin_base}, hence is of the form 
 $$
  q (z,w) = (a z - bw) \prod_{j=0}^4 (z - t_jw)  
 $$
 for some $(a,b)\in\C^2\setminus \{ (0,0) \}$. Using the isomorphism~\eqref{eq:isomorphism_L2}, the corresponding meromorphic quadratic differential reads as 
 \begin{equation}\label{eq:Hopf}
    Q (z) 
    = \frac{(a z - b) \mbox{d} z^{\otimes 2}}{\prod_{j=0}^4 (z - t_j)}. 
 \end{equation}
 The coefficients $(a,b)$ describe natural co-ordinates of the space $\mathcal{B} \cong \C^2$. 
 For any fixed $[z_0 : w_0]\in\CP1$ the condition $a z_0  - b w_0 = 0$ is linear in $(a,b)$, hence $\Delta_t$ is the link of a line passing through $0$ in $\mathcal{B}$. 
 This shows the first assertion. 
 
 The map appearing in the second assertion is 
 $$
  (a,b) \mapsto [b:a]. 
 $$
 This is just the canonical map from $\C^2\setminus \{ (0,0) \}$ to $\CP1$. The result follows. 
\end{proof}

 Using the factorization~\eqref{eq:Hopf} we will assume that the norm of $\mathcal{B}$ 
 is 
 \begin{equation}\label{eq:norm}
    | q | = \sqrt{|a|^2 + |b|^2}.
 \end{equation}
 We use standard Hopf co-ordinates 
 \begin{equation}\label{eq:Hopf_coordinates}
    a = \cos (\theta ) e^{\sqrt{-1} (\varphi - \phi)}, \quad b = \sin (\theta ) e^{\sqrt{-1} ( \varphi + \phi)} 
 \end{equation}
 with $\theta \in [0,\frac{\pi}2]$ and $\varphi \in [0, 2\pi ]$ and $\phi \in [0, \pi ]$. Then, on the chart $\mbox{Spec} ( \C [z] )$ the map $t$ reads as 
 $$
  t\colon (a,b) \mapsto \frac ba = \tan (\theta ) e^{2\sqrt{-1} \phi}. 
 $$
 The parameter of the Hopf circles is $\varphi$. 

\subsection{Perverse Leray filtration}\label{subsec:perverse_filtration}

Consider a general quasi-projective variety $Y$ and denote by $D^b (Y, \Q)$ the derived 
category of bounded complexes of $\Q$-vector spaces $K$ on $Y$ with constructible cohomology sheaves of finite rank. 
Beilinson, Bernstein and Deligne~\cite{BBD} defined truncation functors 
 $$
  ^{\p}\!\tau_{\leq i} :  D^b (Y, \Q) \to  ^{\p}\!\!D^{\leq i} (Y, \Q)
 $$
 encoding the support condition for the middle perversity function, giving rise to a system of truncations 
 $$
  0 \to \cdots \to \, ^{\p}\!\tau_{\leq -p} K \to \, ^{\p} \!\tau_{\leq -p+1} K \to \cdots \to K .
 $$
 This gives rise to the \emph{perverse filtration} 
 $$
    P^p H(Y , K ) = \operatorname{Im} (H(Y, ^{\p}\!\tau_{\leq -p} K) \to 
    H(Y, K )). 
 $$
 
 We will apply the above results to the following setup. 
 Consider the right derived direct image functor 
 $$
    \mathbf{R} H_* \colon D^b (\Mod_{\Dol}, \Q) \to  D^b (\mathcal{B} , \Q)
 $$
 and denote by $R^l H_*$ the $l$'th right derived direct image sheaf. 
Let $\H$ denote hypercohomology of a complex of sheaves and $H$ stand for cohomology of a single sheaf. 
(We hope that the two different usages of the symbol $H$ for the Hitchin map and 
for cohomology groups will not lead to confusion.) 
Let $\underline{\Q}_{\Mod_{\Dol}}$ denote the constant sheaf with fibers $\Q$ on 
$\Mod_{\Dol}$. 
With these notations, we will be interested in the perverse filtration on 
$K = \mathbf{R} H_* \underline{\Q}_{\Mod_{\Dol}}$ over $Y = \mathcal{B}$. 
We then have 
$$
    \H^n (\mathcal{B}, \mathbf{R} H_* \underline{\Q}_{\Mod_{\Dol}} ) \cong 
    H^n (\Mod_{\Dol}, \Q ). 
$$

We will make use of a geometric characterization of the perverse filtration provided 
by M.~de~Cataldo and L.~Migliorini in~\cite[Theorem~4.1.1]{dCM} in terms of the 
flag filtration $F$. 
Namely, let 
\begin{equation}
 Y_{-2}\subset Y_{-1} \subset Y = \mathcal{B}
\end{equation}
be a generic full affine flag in $\mathcal{B}$, namely $Y_{-1}$ a generic line 
and $Y_{-2}$ a generic point within $Y_{-1}$. We then have the equality 
\begin{align*}
  P^p \H^n (Y, \mathbf{R} H_* \underline{\Q}_{\Mod_{\Dol}} ) & = 
  F^{p+n} \H^n (Y, \mathbf{R} H_* \underline{\Q}_{\Mod_{\Dol}} )\\
    & = \Ker (\H^n (Y, \mathbf{R} H_* \underline{\Q}_{\Mod} ) \to \H^n (Y_{p+n-1}, \mathbf{R} H_* \underline{\Q}_{\Mod} |_{Y_{p+n-1}} )), 
\end{align*}
where $F^{\bullet}$ stands to denote the flag filtration. 
It follows immediately from $Y_{-3} = \varnothing$  that $P^{1-n}\H^n = 0$ 
and $P^{-n-2}\H^n = \H^n$, so the only possibly non-trivial graded pieces 
live in degrees $-n-2,-n-1,-n$. Notice that for $p = -1-n$ we get 
$$
    \H^* (Y_{-2}, \mathbf{R} H_* \underline{\Q}_{\Mod} |_{Y_{-2}} ) 
    \cong H^* (H^{-1} (Y_{-2}), \Q ) 
    \cong \Lambda^* H^1 (H^{-1} (Y_{-2}), \Q), 
$$
the exterior algebra over $H^1 (T^4, \Q) \cong \Q^4$. 
Moreover, by the isomorphism theorem 
$$
    \Gr_P^{-n-2} H^n (\Mod_{\Dol}, \Q ) \cong 
    \operatorname{Im} ( H^n (\Mod_{\Dol}, \Q ) \to H^n (H^{-1} (Y_{-2}), \Q) ). 
$$

\section{Large scale behaviour of solutions of Hitchin's equations}

\subsection{Asymptotic abelianization, limiting configuration}

We fix a generic element $q \in S^3_1 $ and consider 
$(\E , \theta) \in \Mod_{\Dol} (\vec{0}, \vec{\alpha})$ such that 
\begin{equation}\label{eq:base_point}
   H (\E , \theta) =  q. 
\end{equation}
As we have explained in Section~\ref{subsec:spectral_curve}, choosing 
such a Higgs bundle $(\E , \theta)$ amounts to fixing a point in an abelian variety 
of complex dimension $2$. 
Then, for any $t \in \C^{\times}$ we have $(\E , t \theta) \in \Mod_{\Dol} (\vec{0}, \vec{\alpha})$, 
i.e. $\C^{\times}$ acts on $\Mod_{\Dol} (\vec{0}, \vec{\alpha})$. 
Obviously, 
$$
  H (\E , t \theta) = t^2 q \in S^3_{|t|^2}. 
$$
For any fixed value of $t$, there exists a unique solution $h_{t}$ of the real Hitchin's equation~\eqref{eq:real_Hitchin} associated to the pair $(\E , t \theta)$. 
We will summarize some results of~\cite{FMSW} (partly based on~\cite{MSWW} and~\cite{Moc})  regarding the asymptotic behaviour of the tame harmonic bundle associated to 
$(\E , \sqrt{R} \theta )$ (the parameter $t>0$ of~\cite{FMSW} thus being replaced by 
$\sqrt{R}$ with $R > 0$). 
The analysis in~\cite{FMSW} relies on the assumption that $\theta$ is generically regular semisimple. This holds for generic $q \in S^3_1 $. 
Indeed, if $\theta$ is not generically regular semisimple then the 
curve~\eqref{eq:nilpotent_spectral_curve} is a section $s$ of $p_L$ with multiplicity 
$2$, which is clearly not the case generically. 
 
Let $\L_{\E } \in \operatorname{Jac}(X_{q})$ be the line bundle such that 
\begin{equation}\label{eq:E_direct_image_L}
    \E = p_{q*} \L_{\E } 
\end{equation} 
(see~\eqref{eq:projection_X}) and for any $R>0$ let us denote by 
$$
  \rho \colon X_{R q} \to X_{R q}
$$
the involution exchanging $Z_+(R q, z)$ and $Z_-(R q, z)$ (see~\eqref{eq:Z_tilde}). 
As $\rho$ is the restriction to $X_{R q}$ of an algebraic involution defined 
over all $\mbox{Tot} ( L )$, we will omit $R q$ from its notation. 
Then, there exists a short exact sequence of sheaves on $X_{R q}$ 
\begin{equation}\label{SES:direct_sum_decomposition}
  0 \to p_{R q}^* \E \to \L_{\E } \oplus \rho^* \L_{\E } \to 
  \O_{\Delta_q} \to 0 . 
\end{equation}
Notice that for any $R>0$ there is an isomorphism 
$$
    X_{R q} \cong X_q
$$
commuting with $p_L$; we deduce that the restriction of the Hitchin map $H$ 
to the $\R^+$-orbit of $q$ is canonically isomorphic to a product 
$$
    \R^+ \times H^{-1} (q). 
$$
Therefore, in the sequel we will often identify $H^{-1} (R q)$ and $H^{-1} (q)$.
 
Let $\tilde{t}(q)\in X_{R q}$ be the preimage of $t(q)$ under $p_{R q}$. 
\begin{prop}
 Formulas~\eqref{eq:Z_tilde2} define univalued holomorphic differentials on $X_{R q}$, 
 vanishing to order $2$ at $\tilde{t}(q)$. 
\end{prop}

\begin{proof}
The fact that the functions are univalued is clear, as $X_{R q}$ is by definition their Riemann surface. 

For simplicity, let us work on the chart $z$ of $\CP1$ and set $[z:1] = z$, a similar analysis works over the chart $w$. 
Furthermore, in this proof we use the notation 
$$
  \zeta_i  = \zeta_i (R q, z)
$$
with $i \in \{ \pm \}$. 
A holomorphic chart of $X_{R q}$ near $t_j$ is given by $\zeta_i $, with local equation 
$$
  \zeta_i ^2 = R (z-t_j ) h_j(z)
$$
for some holomorphic function $h_j$ (depending on $q$) such that $h_j(t_j) \neq 0$. This shows that 
$$
  2 \zeta_i \mbox{d} \zeta_i  = R \mbox{d} z h_j(z) + R (z-t_j ) \mbox{d} h_j .
$$
We derive that the $1$-form $\omega$ defined by 
\begin{align*}
   \omega = \frac{\mbox{d} z}{\zeta_i } & = \frac 1{h_j} \left( \frac{2\mbox{d} \zeta_i }R - \frac{z - t_j}{\zeta_i } \mbox{d} h_j \right) \\
   & = \frac 1{h_j} \left( \frac{\mbox{d} \zeta_i }R - \zeta_i \frac{\mbox{d} h_j }{h_j} \right)
\end{align*}
is holomorphic in $\zeta_i$. The formula shows that $\omega$ is holomorphic near $z = t (q)$ too. 
At any point away from the ramification divisor $\Delta_q$ the form $\omega$ is obviously regular. 

Now, by~\eqref{eq:Z_tilde2} we have 
\begin{align*}
  Z_i  & = \pm \sqrt{R} \sqrt{a z - b} \frac{\mbox{d} z}{\prod_{j=0}^4 \sqrt{z-t_j }} \\
  & =  \sqrt{R} (a z - b) \frac{\mbox{d} z}{\zeta_i } =  \pm \sqrt{R} (a z - b) \omega, 
\end{align*}
where the root of the polynomial $a z - b$ is $t (q)$. The first assertion immediately follows. 
For the second assertion, it is sufficient to notice that near $z = t (q)$ we have 
$$
  a z - b = \zeta_i^2 h (\zeta_i )
$$
for some non-vanishing holomorphic function $h$. This finishes the proof. 
\end{proof}

Fix some  $q \in S^3_1$ and consider a Higgs bundle $(\E, \theta )$
satisfying~\eqref{eq:base_point}, and recall the notation~\eqref{eq:Z_tilde2}. 
Let $\L_{\E }$ be the line bundle satisfying~\eqref{eq:E_direct_image_L}. 
By abelian Hodge theory, there exists (up to multiplication by a constant) a unique 
Hermitian metric $h_{\det (\E )}$ on $\det (\E )$ over $\CP1$ satisfying: 
\begin{itemize}
 \item the associated unitary connection $\nabla_{h_{\det (\E )}}$ on $\det (\E )$ 
 is flat (i.e.,$h_{\det (\E )}$ is Hermitian--Einstein), 
 \item for some local holomorphic trivialization $\vec{e}_1 \wedge \vec{e}_2$ 
 of $\det (\E )$ at $t_j$ we have 
 $$
    \lim_{z\to t_j} |z-t_j|^{-1} |\vec{e}_1 \wedge \vec{e}_2 |_{h_{\det (\E )}} = 1
 $$
 for every $0 \leq j \leq 4$. 
\end{itemize}
Notice that the last condition is imposed by the choice of parabolic 
weights~\eqref{eq:Dolbeault_weights}. 

Moreover, there exists (up to a scalar) a unique abelian Hermitian metric $h_{\L_{\E }}$ on 
$\L_{\E }$ over $X_q$ with parabolic points at $\tilde{\Delta}_q$ such that 
\begin{itemize}
 \item the associated unitary connection $\nabla_{h_{\L_{\E }}}$ on $\L_{\E }$ is flat (i.e.,
 $h_{\L_{\E }}$ is Hermitian--Einstein), 
 \item we have $h_{\L_{\E }} \otimes \rho^* h_{\L_{\E }} = p_q^* h_{\det \E}$ 
 over $\CP1 \setminus \Delta_q$ (see~\eqref{SES:direct_sum_decomposition}), 
 \item for some trivialization $\vec{l}$ of $\L_{\E }$ at each point and some local chart 
 $\zeta$ of $X_{R q}$ centered at $\tilde{t}_j \in \tilde{D}$ we have 
 $$
    \lim_{\zeta \to 0} |\zeta |^{-1} | \vec{l} |_{h_{\L_{\E }}}^2 = 1, 
 $$
 \item for some trivialization $\vec{l}$ of $\L_{\E }$ at each point and some local chart 
 $\zeta$ of $X_{R q}$ centered at $\tilde{t}(q)$ we have 
 $$
    \lim_{\zeta \to 0} |\zeta | | \vec{l} |_{h_{\L_{\E }}}^2 = 1, 
 $$
\end{itemize}
Let $h_{\E,\infty}$ be the orthogonal push-forward of $h_{\L_{\E }}$ by $p_q$ over 
$\CP1 \setminus \Delta_q$, so that in the direct sum 
decomposition~\eqref{SES:direct_sum_decomposition} the  summands 
$\L_{\E}, \rho^* \L_{\E}$ are orthogonal to each other and the restrictions 
of $h_{\E,\infty}$ to these summands are respectively $h_{\L_{\E }}, \rho^* h_{\L_{\E }}$. 
Let $\nabla_{h_{\E,\infty}}$ be the flat $\U(1)\times \U(1)$-connection 
on $\E$ associated to $h_{\E,\infty}$ over $\CP1 \setminus \Delta_q$. 
Over any simply connected subset of $\CP1 \setminus \Delta_q$, let $p_{q,\ast }$ stand 
for the inverse of $p_q^*$ on either branch of $X_q$. Let 
 $$
    B_{\det (\E )} \in \Omega^1 (\CP1 \setminus D , \C ), \qquad 
    \frac 12 p_q^* B_{\det (\E )} + B_{\L_{\E }}  
    \in \Omega^1 ( X_q \setminus \tilde{\Delta}_q ,\C ) 
 $$ 
stand for the connection forms of the flat abelian $\U (1)$-connections 
$\nabla_{h_{\det (\E)}}, \nabla_{h_{\L_{\E }}}$ with respect to some smooth unitary frames. The action of $\rho$ on the second connection form is given by 
$$
    \frac 12 p_q^* B_{\det (\E )} + B_{\L_{\E }} \mapsto 
    \frac 12 p_q^* B_{\det (\E )} - B_{\L_{\E }}. 
$$
By the above properties, the connection form of $\nabla_{h_{\E,\infty}}$ with respect 
to a smooth unitary trivialization of $V$ compatibe with the 
decomposition~\eqref{SES:direct_sum_decomposition} reads as 
 \begin{equation}\label{eq:limiting_unitary_connection_form}
   \begin{pmatrix} 
    \frac 12 B_{\det (\E )} + p_{q,\ast } B_{\L_{\E }} & 0 \\
    0 & \frac 12 B_{\det (\E )} - p_{q,\ast } B_{\L_{\E }} 
   \end{pmatrix}.
 \end{equation}
 Moreover, if one denotes by 
 $\mu_{\det (\E )}, \frac 12 \mu_{\det (\E )} + \mu_{\L_{\E }}$ the $(0,1)$-forms of the 
 $\bar{\partial}$-operators of the corresponding line bundles with respect 
 to smooth unitary frames, then we have 
 \begin{align}\label{eq:limiting_unitary_connection_coefficients1}
    B_{\det (\E )} & = \mu_{\det (\E )} - \bar{\mu}_{\det (\E )},  \\
    B_{\L_{\E }} & = \mu_{\L_{\E }} - \bar{\mu}_{\L_{\E }}.
    \label{eq:limiting_unitary_connection_coefficients2}
 \end{align}
 We then obviously have $p^{0,1} B_{\det (\E )} = \mu_{\det (\E )}$ and 
 $p^{0,1} B_{\L_{\E }} = \mu_{\L_{\E }}$, see~\eqref{eq:projection}. 
We call $(\E, \theta, h_{\E,\infty})$ the {\it limiting configuration} associated to 
$(\E, \theta )$.  
 We introduce the model integrable connection 
 \begin{equation}\label{eq:semisimple_model_flat_connection}
  \nabla_{\sqrt{R}}^{\model} = \nabla_{h_{\E,\infty}} + 
                      \begin{pmatrix}
                       2 \Re Z_+ (R q, z) & 0 \\
                       0 & 2 \Re Z_- (R q , z)
                      \end{pmatrix} 
 \end{equation}
 with respect to the same trivializations as above. 
 On the other hand, we denote by $h_{\sqrt{R}}$ the solution 
 of~\eqref{eq:real_Hitchin} and $\nabla_{\sqrt{R}}$ the Hermitian--Einstein metric 
 and integrable connection associated to $(\E, \sqrt{R} \theta )$. 
 
\begin{theorem}[T.~Mochizuki]~\cite[Corollary~2.13]{Moc}\label{thm:Moc}
 Over any simply connected compact set $K \subset \C \setminus \Delta_q $ there exists a gauge transformation $g_{\sqrt{R}}$ such that 
 $$
  g_{\sqrt{R}} \cdot \nabla_{\sqrt{R}} - \nabla_{\sqrt{R}}^{\model} \to 0
 $$
 (measured with respect to $h_{\sqrt{R}}$) as $R \to \infty$, uniformly over $K$. 
 More precisely, there exist $c_2,C_2>0$ (depending on $K, q$) such that for 
 any $z\in K$ we have 
 $$
  |g_{\sqrt{R}} \cdot \nabla_{\sqrt{R}} (z) - \nabla_{\sqrt{R}}^{\model} (z)|_h
  <  C_2 e^{-c_2 \sqrt{R}} .
 $$
\end{theorem}

 \subsection{Fiducial solution, approximate solutions}
 
 We will equally need the asymptotic form of the solution of Hitchin's equations near the points of $\Delta_q$, where Theorem~\ref{thm:Moc} does not apply. 
 Such a description is provided by R.~Mazzeo, J.~Swoboda, H.~Weiss, F.~Witt in~\cite{MSWW} over a smooth projective curve $X$ of arbitrary genus. 
 This decription is extended by L.~Fredrickson, R.~Mazzeo, J.~Swoboda, H.~Weiss in~\cite{FMSW} to the case of a smooth projective curve $X$ of arbitrary genus for solutions of 
 Hitchin's equations with a finite number of logarithmic singularities and adapted parabolic structure. 
 In accordance with our notations, we let $\sqrt{R}$ be the rescaling parameter of the Higgs field, equal to the parameter $t$ of~\cite{MSWW} and~\cite{FMSW}. 
 We denote the standard holomorphic co-ordinate of $\C$ by $\tilde{z}$ and work in a fixed disc $B_{r_0}(0 ) = \{ |\tilde{z} | \leq r_0 \}$ for some $r_0>0$. 
 We write $\tilde{z} = \tilde{r} e^{\sqrt{-1} \tilde{\varphi}}$ for polar co-ordinates of a point. 
 
 We first describe the solution in the case $0$ is a logarithmic singularity of a harmonic bundle, with Dolbeault parabolic weights denoted by $\alpha^{\pm}\in [0,1 )$ 
 as given in~\eqref{eq:Dolbeault_weights}. Let 
 $$
  m_{\sqrt{R}}\colon \R_+ \to \R
 $$
 be the unique solution of the Painlev\'e III type equation 
 $$
  \left( \frac{\mbox{d}^2}{\mbox{d} \tilde{r}^2} + \frac 1{\tilde{r}} \frac{\mbox{d}}{\mbox{d} \tilde{r}} \right) m_{\sqrt{R}}  = 8 R \tilde{r}^{-1} \sinh (2 m_{\sqrt{R}} ) 
 $$
 satisfying the boundary behaviours 
 \begin{align}
  m_{\sqrt{R}} (\tilde{r}) & \approx \left( \frac 12 + \alpha_j^+ - \alpha_j^- \right) \log (\tilde{r}) = \log (\tilde{r}), \quad \tilde{r} \to 0 + \notag \\
  m_{\sqrt{R}} (\tilde{r}) & \approx \frac 1{\pi} K_0 (8\sqrt{R \tilde{r}}) \approx \frac 1{2 \pi \sqrt{2} \sqrt[4]{R\tilde{r}}} e^{-8 \sqrt{R\tilde{r}}},  \quad \tilde{r} \to \infty \label{eq:Bessel}
 \end{align}
 where the sign $\approx$ stands for complete asymptotic expansion and $K_0$ is the modified Bessel function (or Bessel function of imaginary argument) of order $0$. 
 Furthermore, let us set 
 \begin{equation}\label{eq:function_F}
   F_{\sqrt{R}}(\tilde{r}) = - \frac 18 + \frac 14 \tilde{r} \partial_{\tilde{r}} m_{\sqrt{R}}. 
 \end{equation}
 Making use of the values fixed in~\eqref{eq:Dolbeault_weights}, 
 the \emph{fiducial solution} of Hitchin's equations on $B_{r_0}(0 )$, introduced
 in~\cite[Proposition~3.9]{FMSW}, is defined in a fixed unitary frame as 
 \begin{align}\label{eq:fiducial_solution_singular_unitary}
  A_{\sqrt{R}}^{\fid} & = \left( \frac{\alpha^+ + \alpha^-}4
			\begin{pmatrix}
                         1 & 0 \\
                         0 & 1
                        \end{pmatrix}
    + F_{\sqrt{R}}(\tilde{r}) \begin{pmatrix}
                         1 & 0 \\
                         0 & -1
                        \end{pmatrix}
			 \right) 2 \sqrt{-1} \mbox{d} \tilde{\varphi} \\
			 & = 
			 \left( \frac 14
			\begin{pmatrix}
                         1 & 0 \\
                         0 & 1
                        \end{pmatrix}
    + F_{\sqrt{R}}(\tilde{r}) \begin{pmatrix}
                         1 & 0 \\
                         0 & -1
                        \end{pmatrix}
			 \right) 2 \sqrt{-1} \mbox{d} \tilde{\varphi}, \\
  \theta_{\sqrt{R}}^{\fid} & = \begin{pmatrix}
                       0 & \tilde{r}^{-1/2} e^{m_{\sqrt{R}} (\tilde{r})} \\
                       \tilde{z}^{-1} \tilde{r}^{1/2} e^{-m_{\sqrt{R}} (\tilde{r})} & 0 
                      \end{pmatrix}
		     \mbox{d} \tilde{z}. 
		     \label{eq:fiducial_solution_singular_Higgs}
 \end{align}
 The unitary frame with respect to which the above formulas hold is called \emph{fiducial frame}, and denoted by 
 \begin{equation}\label{eq:fiducial_frame_singular}
  ( e_1^{\fid}, e_2^{\fid} ). 
 \end{equation}

 There is a similar family of solutions for the ramification point $t(q)$ of the 
 spectral curve $X_q$ of the Higgs field. 
 We fix a holomorphic chart $\tilde{z}$ centered at $t(q)$ with associated polar 
 coordinates denoted by $\tilde{r}, \tilde{\varphi}$. 
 Then, with respect to a unitary frame again denoted by 
 \begin{equation}\label{eq:fiducial_frame_apparent}
  ( e_1^{\fid}, e_2^{\fid} ) 
 \end{equation}
 one can introduce 
 \begin{align}\label{eq:fiducial_solution_apparent_unitary}
  A_{\sqrt{R}}^{\fid} & =
    \left( \frac 18 + \frac 14 \tilde{r} \partial_{\tilde{r}} \ell_{\sqrt{R}} \right) 
			\begin{pmatrix}
                         1 & 0 \\
                         0 & -1
                        \end{pmatrix}
			 2 \sqrt{-1} \mbox{d} \tilde{\varphi} \\
  \theta_{\sqrt{R}}^{\fid} & = \begin{pmatrix}
                       0 & \tilde{r}^{1/2} e^{\ell_{\sqrt{R}} (\tilde{r})} \\
                       \tilde{z} \tilde{r}^{-1/2} e^{-\ell_{\sqrt{R}} (\tilde{r})} & 0 
                      \end{pmatrix}
		     \mbox{d} \tilde{z}, \label{eq:fiducial_solution_apparent_Higgs}
 \end{align}
 where $\ell_{\sqrt{R}}$ is the solution of the equation 
 $$
  \left( \frac{\mbox{d}^2}{\mbox{d} \tilde{r}^2} + \frac 1{\tilde{r}} \frac{\mbox{d}}{\mbox{d} \tilde{r}} \right) \ell_{\sqrt{R}}  = 8 R \tilde{r} \sinh (2 \ell_{\sqrt{R}} ) 
 $$
 satisfying the boundary behaviours 
 \begin{align*}
  \ell_{\sqrt{R}} (\tilde{r}) & \approx - \frac 12  \log (\tilde{r}), \quad \tilde{r} \to 0 + \\
  \ell_{\sqrt{R}} (\tilde{r}) & \approx \frac 1{\pi} K_0 \left( \frac 83 \sqrt{R \tilde{r}^3} \right) \approx \frac{\sqrt{3}}{2 \pi \sqrt{2} \sqrt[4]{R\tilde{r}^3}} e^{-\frac 83 \sqrt{R\tilde{r}^3}},  \quad \tilde{r} \to \infty .
 \end{align*}
  The \emph{limiting fiducial solution} is obtained by letting $R\to\infty$ in the above formulas, specifically 
 \begin{align*}
  A_{\infty}^{\fid} & = \frac 18 \begin{pmatrix}
                         1 & 0 \\
                         0 & -1
                        \end{pmatrix}
			 2 \sqrt{-1} \mbox{d} \tilde{\varphi} \\
  \theta_{\infty}^{\fid} & = \begin{pmatrix}
                       0 & \tilde{r}^{1/2}  \\
                       \tilde{r}^{1/2} e^{\sqrt{-1} \tilde{\varphi}} & 0 
                      \end{pmatrix}
		     \mbox{d} \tilde{z}. 
 \end{align*}
 
 In order to assemble the limiting configuration and the fiducial solutions into a family 
 of approximate solutions $h_{R q}^{\app}$,~\cite{FMSW} perform a gluing construction. 
 We describe this construction. 
 
 We start by describing normal forms of the solution of Hitchin's equations near 
 the points of $\Delta_q$. 
 By \cite[Proposition~3.4]{FMSW} there exists a unique holomorphic co-ordinate 
 $\tilde{z}_t$ defined in a neighbourhood of the ramification point $t = t(q)$ such 
 that 
 \begin{equation}
  q (\tilde{z}_t ) = - \tilde{z}_t (\mbox{d} \tilde{z}_t)^2. 
 \end{equation}
 Furthermore, there exists a holomorphic gauge of $\E$ near $t(q)$ with respect to 
 which one has 
 $$
    \theta = \begin{pmatrix}
                0 & 1 \\
                \tilde{z}_t & 0
             \end{pmatrix}
             \mbox{d} \tilde{z}_t, 
             \quad 
             h_{\E,\infty} = Q_t (\tilde{z}_t ) \begin{pmatrix}
                               |\tilde{z}_t |^{\frac 12} & 0 \\
                               0 & |\tilde{z}_t |^{-\frac 12}
                              \end{pmatrix}
 $$
 where $Q_t$ is a locally-defined smooth function completely determined by $h_{\det \E}$ 
 and $q$. 
 Similarly, for any $0\leq j \leq 4$ \cite[Proposition~3.5]{FMSW} shows that there 
 exists some holomorphic co-ordinate $\tilde{z}_j$ of $\E$ near $t_j$ such that 
 we have 
 \begin{equation}
    q (\tilde{z}_j ) = - \tilde{z}_j^{-1} (\mbox{d} \tilde{z}_j)^2. 
 \end{equation}
 Furthermore, there exists a holomorphic gauge of $\E$ near $t_j$ with respect to 
 which one has 
 $$
    \theta = \begin{pmatrix}
                0 & 1 \\
                \tilde{z}_j^{-1} & 0
             \end{pmatrix}
             \mbox{d} \tilde{z}_j, 
             \quad 
             h_{\E,\infty} = Q_j |\tilde{z}_j|^{\alpha_j^+ + \alpha_j^-} 
                              \begin{pmatrix}
                               |\tilde{z}_j|^{-\frac 12} & 0 \\
                               0 & |\tilde{z}_j|^{\frac 12}
                              \end{pmatrix}
                              = 
                              Q_j 
                              \begin{pmatrix}
                               |\tilde{z}_j|^{\frac 12} & 0 \\
                               0 & |\tilde{z}_j|^{\frac 32}
                              \end{pmatrix}                              
 $$
 where $Q_j$ is a locally-defined function completely determined by $h_{\det \E}$ 
 and $q$.
 Let us fix a cutoff function $\chi\colon [0, \infty ) \to [0,1]$ such that 
 $\chi (\tilde{r}) = 1$ for all $\tilde{r} \leq \frac 12$ and 
 $\chi (\tilde{r}) = 0$ for all $\tilde{r} \geq 1$. 
  Define the smooth Hermitian metric $h_{\sqrt{R}}^{\app}$ to be equal 
 \begin{itemize}
  \item to 
    $$
     Q_t (\tilde{z}_t ) 
    \begin{pmatrix}
    |\tilde{z}_t |^{\frac 12} e^{\ell_{\sqrt{R}} (|\tilde{z}_t|) \chi (|\tilde{z}_t|)} & 0 \\
    0 & |\tilde{z}_t |^{-\frac 12} e^{\ell_{\sqrt{R}} (|\tilde{z}_t|) \chi (|\tilde{z}_t|)}
    \end{pmatrix}
    $$
    on $|\tilde{z}_t| \leq 1$ in a holomorphic co-ordinate and gauge with respect to 
    which the above normal form holds ; 
  \item to 
    $$
     Q_j (\tilde{z}_j )
    \begin{pmatrix}
    |\tilde{z}_j |^{\frac 12} e^{m_{\sqrt{R}} (|\tilde{z}_j|) \chi (|\tilde{z}_j|)} & 0 \\
    0 & |\tilde{z}_j |^{\frac 32} e^{m_{\sqrt{R}} (|\tilde{z}_j|) \chi (|\tilde{z}_j|)}
    \end{pmatrix}
    $$
    on $|\tilde{z}_j| \leq 1$ in a holomorphic co-ordinate and gauge with respect to 
    which the above normal form holds ; 
 \item to $h_{\E,\infty}$ on the complement of the above discs. 
 \end{itemize}
 
 Fix a background Hermitian metric $h_0$ on $V$ and let us denote by 
 $H_{\sqrt{R}}, H_{\sqrt{R}}^{\app}$ the $h_0$-Hermitian sections of 
 $End(V)$ satisfying 
 \begin{equation}\label{eq:h_t}
      h_{\sqrt{R}} (v,w) = h_0 \left( (H_{\sqrt{R}})^{\frac 12} v ,
      (H_{\sqrt{R}})^{\frac 12} w \right)
 \end{equation}
 and similarly 
 $$
    h_{\sqrt{R}}^{\app} (v,w) = h_0 \left( (H_{\sqrt{R}}^{\app})^{\frac 12} v ,
    (H_{\sqrt{R}}^{\app})^{\frac 12} w \right).
 $$
 Then, for a fixed Higgs bundle $(\E, \theta )$  one may look for solutions 
 $(\E, {\sqrt{R}} \theta , h_{\sqrt{R}})$ of Hitchin's equations (i.e., the 
 Hermite--Einstein equation for the pair $(\E, {\sqrt{R}} \theta )$) in the 
 form 
 \begin{equation}\label{eq:FMWW_4.1_utan}
  (H_{\sqrt{R}})^{\frac 12} = e^{\gamma_{\sqrt{R}}} (H_{\sqrt{R}}^{\app})^{\frac 12} 
 \end{equation}
 for some $\sqrt{-1} \su(V, h_{\sqrt{R}}^{\app})$-valued section $\gamma_{\sqrt{R}}$.
 
 \begin{theorem}~\cite[Theorem~6.7]{MSWW},\cite[Theorem~6.2]{FMSW}\label{thm:MSWW6.7}
 Assume that all the zeroes of $q$ are simple. 
 Then, there exists $C, \mu > 0$ and a unique section $\gamma_{\sqrt{R}}$ such that the 
 Hermitian metric~\eqref{eq:h_t} with~\eqref{eq:FMWW_4.1_utan} satisfies 
 the Hermite--Einstein equation, and 
 $$
    \Vert \gamma_{\sqrt{R}} \Vert_{\mathcal{C}^{2,\alpha}_b} 
    \leq C e^{-(\mu /2)\sqrt{R}}  
 $$ 
 for an appropriate H\"older norm $\mathcal{C}^{2,\alpha}_b$. 
 \end{theorem}
 
 The practical implication of this result for our purpose is that one may perturb 
 the approximate solution by a term exponentially small in $\sqrt{R}$ so as to 
 obtain the solution of Hitchin's equations. 
 We will denote by $\nabla_{\sqrt{R}}$ the flat connection associated to the 
 solution $(\E, {\sqrt{R}} \theta , h_{\sqrt{R}})$, i.e. 
 \begin{equation}\label{eq:flat_connection}
  \nabla_{\sqrt{R}} = \bar{\partial}_{\E} + \partial^{h_{\sqrt{R}}} + 
  \sqrt{R} \theta + \sqrt{R} \theta^{\dagger, h_{\sqrt{R}}}
 \end{equation}
 where $\dagger, h_{\sqrt{R}}$ stands for adjoint with respect to $h_{\sqrt{R}}$. 
 Then, $\nabla_{\sqrt{R}}$ is approximated up to exponentially decreasing error 
 terms in $R$ by 
 $$
    \nabla_{\sqrt{R}}^{\app} = \bar{\partial}_{\E} + \partial^{h_{\sqrt{R}}^{\app}} + 
  \sqrt{R} \theta + \sqrt{R} \theta^{\dagger, h_{\sqrt{R}}^{\app}}. 
 $$

\section{Simpson's Fenchel--Nielsen co-ordinates}\label{sec:Simpson}

Simpson has defined in~\cite[Section~10]{Sim} co-ordinates of $\Mod_{\Betti}(\vec{c}, \vec{\gamma})$. 
In this section, we will recall the definition of these co-ordinates. 
The general element of the Betti moduli space is a local system $V$ on 
$\CP1 \setminus D$, given by a representation $\chi$ of its fundamental group, 
with eigenvalues around the punctures $t_j$ equal to $c_j^{\pm} = \pm \sqrt{-1}$. 
For each $2\leq i \leq 3$ there are two different co-ordinates: $l_i \in \C$ and $[p_i : q_i] \in \CP1$. 
By analogy with classical Teichm\"uller theory, we will call co-ordinates of the first type $l_i$ the \emph{complex length co-ordinates} and those of the second 
type $[p_i : q_i]$ the \emph{complex twist co-ordinates}. 
Indeed, the traditional length co-ordinates in Teichm\"uller space belong to $\R$ and $\C$ is its complexification; similarly, the twist co-ordinates take values in 
$S^1$, which is the real part $\R P^1$ of $\CP1$ for the canonical real structure.

\begin{remark}
 The construction of the co-ordinates depend on some choices, most importanty the 
 radii $r_0$ of discs around the punctures. However, as our main concern is 
 the homotopy type of the diffeomorphism $\RH \circ \psi$ and different 
 choices obviously result in homotopic maps, the actual choices do not fundamentally 
 affect our arguments. As a matter of fact, throughout we will let $r_0 \to 0$. 
\end{remark}

\subsection{Complex length co-ordinates}\label{sec:length}

We fix disjoint open discs $D_j$ around the points $t_j$ for $0\leq j \leq 4$; to fix our ideas we pick $D_j = B_{r_0}(t_j )^o = \{ |z - t_j| < r_0 \}$ for some $0 < r_0 \ll 1$ 
so that the different discs $D_j$ are disjoint. We then set 
\begin{equation}\label{eq:surface_S}
 S = \CP1 \setminus (D_0 \cup \cdots \cup D_4). 
\end{equation}
Then $S$ is a smooth surface with boundary, inheriting an orientation from $\CP1$. 
Let us denote by $\xi_j$ the boundary component $\partial D_j$, taken with the orientation induced from $S$. 
Specifically, we let 
\begin{equation}\label{eq:xi_j}
  \xi_j (\varphi ) = t_j + r_0 e^{\sqrt{-1} \varphi} \quad \mbox{for} \; \varphi \in [0,2\pi ]. 
\end{equation}
Fix a simple loop $\rho_2$ in $S$ separating the boundary components $\xi_1, \xi_2$ from the remaining boundary components $\xi_3, \xi_4, \xi_0$, 
and a simple loop $\rho_3$ in $S$ separating the boundary components $\xi_4, \xi_0$ from the remaining boundary components $\xi_1, \xi_2, \xi_3$, 
so that $\rho_2$ and $\rho_3$ be disjoint from each other. 
These curves then decompose $S$ into the union 
\begin{equation}\label{eq:pants}
 S = S_2 \cup S_3 \cup S_3 
\end{equation}
of three pairs of pants: 
\begin{itemize}
 \item $S_2$ with boundary components $\xi_1, \xi_2, \rho_2$; 
 \item $S_3$ with boundary components $\xi_3, \rho_2, \rho_3$;
 \item and $S_3$ with boundary components $\xi_4, \xi_0, \rho_3$. 
\end{itemize}
This decomposition gives rise to a decomposition of $\CP1$ into the three closed connected analytic subsets 
\begin{align}
 X_2 & = S_2 \cup D_1 \cup D_2 \label{eq:X2} \\
 X_3 & = S_3 \cup D_3  \label{eq:X3} \\
 X_4 & = S_3 \cup D_4 \cup D_0.  \label{eq:X4}
\end{align}
Furthermore, we fix 
\begin{itemize}
 \item base points $x_i \in \mbox{int}(S_i)$ and  $s_i \in \rho_i$;
 \item paths $\psi_i$ connecting $x_i$ to $x_{i+1}$ passing through $s_i$; 
 \item paths $\eta_1, \eta_2$ connecting $x_2$ respectively to the base points of $\xi_1, \xi_2$;
 \item a path $\eta_3$ connecting $x_3$ to the base point of $\xi_3$; 
 \item paths $\eta_4, \eta_0$ connecting $x_4$ respectively to the base points of $\xi_4, \xi_0$. 
\end{itemize}
As the  $D_j$ will actually depend on its radius $r_0$, we need to make a coherent choice for the paths $\eta_0, \ldots , \eta_4$. 
We achieve this for instance for $\eta_1$ by first fixing a path starting at $x_2$ and ending at $t_1$, and then restricting this fixed path to the (uniquely determined) sub-interval of its domain 
such that the restriction connects $x_2$ to the base point of $\xi_1$. We apply a similar procedure to $\eta_j$ for all $0 \leq j \leq 4$. 
We set $\rho_1 = \xi_1$ and $\rho_4 = \xi_0$. 

\begin{figure}
\centering
\begin{tikzpicture}[baseline = 5cm]
 \draw (0,0) circle [x radius = 6cm, y radius = 4cm];
 \filldraw [fill=gray, draw=black] (-4.2,1) circle [x radius = 0.5cm, y radius = 0.75cm];
 \filldraw [fill=gray, draw=black] (-4.2,-1) circle [x radius = 0.5cm, y radius = 0.75cm];
 \filldraw [fill=gray, draw=black] (0,-2.5) circle [x radius = 0.75cm, y radius = 0.5cm];
 \filldraw [fill=gray, draw=black] (4.2,1) circle [x radius = 0.5cm, y radius = 0.75cm];
 \filldraw [fill=gray, draw=black] (4.2,-1) circle [x radius = 0.5cm, y radius = 0.75cm];
 \draw (-2,-3.75) arc [x radius = 1cm, y radius = 3.75cm, start angle = -90, end angle = 90];
 \draw [dashed] (-2,3.75) arc (90:270: 1cm and 3.75cm);
 \draw (2,3.75) arc (90:-90: 1cm and 3.75cm);
 \draw [dashed] (2,3.75) arc (90:270: 1cm and 3.75cm);
 \filldraw [black] (-2.5,2.5) circle (2pt) node [anchor = west] {$x_2$};
 \filldraw [black] (0,2.5) circle (2pt) node [anchor = south] {$x_3$};
 \filldraw [black] (3.5,2.5) circle (2pt) node [anchor = north] {$x_4$};
 \filldraw [black] (-1,0) circle (2pt) node [anchor = south] {$s_2$};
 \filldraw [black] (3,0) circle (2pt) node [anchor = south] {$s_3$};
 \filldraw [black] (-3.7,1) circle (2pt);
 \filldraw [black] (-3.7,-1) circle (2pt);
 \filldraw [black] (0,-2) circle (2pt);
 \filldraw [black] (4.7,1) circle (2pt);
 \filldraw [black] (4.7,-1) circle (2pt);
 \draw (-2.5,2.5) .. controls (-2.5,1.5) and (-2,0) .. (-1,0);
 \draw (-1,0) .. controls (0,0) and (-1,2.5) .. (0,2.5);
 \draw (0,2.5) .. controls (1,2.5) and (2,0) .. (3,0);
 \draw (3,0) .. controls (4,0) and (2.5,2.5) .. (3.5,2.5);
 \draw (-2.5,2.5) to [out = 180, in = 0] (-3.7,1) (-2.5,2.5) to [out = 225, in = 45] (-3.7,-1); 
 \draw (0,2.5) to [out = 295, in = 75] (0,-2); 
 \draw (3.5,2.5) .. controls (5, 2.5) and (5,1) .. (4.7,1) (3.5,2.5) .. controls (5.5, 2.5) and (5.7,0) .. (4.7,-1); 
 \draw (-4.2,1.9) node {$\rho_1 = \xi_1$};
 \draw (-4.2,1) node {$D_1$};
 \draw (-5,-1) node {$\xi_2$};
 \draw (-4.2,-1) node {$D_2$};
 \draw (0.5,-1.9) node {$\xi_3$};
 \draw (0,-2.5) node {$D_3$};
 \draw (4.2,1.9) node {$\xi_0 = \rho_4$};
 \draw (4.2,1) node {$D_0$};
 \draw (4.2,-2) node {$\xi_4$};
 \draw (4.2,-1) node {$D_4$};
 \draw (3,3) node {$\rho_3$};
 \draw (-1,3) node {$\rho_2$};
 \draw (-2.5,1) node {$\psi_2$};
 \draw (2.2,1) node {$\psi_3$};
 \draw (-3.4,1.5) node {$\eta_1$};
 \draw (-3.4,0) node {$\eta_2$};
 \draw (0.2,0) node {$\eta_3$};
 \draw (5,1) node {$\eta_0$};
 \draw (5,0) node {$\eta_4$};
 \draw [decoration=brace,thick,decorate] (-2.1,-4.5) -- node[below=.5cm] {$S_2$} (-6,-4.5);
 \draw [decoration=brace,thick,decorate] (1.9,-4.5) -- node[below=.5cm] {$S_3$} (-1.9,-4.5);
 \draw [decoration=brace,thick,decorate] (6,-4.5) -- node[below=.5cm] {$S_3$} (2.1,-4.5);
\end{tikzpicture}
\caption{Decomposition of $S$ into three pairs of pants, indicating base points and paths. The shaded regions do not belong to $S$.}
\label{fig}
\end{figure}
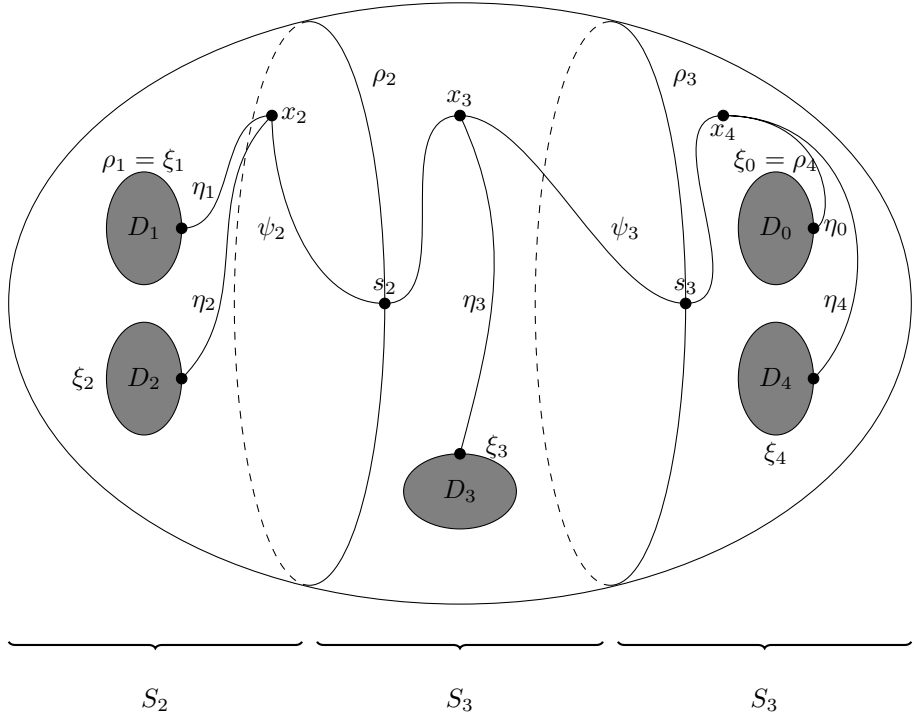

Following~\cite{Sim}, for $2\leq i \leq 3$ we set $l_i (V) = l_i(\chi )$ for the 
trace of $\chi$ evaluated on the class of the loop $\rho_i$: 
$$
    l_i (V) = \tr \chi  [ \rho_i ]
$$
By definition, $l_i$ is the $i$'th \emph{complex length co-ordinate}. 

\subsection{Complex twist co-ordinates}\label{sec:twist}

Twist co-ordinates are only defined over the part $\Mod_{\Betti}(\vec{c}, \vec{\gamma})'$ of the moduli space where we have $| l_i | \neq 2$ (equivalently, 
the eigenvalues of $\chi [ \rho_i ]$ are distinct) 
for both $2\leq i \leq 3$, for the complex length co-ordinates $l_i$ introduced in Section~\ref{sec:length}, and a further stability condition holds (see \cite[Definition~5.1]{Sim}). 
It is proven in \cite[Corollary~9.2]{Sim} that the homotopy type of the dual boundary complex of $\Mod_{\Betti}(\vec{c}, \vec{\gamma})$ agrees with the one of $\Mod_{\Betti}(\vec{c}, \vec{\gamma})'$. 

Let us introduce the scalar quantities 
\begin{align*}
 l_1 & = c_1^+ + c_1^- \\
 l_4 & = c_4^+ + c_4^- \\
 u_i & = \frac{l_{i-1} - c_i^- l_i}{c_i^+ - c_i^-} \\
 w_i & = u_i (l_i - u_i) -1 
\end{align*}
for $2\leq i \leq 4$, where $l_i$ are the complex length co-ordinates associated to $V$ as in Section~\ref{sec:length}. Furthermore, introduce the matrices
\begin{align*}
  A_i & = \begin{pmatrix}
          c_i^+ & 0 \\
          0 & c_i^-
         \end{pmatrix}
         \\
  R_i & = \begin{pmatrix}
          u_i & 1 \\
          w_i & ( l_i - u_i)
         \end{pmatrix}
         \\
  R'_{i-1} & = A_i R_i = \begin{pmatrix}
          c_i^+ u_i & c_i^+ \\
          c_i^- w_i & c_i^- ( l_i - u_i)
         \end{pmatrix}
         \\
  T_i & = \begin{pmatrix}
          0 & 1 \\
          -1 & l_i
         \end{pmatrix}
         \\
  U_i & = \begin{pmatrix}
          1 & 0 \\
          u_i & 1
         \end{pmatrix}. 
\end{align*}
These quantities are all determined by the fixed constants $c_i^{\pm}$ and the length co-ordinates $l_2, l_3$. 

Let $V_i (l_{i-1}, l_i )$ denote the local system on $S_i$ whose monodromy matrices around $\rho_{i-1},\rho_i$ and $\xi_i$, acting on its fiber over $x_i$, are respectively $R'_{i-1}, R_i, A_i$.  
\cite[Corollary~10.3]{Sim} implies that if $V|_{S_i}$ is stable then there exists a unique (up to a scalar) isomorphism 
\begin{equation}\label{eq:hi}
 h_i \colon V|_{S_i} \to V_i (l_{i-1}, l_i ). 
\end{equation}
By an abuse of notation, we denote by 
\begin{equation}\label{eq:psii}
  \psi_i \colon V_{x_i} \to V_{x_{i+1}} 
\end{equation}
the parallel transport map of $V$ along the path $\psi_i$. Introduce 
\begin{equation}
 P_i = h_{i+1} \circ \psi_i \circ h_i^{-1} \colon V_i (l_{i-1}, l_i )_{x_i} \to V_{i+1} (l_i, l_{i+1} )_{x_{i+1}} 
\end{equation}
and 
\begin{equation}
 Q_{i-1} = A_i^{-\frac 12} U_i P_{i-1} U_{i-1}^{-1},
\end{equation}
for any choice of the square root of $A_i$. 
It turns out that one has 
\begin{equation}\label{eq:twist}
   Q_i = \begin{pmatrix}
          p_i & q_i \\
          -q_i & p_i + l_i q_i
         \end{pmatrix}
\end{equation}
for some $[p_i : q_i]\in\CP1$ satisfying
$$
  p_i^2 + l_i p_i q_i + q_i^2 \neq 0. 
$$
By definition, $[p_i : q_i]\in\CP1$ for $i \in \{ 2, 3 \}$ is the $i$'th 
\emph{complex twist co-ordinate}. Notice that a scalar factor on $Q_i$ has no 
impact on $[p_i : q_i]$. Let us introduce 
$$
  \mathbf{Q} = \{ (l, [p:q]) \in (\C \setminus \{ \pm 2 \}) \times \CP1 \; \mbox{satisfying} \; p^2 + l p q + q^2 \neq 0  \}. 
$$
According to \cite[Theorem~10.6]{Sim}, the map 
\begin{align*}
 \Mod_{\Betti}(\vec{c}, \vec{\gamma})' & \to \mathbf{Q}^2 \\
 V & \mapsto ( (l_2, [p_2 : q_2]), (l_3, [p_3 : q_3]) )
\end{align*}
is a diffeomorphism.

\subsection{Homotopy type of compactifying divisor}\label{ssec:homotopy}

According to~\cite[Lemma~10.7]{Sim} and~\cite[Lemma~6.2]{Pay} we have 
homotopy equivalences 
\begin{align}
  \mathbb{D} \partial \mathbf{Q} & \sim S^1 \label{eq:boundary_Q} \\ 
  \mathbb{D} \partial \mathbf{Q}^2 & \sim S^1 \ast S^1 \sim S^3, \label{eq:boundary_Q_squared}
\end{align}
where $X \ast Y$ stands for the join of the topological spaces $X,Y$. 
Combining these arguments, \cite[Corollary~10.8]{Sim} shows that 
$$
  \mathbb{D} \partial \Mod_{\Betti}(\vec{c}, \vec{\gamma}) \sim S^3. 
$$

Let us spell out explicitly the homotopy equivalence~\eqref{eq:boundary_Q}. 
We now consider two copies of 
$$
    \mathbf{Q} \subset \CP1 \times \CP1, 
$$
that we will denote by $\mathbf{Q}_i$ for $i \in \{ 2, 3 \}$. 
A compactification of $\mathbf{Q}_i$ is $\CP1 \times \CP1$, an open affine of
the first component being parametrised by $l_i$, and the second component being
parametrised by $[p_i : q_i]$. 
Let us denote by $F_{i,+}, F_{i,-}, F_{i,\infty}$ the fibers of the first projection over $2,-2$ and $\infty$ respectively. 
The irreducible decomposition of the compactifying divisor of $\mathbf{Q}_i$ 
in $\CP1 \times \CP1$ reads as 
$$
  \partial \mathbf{Q}_i = \CP1 \times \CP1 \setminus \mathbf{Q}_i = 
  C_i \cup F_{i,+} \cup F_{i,-} \cup F_{i,\infty}
$$
where $C_i$ is the quadric defined by $p_i^2 + l_i p_i q_i + q_i^2 = 0$, see Figure~\ref{fig:compactifying_divisor}. 
\begin{figure}
 \centering
 \begin{tikzpicture}[baseline = 5cm]
    \draw[dashed] (0,0) -- (0,5);
    \draw (0,5) -- (5,5);
    \draw (0,0) -- (5,0);
    \draw[dashed] (5,0) -- (5,5);
    \draw[dashed] (2.5,2.5) circle (2.5); 
    \draw[dashed] (2.5,0) -- (2.5,5);
    \draw (-0.5,1) node {$F_{i,-}$};
    \draw (5.5,1) node {$F_{i,+}$};
    \draw (2.5,1) node [anchor = west] {$F_{i,\infty}$};
    \draw (4,4) node {$C_i$};
    \draw (-.5,5) node {$[0:1]$}; 
    \draw (-.5,0) node {$[1:0]$};
 \end{tikzpicture}
 \caption{Compactifying divisor in $\CP1 \times \CP1$.}
\label{fig:compactifying_divisor}
\end{figure}
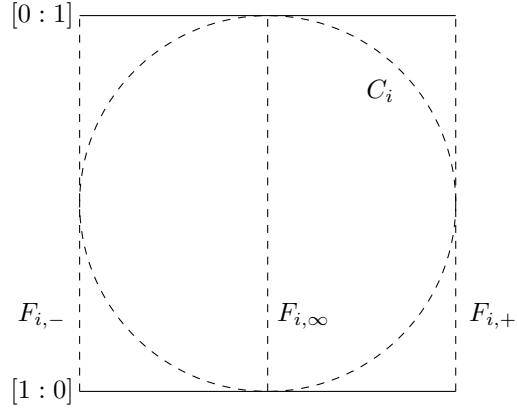
Clearly, $C_i$ is generically $2:1$ over $\CP1_t$, with ramification points in the fibers $F_{i,+}, F_{i,-}$. 
Therefore, the compactifying divisor in $\CP1 \times \CP1$ is not normal crossing. 
To remedy this failure, we consider the blow up of $\CP1 \times \CP1$ in the intersection points $(2, [1:-1])$ and $(-2, [1:1])$, see Figure~\ref{fig:compactifying_divisor1}. 
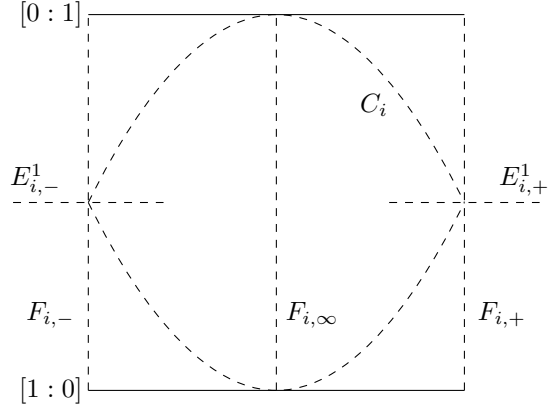
\begin{figure}
 \centering
 \begin{tikzpicture}[baseline = 5cm]
    \draw[dashed] (0,0) -- (0,5);
    \draw (0,5) -- (5,5);
    \draw (0,0) -- (5,0);
    \draw[dashed] (5,0) -- (5,5);
    \draw[dashed] (0,2.5) parabola bend (2.5,5) (5,2.5); 
    \draw[dashed] (0,2.5) parabola bend (2.5,0) (5,2.5); 
    \draw[dashed] (2.5,0) -- (2.5,5);
    \draw (-0.5,1) node {$F_{i,-}$};
    \draw (5.5,1) node {$F_{i,+}$};
    \draw (2.5,1) node [anchor = west] {$F_{i,\infty}$};
    \draw (-0.7,2.8) node {$E_{i,-}^1$};
    \draw (5.8,2.8) node {$E_{i,+}^1$};
    \draw (3.8,3.8) node {$C_i$};
    \draw (-.5,5) node {$[0:1]$}; 
    \draw (-.5,0) node {$[1:0]$};
    \draw[dashed] (-1,2.5) -- (1,2.5);
    \draw[dashed] (4,2.5) -- (6,2.5);
 \end{tikzpicture}
 \caption{Compactifying divisor in first blow-up.}
\label{fig:compactifying_divisor1}
\end{figure}
We continue to denote by $C_i,F_{i,+}, F_{i,-}, F_{i,\infty}$ the proper transforms in $X$ of the named divisors, and we denote by $E_{i,+}^1, E_{i,-}^1$ the exceptional divisors. 
The compactifying divisor in the blow-up is 
$$
  C_i \cup F_{i,+} \cup F_{i,-} \cup F_{i,\infty} \cup E_{i,+}^1 \cup E_{i,-}^1. 
$$
However, this is still not simple normal crossing, because of the triple intersection points of $C_i, E_{i,+}^1$ and  $F_{i,+}$ on the one 
hand, and $C_i, E_{i,-}^1$ and  $F_{i,-}$ on the other hand. Therefore, we need to blow up again in these intersection points, see Figure~\ref{fig:compactifying_divisor2}. 
\begin{figure}
 \centering
 \begin{tikzpicture}[baseline = 5cm]
    \draw[dashed] (0,0) -- (0,5);
    \draw (0,5) -- (5,5);
    \draw (0,0) -- (5,0);
    \draw[dashed] (5,0) -- (5,5);
    \draw[dashed] (2.5,2.5) ellipse (2 and 2.5); 
    \draw[dashed] (2.5,0) -- (2.5,5);
    \draw (-0.5,1) node {$F_{i,-}$};
    \draw (5.5,1) node {$F_{i,+}$};
    \draw (2.5,1) node [anchor = west] {$F_{i,\infty}$};
    \draw (-0.6,3) node {$E_{i,-}^1$};
    \draw (6.4,3) node {$E_{i,+}^1$};
    \draw (1.2,2.5) node {$E_{i,-}^2$};
    \draw (3.8,2.5) node {$E_{i,+}^2$};
    \draw (3.8,3.8) node {$C_i$};
    \draw (-.5,5) node {$[0:1]$}; 
    \draw (-.5,0) node {$[1:0]$};
    \draw[dashed] (-1.5,2.5) -- (1,2.5);
    \draw[dashed] (4,2.5) -- (6.5,2.5);
    \draw[dashed] (-1,1.5) -- (-1,3.5);
    \draw[dashed] (6,1.5) -- (6,3.5);
 \end{tikzpicture}
 \caption{Compactifying divisor in second blow-up.}
\label{fig:compactifying_divisor2}
\end{figure}
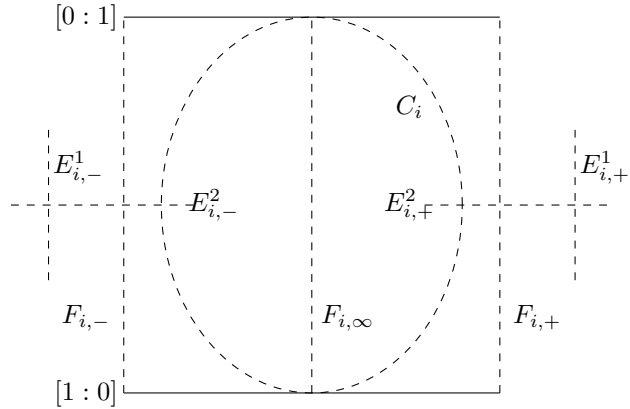
The compactifying divisor in this surface is of normal crossing. 
Dropping the subscripts $i$ for simplicity, its dual complex is 

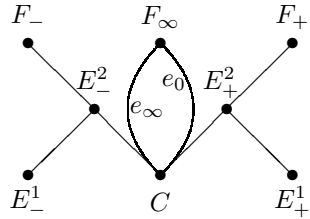
\begin{picture}(200,100)(-100,0)

 \put(50,50){\circle*{4}} 	
 \put(100,50){\circle*{4}}	
 \put(75,75){\circle*{4}}	
 \put(75,25){\circle*{4}}	
 \put(25,25){\circle*{4}}	
 \put(25,75){\circle*{4}}	
 \put(125,25){\circle*{4}}	
 \put(125,75){\circle*{4}}	
 
 \put(75,25){\line(1,1){50}}
 \put(75,25){\line(-1,1){50}}
 \put(50,50){\line(-1,-1){25}}
 \put(100,50){\line(1,-1){25}}
 
 \qbezier(75,25) (50,50) (75,75)
 \qbezier(75,25) (100,50) (75,75)
 
 \put(25,85){\makebox(0,0){$F_{-}$}}
 \put(50,60){\makebox(0,0){$E_{-}^2$}}
 \put(98,60){\makebox(0,0){$E_{+}^2$}}
 \put(125,85){\makebox(0,0){$F_{+}$}}
 \put(75,85){\makebox(0,0){$F_{\infty}$}}
 \put(75,15){\makebox(0,0){$C$}}
 \put(25,15){\makebox(0,0){$E_{-}^1$}}
 \put(125,15){\makebox(0,0){$E_{+}^1$}}
 
 \put(70,50){\makebox(0,0){$e_\infty$}}
 \put(80,60){\makebox(0,0){$e_0$}}
\end{picture}

Obviously, this graph deformation retracts to the cycle defined by the vertices 
$F_{i,\infty}, C_i$ together with the edges $e_0, e_{\infty}$ connecting them. 
Notice that in $\CP1 \times \CP1$, this cycle reduces to the normal crossing components $F_{i,\infty}, C_i$. 

Next, let us be more precise about the homotopy
equivalence~\eqref{eq:boundary_Q_squared} following~\cite[Lemma~6.2]{Pay}. 
The compactifying divisor of $\mathbf{Q}^2 = \mathbf{Q}_2 \times \mathbf{Q}_3$ 
in 
$$
    (\CP1)^4 = (\CP1)^2 \times (\CP1)^2
$$ 
can be given as 
$$
   ( \partial \mathbf{Q}_2 \times (\CP1)^2 ) \cup 
   ( (\CP1)^2 \times \partial \mathbf{Q}_3 ). 
$$
As we have explained above, up to homotopy of $\mathbb{D}\partial \mathbf{Q}_i$ 
we only need to consider the divisor components $C_i$ and $F_{i,\infty}$ of 
$\partial \mathbf{Q}_i$ and the edges connecting them; 
in the rest of this section, we will thus replace $\mathbb{D}\partial \mathbf{Q}_i$ 
by this subcomplex without changing the notation. 
Our notation in the dual graph of $\partial \mathbf{Q}_i$ is that $e_{i, \infty}$ stands 
for the edge corresponding to the point 
$$
    (l_i, [p_i : q_i]) = (\infty, [1:0]) \in (\CP1)^2,
$$
and $e_{i,0}$ for the edge corresponding to 
$$
    (l_i, [p_i : q_i]) = (\infty, [0:1]) \in (\CP1)^2. 
$$
Now, to each of the four points 
\begin{align}
     (\infty, [1:0]), (\infty, [1:0]) \notag \\
     (\infty, [1:0]), (\infty, [0:1]) \notag \\
     (\infty, [0:1]), (\infty, [1:0]) \notag \\
     (\infty, [0:1]), (\infty, [0:1]) \label{eq:point_of_interest}   
\end{align}
of $(\CP1)^4$, there corresponds in $\mathbb{D}\partial \mathbf{Q}^2$ a 
$3$-dimensional simplex, namely the join of the edges in 
$\mathbb{D}\partial \mathbf{Q}_i$ corresponding to each component. 
Thus, the natural $\Delta$-complex structure of 
$\mathbb{D}\partial \mathbf{Q}^2 \sim S^3$ contains these four $3$-simplices, 
which in order are 
\begin{align*}
 e_{2, \infty} & \ast e_{3, \infty} \\
 e_{2, \infty} & \ast e_{3, 0} \\
 e_{2, 0} & \ast e_{3, \infty} \\
 e_{2, 0} & \ast e_{3, 0}.
\end{align*}

\section{Asymptotic behaviour of Fenchel--Nielsen co-ordinates}

In this section we will determine the asymptotic behaviour of the co-ordinates reviewed in Section~\ref{sec:Simpson} as $R\to\infty$, for fixed $q \in S^3_1$. 
The constants we will find in this section may all depend on the divisor $D$. 
On the other hand, their dependence on $q$ is crucial, hence we will indicate when 
a constant depends on $q$. 
 
 \subsection{Monodromy of diagonalizing frames}\label{ssec:monodromy}
 For our purpose, we first need to determine the monodromy transformation of a diagonalizing frame of the solution to Hitchin's equations along a loop around 
 a logarithmic point or a ramification point.

 Clearly, the gauge transformation $g_{\sqrt{R}}$ provided by Theorem~\ref{thm:Moc} is unique up to a reducible transformation, i.e. one preserving the decomposition of~\eqref{eq:semisimple_model_flat_connection} into abelian summands. 
 Consider now any simple loop 
 $$
 \gamma\colon[0,1] \to \C \setminus \Delta_{q}. 
 $$
 \begin{defn}
 Let $k(\gamma , q ) \in \Z / (2)$ be the number of points of $\Delta_{q}$ 
 contained in one of the connected components of $\CP1 \setminus \gamma$, counted with multiplicity and modulo $2$. 
 \end{defn}
 Notice that the number of points of $\Delta_{q}$ in the two connected components 
 of $\CP1 \setminus \gamma$ add up to $6$, so $k(\gamma , q )$ is independent of the chosen component. 
 
 The loop $\gamma$ may be covered by a finite union of compact discs $K_1, \ldots ,K_N$ as in Theorem~\ref{thm:Moc}, so we get for each $K_l$ a local holomorphic trivialization 
 $(\vec{e}_1^l, \vec{e}_2^l)$ of $\E$ specified by the local gauges $g_R$. 
 We assume that for each $1\leq l \leq N$ we have $K_l \cap K_{l+1} \cap \gamma ( [0,1] ) \neq \varnothing$ (where $l=N+1$ is identified with $l=1$), and pick any point $\gamma ( \tau_l ) \in K_l \cap K_{l+1}$. 
 For $1\leq l \leq N -1$, up to applying a constant gauge transformation over $K_{l+1}$ we may assume that 
 $$
 (\vec{e}_1^l (\gamma ( \tau_l ) ), \vec{e}_2^l(\gamma ( \tau_l ) ) )  = (\vec{e}_1^{l+1} (\gamma ( \tau_l ) ), \vec{e}_2^{l+1}(\gamma ( \tau_l ) ) ). 
 $$
 Let $M(\gamma , Rq )$ be the monodromy transformation of the local trivializations, defined by 
 $$
  (\vec{e}_1^N (\gamma ( \tau_N ) ), \vec{e}_2^N(\gamma ( \tau_N ) )) = (\vec{e}_1^1 (\gamma ( \tau_N ) ), \vec{e}_2^1 (\gamma ( \tau_N ) )) M(\gamma , Rq ). 
 $$
 Let  $T$ stand for the transposition matrix 
 $$
  T = \begin{pmatrix}
   0 & 1 \\
   1 & 0
  \end{pmatrix}. 
 $$
 \begin{prop}\label{prop:monodromy_frame1}
  For any simple loop $\gamma$ we have 
  $$
    M(\gamma , Rq ) = \begin{pmatrix}
                  \alpha (\gamma , Rq )  & 0 \\
                  0 & \delta (\gamma , Rq )
                 \end{pmatrix}
    T^{k(\gamma , q )} 
  $$
  for some $\alpha (\gamma , Rq ) , \delta (\gamma , Rq ) \in \C^{\times}$. 
 \end{prop}

 \begin{proof}
 Recall from~\eqref{eq:discriminant2} that $\Delta_{q} = D \cup \{ t(q) \}$. 
  Assume first that $t(q)\notin D$. Then $X_{R q}$ is smooth. Now, since all 
  ramification points of $p|_L\colon X_{R q}\to \CP1$ 
  are of index $2$, the lift $\tilde{\gamma}$ of $\gamma$ to $X_{R q}$ is a loop if and only if $k( \gamma , q ) = 0$. 
  Let $\tilde{\zeta}_{\pm}$ be a continuous lift of $\zeta_{\pm}$ to 
  $\tilde{\gamma}$. Then, we have  
  $$
    \tilde{\zeta}_{\pm} (R q, \tilde{\gamma}(1)) = (-1)^{k( \gamma , q )} 
    \tilde{\zeta}_{\pm} (R q, \tilde{\gamma}(0)) .
  $$
  As $\vec{e}_1^l$ and $\vec{e}_2^l$ belong to the $\tilde{\zeta}_+$- and $\tilde{\zeta}_-$-eigenspaces of $\theta$ respectively over $K_l$, 
  and eigenvectors $\vec{e}_1^0$ and $\vec{e}_1^N$ in case $k( \gamma , q ) = 0$ 
  (respectively, $\vec{e}_1^0$ and $\vec{e}_2^N$ in case $k( \gamma , q ) = 1$) 
  for the eigenvalue $\tilde{\zeta}_+$ only differ by some nonzero scalar $\alpha$, we get the result. 

  In case $t(q)\in D$, i.e. $t(q) = t_j$ for some $0 \leq j \leq 4$, the curve $X_{R q}$ has an ordinary double point at $(t_j,0)$, 
  hence the form $\omega$ is unramified over $t_j$. 
  If $\gamma$ is a loop enclosing $t_j$ and no other point of $D$ then $M(\gamma )= \mbox{I}$ and $k (\gamma , q ) \equiv 0 \pmod{2}$ 
  (because the points of $D$ are counted with multiplicity), so we conclude by the equality 
  $$
     T^2 = \mbox{I}. 
  $$ 
  In the case of an arbitrary loop $\gamma$, one concludes by a combination of the above arguments.
 \end{proof}

Now, assume that $t ( q ) \notin D$. Using the notations of~\eqref{eq:Hopf}, let us set 
\begin{equation}\label{eq:tauj}
   \tau_j  =  \tau_j ( q )  = \res_{z = t_j} ( \partial_z^{\otimes 2} \angle Q ( z ) ) = \frac{a t_j - b}{\prod_{0 \leq k \leq 4, k\neq j} (t_j - t_k )} \in \C
\end{equation}
(where $\angle$ stands for contraction of tensor fields) and introduce the local holomorphic co-ordinate 
\begin{equation}\label{eq:zj}
   \tilde{z}_j = \tau_j (z- t_j). 
\end{equation}
This is indeed a local co-ordinate by the assumption that the root 
$t ( q )$ of the linear functional $az-b$ does not belong to $D$, which 
means $\tau_j \neq 0$. 
Notice that as $\tau_j$ depends continuously on $q$, there exists some $M>0$ only depending on $t_0, \ldots , t_4$ such that for all $q \in S^3_1 (\vec{0})$ and all $0\leq j \leq 4$ we have 
\begin{equation}\label{eq:bound_tauj}
    | \tau_j ( q ) | \leq M. 
\end{equation}
Then, a simple computation shows that up to holomorphic terms in $z - t_j$ we have 
$$
   \frac{\mbox{d} \tilde{z}_j ^{\otimes 2}}{\tilde{z}_j} \approx Q (z) . 
$$
We write $\tilde{z}_j = \tilde{r}_j e^{\sqrt{-1} \tilde{\varphi}_j}$ for the polar co-ordinates of the local parameter. 
With respect to these polar co-ordinates, the circle $\xi_j$ of radius $r_0$ centered at $t_j$ then has the equation 
\begin{equation}\label{eq:rj_r0}
   \tilde{r}_j = | \tau_j | r_0. 
\end{equation}
More precisely, it follows from~\eqref{eq:tauj} that we have 
$$
  \arg ( \tilde{z}_j ) = \arg ( \tau_j ) + \arg (z- t_j), 
$$
so the parameterization~\eqref{eq:xi_j} of $\xi_j$ becomes 
\begin{equation}\label{eq:xi_j_reparameterization}
   \xi_j (\tilde{\varphi}_j ) = t_j + r_0 e^{\sqrt{-1} \tilde{\varphi}_j} \quad \mbox{for} \; \tilde{\varphi}_j \in [\arg ( \tau_j ), 2\pi + \arg ( \tau_j ) ]. 
\end{equation}
 
  Let $\gamma$ denote the positively oriented simple loop around $t_j$ defined by $\tilde{r}_j = r_j$ for some fixed $0< r_j \ll 1$ chosen so that one of the connected components of $\CP1 \setminus \gamma$ 
  contains no other point of $\Delta_{q}$ than $t_j$. 
  (In the case $r_j = | \tau_j | r_0$ we get $\gamma = \xi_j$, however we are not guaranteed that for a given $q$ this choice of $r_j$ satisfies the above requirement.)
  Let us define the unit norm trivialization 
 \begin{align}\label{eq:diagonalizing_frame_singular}  
  \vec{f}_1 (\tilde{z}_j ) & = 
  \frac 1{\sqrt{ e^{2m_{\sqrt{R}}(\tilde{r}_j)} + 1}}
    \begin{pmatrix}
     e^{m_{\sqrt{R}}(\tilde{r}_j)} \\
     e^{-\sqrt{-1} \tilde{\varphi}_j/2} 
    \end{pmatrix}
  \\ 
  \vec{f}_2 (\tilde{z}_j ) & = 
  \frac 1{\sqrt{ e^{2m_{\sqrt{R}}(\tilde{r}_j)} + 1}}
    \begin{pmatrix}
     e^{m_{\sqrt{R}}(\tilde{r}_j)} \\
     - e^{-\sqrt{-1} \tilde{\varphi}_j/2}
    \end{pmatrix}
    \label{eq:diagonalizing_frame_singular2}
 \end{align}
 in the disc $\tilde{r}_j \leq r_j$ with respect to the fiducial frame~\eqref{eq:fiducial_frame_singular}.  

 \begin{prop}\label{prop:monodromy_frame2}
 Let $\gamma$ be the positive simple loop $\tilde{r}_j = r_j$. 
 \begin{enumerate}
  \item\label{prop:monodromy_frame2_part_1} We have 
  $$
    h^{\fid} (\vec{f}_1, \vec{f}_2) \to 0 \quad (R\to\infty).
  $$
  \item\label{prop:monodromy_frame2_part_2}
  The frame $\vec{f}_1, \vec{f}_2$ diagonalizes the fiducial Higgs field~\eqref{eq:fiducial_solution_singular_Higgs} with eigenvalues 
  \begin{equation}\label{eq:eigenvalues_fiducial_logarithmic}
        \pm \sqrt{R} \tilde{r}_j^{-\frac 12} e^{-\sqrt{-1} \tilde{\varphi}_j/2}\mbox{d} \tilde{z}_j , 
  \end{equation}
  where we take the determination of the angle $\tilde{\varphi}_j \in [0, 2\pi )$. 
  \item\label{prop:monodromy_frame2_part_3} The corresponding factors found in Proposition~\ref{prop:monodromy_frame1} fulfill 
  $$
    \alpha (\gamma , Rq ) = 1 = \delta (\gamma , Rq ). 
  $$
 \end{enumerate}
 \end{prop}

 \begin{proof}
  For part~\eqref{prop:monodromy_frame2_part_1}, as the vectors $\vec{f}_1, \vec{f}_2$ are 
  written in a unitary frame, we simply compute 
  $$
    h^{\fid} (\vec{f}_1, \vec{f}_2) = \frac{e^{2m_{\sqrt{R}}(\tilde{r}_j)} - 1}{e^{2m_{\sqrt{R}}(\tilde{r}_j)} + 1}.
  $$
  Now, observe that by~\eqref{eq:Bessel} we have 
  \begin{align}
    e^{m_{\sqrt{R}}(r_j)} & \approx \exp \left( \frac 1{\pi} K_0 (8\sqrt{R r_j}) \right) \notag \\
    & \approx \exp \left( \frac 1{2 \pi \sqrt{2} \sqrt[4]{R r_j}} e^{-8 \sqrt{R r_j}} \right) \notag \\
    & \to 1 \label{eq:limit_m}
 \end{align}
 as $R\to\infty$, since 
 $$
  (R r_j)^{-\frac 14} e^{-8 \sqrt{R r_j}} \to 0. 
 $$
 
  For part~\eqref{prop:monodromy_frame2_part_2}, we first need to determine the eigendirections of the fiducial Higgs field $\sqrt{R} \theta_{\sqrt{R}}^{\fid}$ with respect to the fiducial frame. 
  We need to find the eigenvalues $\lambda_{\pm}$ of 
 \begin{align*}
   & \sqrt{R}    \begin{pmatrix}
                       0 & \tilde{r}_j^{-1/2} e^{m_{\sqrt{R}} (\tilde{r}_j)} \\
                       \tilde{z}_j^{-1} \tilde{r}_j^{1/2} e^{-m_{\sqrt{R}} (\tilde{r}_j )} & 0 
                      \end{pmatrix} 
                      \mbox{d} \tilde{z}_j
                      \\
                      & = \sqrt{R}
                      \begin{pmatrix}
                       0 & \tilde{r}_j^{-1/2} e^{m_{\sqrt{R}} (\tilde{r}_j)} \\
                       \tilde{r}_j^{-1/2} e^{- \sqrt{-1} \tilde{\varphi}_j -m_{\sqrt{R}} (\tilde{r}_j)} & 0 
                      \end{pmatrix} 
                      \mbox{d} \tilde{z}_j.
 \end{align*}
 A direct computation gives that $\lambda_{\pm}$ are given by~\eqref{eq:eigenvalues_fiducial_logarithmic}, with corresponding eigenspaces spanned by unit vectors $\vec{f}_1, \vec{f}_2$ introduced 
 in~\eqref{eq:diagonalizing_frame_singular},~\eqref{eq:diagonalizing_frame_singular2}. 
 
 For part~\eqref{prop:monodromy_frame2_part_3},
 we fix $\tilde{r}_j = r_j$ and let $\tilde{\varphi}_j$ range over $[0,2\pi ]$, with a branch cut at $\tilde{\varphi}_j = 0$, and we write 
 $$
  \vec{f}_i ( \tilde{\varphi}_j ) = \vec{f}_i ( \tilde{z}_j ). 
 $$
 We find 
 \begin{align*}
  \vec{f}_1 (2\pi ) & = \frac 1{\sqrt{e^{2m_{\sqrt{R}}(\tilde{r}_j)} + 1}}
    \begin{pmatrix}
     e^{m_{\sqrt{R}}(\tilde{r}_j)} \\
     -1 
    \end{pmatrix} 
    = \vec{f}_2 (0 )
    \\
    \vec{f}_2 (2\pi ) & = \frac 1{\sqrt{e^{2m_{\sqrt{R}}(\tilde{r}_j)} + 1}}
    \begin{pmatrix}
     e^{m_{\sqrt{R}}(\tilde{r}_j)} \\
     1 
    \end{pmatrix} 
    = \vec{f}_1 (0 ).
 \end{align*}
 \end{proof}

 Finally, let us study the neighbourhood of the ramification point $t = t(q) = \frac ba$. Here, using the notation of~\eqref{eq:Hopf}, we introduce the local holomorphic co-ordinate 
 $$
  \tilde{z}_t = \left( \frac a{\prod_{j=0}^4 (t-t_j)} \right)^{\frac 13} \left( z - \frac ba \right). 
 $$
 Then, up to at least quadratic terms in $\tilde{z}_t$ we have 
 $$
  \tilde{z}_t \mbox{d} \tilde{z}_t^{\otimes 2} \approx Q (z) . 
 $$
 We then write $\tilde{z}_t = \tilde{r}_t e^{\sqrt{-1} \tilde{\varphi}_t}$ 
 for polar co-ordinates. 
 Let $\gamma$ be the simple positive loop defined by $\tilde{r}_t = r_5$ for some 
 $0< r_5 \ll 1$ so that $\gamma$ separates $t(q)$ from the logarithmic points. 
 Finally, introduce the unit length trivialization 
 \begin{align*}
  \vec{g}_1 (\tilde{z}_t ) & = 
  \frac 1{\sqrt{e^{2\ell_{\sqrt{R}}(\tilde{r}_t)} + 1}}
    \begin{pmatrix}
     e^{\ell_{\sqrt{R}}(\tilde{r}_t)} \\
      e^{\sqrt{-1} \tilde{\varphi}_t/2}
    \end{pmatrix},
  \\
   \vec{g}_2 (\tilde{z}_t ) & = 
   \frac 1{\sqrt{e^{2\ell_{\sqrt{R}}(\tilde{r}_t)} + 1}}
    \begin{pmatrix}
     e^{\ell_{\sqrt{R}}(\tilde{r}_t)} \\
     - e^{\sqrt{-1} \tilde{\varphi}_t/2}
    \end{pmatrix}
 \end{align*}
 over the disc $\tilde{r}_t \leq r_5$ with respect to the fiducial frame~\eqref{eq:fiducial_frame_apparent}. 
 
 \begin{prop}\label{prop:monodromy_frame3}
 Let $\gamma$ be a simple loop enclosing the ramification point $t(q)$ in counterclockwise direction such that the component of $\CP1 \setminus \gamma$ containing $t(q)$ contains no logarithmic point $t_j$. 
 Then the frame $\vec{g}_1, \vec{g}_2$ diagonalizes the fiducial Higgs field~\eqref{eq:fiducial_solution_apparent_Higgs}, and the corresponding factors found in Proposition~\ref{prop:monodromy_frame1} fulfill 
  $$
    \alpha (\gamma , Rq ) = 1 = \delta (\gamma , Rq ) .
  $$
\end{prop}
 
\begin{proof}
 Similar to Proposition~\eqref{prop:monodromy_frame2}, up to the following modifications: the eigenvalues of the fiducial Higgs field are 
 $\pm \sqrt{R\tilde{r}_t}e^{\sqrt{-1} \tilde{\varphi}_t/2}\mbox{d} z$, with corresponding eigendirections $\vec{g}_1, \vec{g}_2$. 
\end{proof}

\subsection{Asymptotics of complex length co-ordinates}\label{sec:asymptotic_length}

Here, we will study the behaviour of the complex length co-ordinates 
$(l_2, l_3)$ of $\Mod_{\Betti}(\vec{c}, \vec{0})$ introduced in 
Section~\ref{sec:length} for the local systems obtained by applying the non-abelian 
Hodge and Riemann--Hilbert correspondences to a Higgs bundle in a Hitchin 
fiber close to infinity. More precisely, we set 
\begin{equation}\label{eq:ti}
 l_i ( \E, \sqrt{R} \theta  ) = \tr \RH ( \nabla_{\sqrt{R}}) [ \rho_i ]. 
\end{equation}
Notice that the connection $\nabla_{\sqrt{R}}$ depends on 
$(\E, \sqrt{R} \theta )$, hence it is justified to include the dependence of 
$l_i$ on $(\E, \sqrt{R} \theta )$ in the notation. However, to lighten 
notation we will sometimes simply write $l_i$. With these notations, we will 
determine the asymptotic behaviour as $R\to \infty$ of 
$l_2 (\E, \sqrt{R} \theta ), l_3 (\E, \sqrt{R} \theta )$ for any 
$(\E, \theta ) \in H^{-1} (q)$, where $q\in S^3_1$ is fixed. 

According to Theorem~\ref{thm:Moc} for $R\gg 0$ and at any point in the complement of $ \Delta_q$ there exists a $1$-parameter family 
of frames that asymptotically diagonalize $\theta$. The family is obtained by rescaling a given frame by diagonal elements of $\Sl (2, \C )$. 
(One may additionally apply the only non-trivial element $T$ of the Weyl group, with the effect of exchanging the two trivializations of the frame.) 
It follows that there exists (up to permutation and the action of the Cartan 
subgroup $S^1 \subset \Sl(2, \C )$) a unique such orthonormal frame. For any loop $\gamma$ in $\CP1 \setminus  \Delta_q$ let us write 
\begin{equation}\label{eq:RH_entries}
   \RH ( \nabla_{\sqrt{R}}) [ \gamma ] = \begin{pmatrix}
                                                       a ( \gamma, (\E, \sqrt{R} \theta )  ) & b ( \gamma, (\E, \sqrt{R} \theta )  ) \\
                                                       c ( \gamma, (\E, \sqrt{R} \theta )  ) & d ( \gamma, (\E, \sqrt{R} \theta )  )
                                                      \end{pmatrix}
\end{equation}
with respect to this (essentially) unique orthonormal base of the fiber $V|_{\gamma(0)}$ of the underlying smooth vector bundle $V$ at $\gamma(0)$. 
Notice that the effect of the action by the Cartan subgroup 
means that the off-diagonal entries are only defined up to a common phase factor. 
Our aim in this section is to study the asymptotic behaviour of the entries of $\RH ( \nabla_{\sqrt{R}}) [ \rho_i ]$ for $2\leq i \leq 3$, and 
in particular their trace. Clealy, the set of eigenvalues (hence the trace) 
is invariant with respect to the action of the Weyl group. 
In order to achieve this, we will decompose the class of $\rho_i$ in $\pi_1 ( \CP1 \setminus  \Delta_q, \rho_i (0) )$ into a concatenation of several loops (see Figure~\ref{fig}). 
The number of loops appearing in this decomposition will be $2$ or $3$, depending on the position of the ramification point $t(q)$ with respect to the decomposition of $S$ into pairs of pants~\eqref{eq:pants}. 
Around each of the loops appearing in the decomposition we will explicitly determine the monodromy, and the monodromy around $\rho_i$ is essentially the product of the monodromies of the consituent loops. 

\begin{prop}\label{prop:flat_connection_form}
 For any fixed $q \in S^3_1 $ such that $t(q) \notin D_j$, the connection form of the flat connection associated to the fiducial solution~\eqref{eq:fiducial_solution_singular_unitary},~\eqref{eq:fiducial_solution_singular_Higgs} 
 restricted to the curve $\xi_j$ (given by $\tilde{r}_j = r_j$) with respect to the unit diagonalizing frame~\eqref{eq:diagonalizing_frame_singular},~\eqref{eq:diagonalizing_frame_singular2} 
 of the Higgs field reads as 
 $$
  \begin{pmatrix}
    \frac 34 +  2 \sqrt{-1} \sqrt{R r_j} \sin \left( \frac{\tilde{\varphi}_j}2 \right) & -  \frac 12 r_j \partial_{\tilde{r}} m_{\sqrt{R}} (r_j) \\
   - \frac 12 r_j \partial_{\tilde{r}} m_{\sqrt{R}} (r_j) &  \frac 34 -  2 \sqrt{-1} \sqrt{R r_j} \sin \left( \frac{\tilde{\varphi}_j}2 \right)
  \end{pmatrix}
  \sqrt{-1} \mbox{d} \tilde{\varphi}_j.
 $$
\end{prop}
\begin{proof}
 This follows from a straightforward computation. Set 
 $$
  F = \frac 1{\sqrt{e^{2m_{\sqrt{R}}(r_j)} + 1}}
    \begin{pmatrix}
     e^{m_{\sqrt{R}}(r_j)} &  e^{m_{\sqrt{R}}(r_j)} \\
     e^{-\sqrt{-1} \tilde{\varphi}_j/2} & - e^{-\sqrt{-1} \tilde{\varphi}_j/2}
    \end{pmatrix}
 $$
 for the matrix formed by the restrictions of the column vectors~\eqref{eq:diagonalizing_frame_singular},~\eqref{eq:diagonalizing_frame_singular2} to $\xi_j$, with determinant 
 $$
  \det (F) = - \frac{2 e^{m_{\sqrt{R}}(r_j) - \sqrt{-1} \tilde{\varphi}_j/2}}{e^{2m_{\sqrt{R}}(\tilde{r}_j)} + 1}
 $$
 and inverse matrix given by 
 $$
  F^{-1} = \frac{\sqrt{e^{2m_{\sqrt{R}}(r_j)} + 1}}{2 e^{m_{\sqrt{R}}(r_j) - \sqrt{-1} \tilde{\varphi}_j/2}}
    \begin{pmatrix}
       e^{-\sqrt{-1} \tilde{\varphi}_j/2} &  e^{m_{\sqrt{R}}(r_j)} \\
      e^{-\sqrt{-1} \tilde{\varphi}_j/2} & - e^{m_{\sqrt{R}}(r_j)}
    \end{pmatrix}. 
 $$
 Since on $\xi_j$ we have $\mbox{d} \tilde{r}_j = 0$, we need to compute the $\mbox{d} \tilde{\varphi}_j$-part of 
 $$
  F^{-1} \cdot ( \mbox{d} + A_{\sqrt{R}}^{\fid} + \theta_{\sqrt{R}}^{\fid} + (\theta_{\sqrt{R}}^{\fid})^{\dagger } ) = 
  - F^{-1} \mbox{d} F +  \mbox{Ad}_{F^{-1}} ( A_{\sqrt{R}}^{\fid} + \theta_{\sqrt{R}}^{\fid} + (\theta_{\sqrt{R}}^{\fid})^{\dagger } ) .
 $$
 Note first that $\mbox{Ad}_{F^{-1}}$ acts trivially on the central part of ~\eqref{eq:fiducial_solution_singular_unitary}. On the other hand, a computation shows that 
 \begin{align*}
  & F^{-1} \begin{pmatrix}
          F_{\sqrt{R}}(r_j) & 0 \\
          0 & - F_{\sqrt{R}}(r_j)
         \end{pmatrix}
  F \\
  & = \frac 1{2 e^{m_{\sqrt{R}}(r_j) - \sqrt{-1} \tilde{\varphi}_j/2}}
  \begin{pmatrix}
   0 & - 2  e^{m_{\sqrt{R}}(r_j) - \sqrt{-1} \tilde{\varphi}_j/2} F_{\sqrt{R}}(r_j) \\
   - 2 e^{m_{\sqrt{R}}(r_j) - \sqrt{-1} \tilde{\varphi}_j/2} F_{\sqrt{R}}(r_j) & 0
  \end{pmatrix} 
  \\
  & = 
   \begin{pmatrix}
   0 & - F_{\sqrt{R}}(r_j) \\
   - F_{\sqrt{R}}(r_j) & 0
  \end{pmatrix} .
 \end{align*}
 By Proposition~\ref{prop:monodromy_frame2}, $\mbox{Ad}_{F^{-1}} ( \theta_{\sqrt{R}}^{\fid} + (\theta_{\sqrt{R}}^{\fid})^{\dagger })$ is diagonal with eigenvalues given by 
 $$
  \pm \sqrt{R r_j} \sqrt{-1}  ( e^{\sqrt{-1} \tilde{\varphi}_j/2}  - e^{-\sqrt{-1} \tilde{\varphi}_j/2}  )  \mbox{d} \tilde{\varphi}_j  = \mp 2 \sqrt{R r_j}  \sin \left( \frac{\tilde{\varphi}_j}2 \right) \mbox{d} \tilde{\varphi}_j . 
 $$
 Lastly, restricted to $\xi_j$ we find 
 \begin{align*}
  F^{-1} \mbox{d} F = & \frac 1{2 e^{m_{\sqrt{R}}(r_j) - \sqrt{-1} \tilde{\varphi}_j/2}}
    \begin{pmatrix}
      e^{-\sqrt{-1} \tilde{\varphi}_j/2} &   e^{m_{\sqrt{R}}(r_j)} \\
      e^{-\sqrt{-1} \tilde{\varphi}_j/2} & - e^{m_{\sqrt{R}}(r_j)}
    \end{pmatrix} \\
    & \cdot 
    \begin{pmatrix}
     0 & 0 \\
     - \frac{\sqrt{-1}}2 e^{-\sqrt{-1} \tilde{\varphi}_j/2} & \frac{\sqrt{-1}}2 e^{-\sqrt{-1} \tilde{\varphi}_j/2}
    \end{pmatrix}
    \mbox{d} \tilde{\varphi}_j 
    \\
    = & \frac{\sqrt{-1}}4 
        \begin{pmatrix}
         - 1 & 1 \\
         1 & - 1
        \end{pmatrix}
       \mbox{d} \tilde{\varphi}_j . 
 \end{align*}
 We conclude combining the above computations and using the identity (see~\eqref{eq:function_F}) 
 $$
  2 F_{\sqrt{R}}(r_j) + \frac 14 = \frac 12 r_j \partial_{\tilde{r}} m_{\sqrt{R}} (r_j).
 $$
\end{proof}

The behaviour of the entries of the matrix~\eqref{eq:RH_entries} for $R\gg 0$ and the choice $\gamma = \xi_j$ is given by the following. 

\begin{prop}\label{prop:asymptotic_RH}
 Fix any $q \in S^3_1 $ and consider the loop $\gamma = \xi_j$. 
 \begin{enumerate}
  \item \label{prop:asymptotic_RH1}
 The behaviour of the diagonal entries of~\eqref{eq:RH_entries} as $R \to \infty$ is given by the limits  
 \begin{align*}
  a ( \xi_j, (\E, \sqrt{R} \theta )  ) & \to 0 \\
  d ( \xi_j, (\E, \sqrt{R} \theta )  ) & \to 0 
 \end{align*}
 as $R \to \infty$, at exponential rate in $\sqrt{R}$. 
  \item \label{prop:asymptotic_RH2}
  The behaviour of the off-diagonal entries of~\eqref{eq:RH_entries} as $R \to \infty$ is given by the limits 
 \begin{align*}
   b ( \xi_j, (\E, \sqrt{R} \theta )  )e^{ 8 \Re \sqrt{\tau_j} \sqrt{R r_0}} & \to \sqrt{-1}  \\
   c ( \xi_j, (\E, \sqrt{R} \theta )  )e^{-8 \Re \sqrt{\tau_j} \sqrt{R r_0}} & \to \sqrt{-1} 
 \end{align*}
 where $\tau_j$ is defined in~\eqref{eq:tauj} and $r_0 > 0$ is the radius of $\xi_j$ in the Euclidean metric. 
 \end{enumerate}
\end{prop}

\begin{proof}
Let us set $r_j = | \tau_j | r_0$ (see~\eqref{eq:rj_r0}). 
Recall the reparameterization~\eqref{eq:xi_j_reparameterization} of $\xi_j$ with respect to the polar co-ordinates of the local holomorphic chart~\eqref{eq:zj}. 
Integrating the connection form of the flat connection found in Proposition~\ref{prop:flat_connection_form} from $\arg ( \tau_j )$ to $2\pi + \arg ( \tau_j )$ with respect to $\tilde{\varphi}_j$ we find the matrix 
 \begin{equation}\label{eq:integral_connection_form}
  \begin{pmatrix}
   \sqrt{-1} \frac{3 \pi}2 - 8 \cos \left( \frac{\arg ( \tau_j )}2 \right) \sqrt{R r_j} & - \sqrt{-1} \pi r_j \partial_{\tilde{r}} m_{\sqrt{R}} (r_j) \\
  - \sqrt{-1} \pi r_j \partial_{\tilde{r}} m_{\sqrt{R}} (r_j) & \sqrt{-1} \frac{3 \pi}2 + 8 \cos \left( \frac{\arg ( \tau_j )}2 \right) \sqrt{R r_j}
  \end{pmatrix}.
 \end{equation}   
There are two cases to consider depending on whether $\Re \sqrt{\tau_j} = 0$ or $\Re \sqrt{\tau_j} \neq 0$. 

We first treat the case $\Re \sqrt{\tau_j} \neq 0$. 
This condition is equivalent to the pair of conditions $|\tau_j|\neq 0$ (equivalently, $t(q) \notin D$) and  
$$
  \cos \left( \frac{\arg ( \tau_j )}2 \right) \neq 0, 
$$
equivalently $\arg(\tau_j ) \neq \pi + 2k \pi$ for $k\in\Z$. 
 The matrix~\eqref{eq:integral_connection_form} is then of the form 
 $$
  \begin{pmatrix}
   A - C & B \\
   B & A + C
  \end{pmatrix}
 $$
 with 
  \begin{align*}
    A & = A_j = \sqrt{-1} \frac{3 \pi}2 \\
    B & = B_j = - \sqrt{-1} \pi r_j \partial_{\tilde{r}} m_{\sqrt{R}} (r_j), \\
    C & = C_j = 8 \cos \left( \frac{\arg ( \tau_j )}2 \right) \sqrt{R r_j} = 8 \Re \sqrt{\tau_j} \sqrt{R r_0} \neq 0, 
 \end{align*}
 where we have used~\eqref{eq:rj_r0} in the third line. 
 A straightforward computation shows that setting $D = \sqrt{B^2 + C^2}$ the exponential of the negative of the above matrix is 
 \begin{equation}\label{eq:exponential_of_integral}
  \frac{-\sqrt{-1}}{2 D}
  \begin{pmatrix}
   ( - C + D ) e^{-D} +  ( C + D ) e^{D} & - 2 B \sinh D \\
   - 2 B \sinh D & ( C + D ) e^{-D} +  ( - C + D ) e^{D}
  \end{pmatrix}.  
 \end{equation} 
 According to~\eqref{eq:Bessel} as $R\to \infty$ with $r_j$ freezed, we have 
 \begin{align}
  B \frac{e^{8 \sqrt{R r_j}}}{\sqrt{\pi}\sqrt[4]{R r_j}} & \to \sqrt{-1} \label{eq:asymptotics_B} \\  
  \frac C{\sqrt{R r_j}} & \to \upsilon_j  \label{eq:asymptotics_C}
 \end{align}
 where we have set 
 \begin{equation}\label{eq:upsilon}
   \upsilon_j = 8 \cos \left( \frac{\arg ( \tau_j )}2 \right) \in [-8,8] \setminus \{ 0 \}. 
 \end{equation}
 As a consequence we find 
 $$
    D \approx |C|. 
 $$
  Moreover, notice that up to higher order terms we have 
 $$
  \sqrt{B^2 + C^2} = C \sqrt{1 + \frac{B^2}{C^2}} \approx C \left( 1 + \frac{B^2}{2 C^2} \right), 
 $$
 implying 
 $$
  \frac{-C + \sqrt{B^2 + C^2}}{2\sqrt{B^2 + C^2}} e^{\pm \sqrt{B^2 + C^2}} \approx \frac{B^2}{4 C^2} e^{\pm C} \approx -\frac{\pi}{4 \upsilon_j^2} \frac{e^{( -16 \pm \upsilon_j) \sqrt{R r_j}}}{\sqrt{R r_j}}.
 $$
 We infer that the leading order terms of the matrix~\eqref{eq:exponential_of_integral} are equal to 
 \begin{align*}
  & \sqrt{-1} 
  \begin{pmatrix}
    \frac{B^2}{4 C^2} e^{-C} +  e^C & - \frac BC  e^{|C|} \\
   - \frac BC e^{|C|} & e^{-C} +  \frac{B^2}{4 C^2} e^C 
  \end{pmatrix}
  \\
  & \approx - \sqrt{-1} \cdot \\
  & \cdot 
    \begin{pmatrix}
     \frac{\pi}{4 \upsilon_j^2 \sqrt{R r_j}} e^{( -16 - \upsilon_j) \sqrt{R r_j}} - e^{\upsilon_j \sqrt{R r_j} } & \frac{\sqrt{-1}\sqrt{\pi}e^{-(8  - |\upsilon_j| ) \sqrt{R r_j}}}{\upsilon_j \sqrt[4]{R r_j}} \\
     \frac{\sqrt{-1}\sqrt{\pi}e^{-(8  - |\upsilon_j| ) \sqrt{R r_j}}}{\upsilon_j \sqrt[4]{R r_j}} & - e^{- \upsilon_j \sqrt{R r_j}} + \frac{\pi}{4 \upsilon_j^2 \sqrt{R r_j}} e^{( -16 + \upsilon_j) \sqrt{R r_j}}
    \end{pmatrix}
 \end{align*}
 This matrix describes the action of $\RH ( \nabla_{\sqrt{R}}) [ \xi_j ]$ with respect to the bases 
 $$
  \vec{f}_1 (0 ), \vec{f}_2 (0 ) \quad \mbox{and} \quad \vec{f}_1 (2\pi ), \vec{f}_2 (2\pi )
 $$
 of the fiber $V|_{\xi_j (0)} = V|_{\xi_j (1)}$. In order to find the matrix of $\RH ( \nabla_{\sqrt{R}}) [ \xi_j ]$ with respect to the single basis $\vec{f}_1 (0 ), \vec{f}_2 (0 )$, 
 we need to multiply the above matrix from the left by the inverse of $M(\xi_j , Rq )$. By Propositions~\ref{prop:monodromy_frame1} and~\ref{prop:monodromy_frame2}, 
 $M(\xi_j , Rq ) = T$ and the product is 
 \begin{align}	
  & \RH ( \nabla_{\sqrt{R}}) [ \xi_j ] = - \sqrt{-1} \cdot \label{eq:RH_xi} \\
  & \cdot \begin{pmatrix}
     \frac{\sqrt{-1}\sqrt{\pi}e^{-(8  - |\upsilon_j| ) \sqrt{R r_j}}}{\upsilon_j \sqrt[4]{R r_j}} & - e^{- \upsilon_j \sqrt{R r_j}} + \frac{\pi}{4 \upsilon_j^2 \sqrt{R r_j}} e^{( -16 + \upsilon_j) \sqrt{R r_j}} \\
     \frac{\pi}{4 \upsilon_j^2 \sqrt{R r_j}} e^{( -16 - \upsilon_j) \sqrt{R r_j}} - e^{\upsilon_j \sqrt{R r_j} } & \frac{\sqrt{-1}\sqrt{\pi}e^{-(8  - |\upsilon_j| ) \sqrt{R r_j}}}{\upsilon_j \sqrt[4]{R r_j}}
    \end{pmatrix}.\notag
 \end{align}
 By~\eqref{eq:upsilon} we have $8  - |\upsilon_j| \geq 0$, whence we immediately get part~\eqref{prop:asymptotic_RH1} (in the case $8  - |\upsilon_j| = 0$, the assertion follows from the factor $\sqrt[4]{R}$ in the denominator). 
 On the other hand,~\eqref{eq:upsilon} also shows that $-16 - \upsilon_j \leq \upsilon_j$, with equality if and only if $\cos \left( \frac{\arg ( \tau_j )}2 \right) = - 1$. 
 In the case where  $-16 - \upsilon_j < \upsilon_j$, the first term of $c ( \xi_j, (\E, \sqrt{R} \theta )  )$ is negligible compared to the second one, and we get~\eqref{prop:asymptotic_RH2} for $c ( \xi_j, (\E, \sqrt{R} \theta )  )$. 
 In case $\cos \left( \frac{\arg ( \tau_j )}2 \right) = - 1$, the exponential factors in the two terms of $c ( \xi_j, (\E, \sqrt{R} \theta )  )$ agree, however the polynomial term in $R$ converges to $0$ 
 for the first term, while is constant for the second term, again implying~\eqref{prop:asymptotic_RH2} for $c ( \xi_j, (\E, \sqrt{R} \theta )  )$. 
 A similar argument may be applied to get~\eqref{prop:asymptotic_RH2} for $b ( \xi_j, (\E, \sqrt{R} \theta )  )$. 

 We now turn to the case $\Re \sqrt{\tau_j} = 0$. In this case,~\eqref{eq:integral_connection_form} simplifies to 
$$
  \sqrt{-1}
   \begin{pmatrix}
     \frac{3 \pi}2 & -  \pi r_j \partial_{\tilde{r}} m_{\sqrt{R}} (r_j) \\
  - \pi r_j \partial_{\tilde{r}} m_{\sqrt{R}} (r_j) & \frac{3 \pi}2
  \end{pmatrix}.
$$
The diagonal entries of this matrix are constant, and its off-diagonal ones converge to $0$ as $R \to \infty$. 
By continuity, the matrix exponential of the negative of this matrix converges to 
$$
  \sqrt{-1}
  \begin{pmatrix}
   1 & 0 \\
   0 & 1
  \end{pmatrix}. 
$$
In order to obtain the matrix of $\RH ( \nabla_{\sqrt{R}}) [ \xi_j ]$ with respect to the single basis 
$\vec{f}_1 (0 ), \vec{f}_2 (0 )$, we again need to multiply by the transposition matrix $T$, and this gives the 
desired formulas. 
\end{proof}

Next, we will consider the loop based at $x_2$ 
$$
  \zeta_2 = \eta_1 \ast \xi_1 \ast \eta_1^{-1} \ast \eta_2 \ast \xi_2 \ast \eta_2^{-1}
$$
enclosing the punctures $t_1, t_2$ once in counterclockwise direction. Clearly, the classes 
$$
 [\rho_2 ]\in \pi_1 ( \CP1 \setminus D, s_2  )  \quad \mbox{and} \quad [\zeta_2 ] \in \pi_1 ( \CP1 \setminus D, x_2  )
$$
are conjugate to each other by $\psi_2$, so that $\RH ( \nabla_{\sqrt{R}}) [ \zeta_2 ]$ is conjugate in $\Sl (2, \C)$ to $\RH ( \nabla_{\sqrt{R}}) [ \rho_2 ]$ 
by the parallel transport map of $\nabla_{\sqrt{R}}$ along $\psi_2$. In particular, this implies that the trace of $\RH ( \nabla_{\sqrt{R}}) [ \zeta_2 ]$ 
agrees with the trace $l_2 (\E, \sqrt{R} \theta )$ of $\RH ( \nabla_{\sqrt{R}}) [ \rho_2 ]$. 
Our next aim is to compute the asymptotic behaviour of the coefficients of $\RH ( \nabla_{\sqrt{R}}) [ \zeta_2 ] $ as $R \to \infty$. 
In order to state the result, we need some preparation. 
We will consider the part lying between $t_j$ and $t_j + r_j$ of the ray emanating out from $t_j$ with direction parallel to the positive real line, and denote this path by $\sigma_j$ (see Figure~\ref{fig2}). 
We then set for $j\in \{1 , 2 \}$ 
\begin{equation}\label{eq:pij}
   \pi_j = \pi_j (q) = \int_{x_2}^{t_j} Z_+ (q, z), 
\end{equation}
the contour of the line integral being $\eta_j * \sigma_j^{-1}$. 
Notice that this is a convergent improper integral; indeed, by~\eqref{eq:Z_tilde2} and~\eqref{eq:Hopf} the integrand grows as $|z-t_j|^{-\frac 12}$ near $t_j$. 
We will indicate the dependence of $\pi_j$ on $q$ whenever we vary it. 
Notice that by its definition~\eqref{eq:Z_tilde2}, $Z_+$ (hence $\pi_j$) is only defined up to a sign. 
We take $Z_+$ to be the square root that is the continuous extension to $\eta_j$ of the square root corresponding to the negative sign  in~\eqref{eq:eigenvalues_fiducial_logarithmic}. 
We will return to our choice of sign in~\eqref{eq:sign_choice}.

\begin{prop}\label{prop:monodromy_zeta2}
Fix $q \in S^3_1 $ and consider the loop $\gamma = \rho_2$. 
In case $\Re ( \pi_1 - \pi_2) \neq 0$ we have the limit 
$$
  l_2 (\E, \sqrt{R} \theta )^{-1} 
 2 \cosh \left( 2 \int_{\eta_2 - \eta_1}  B_{\L_{\E }} + 4 \sqrt{R} \Re ( \pi_2 - \pi_1) \right) 
  \to - 1 
$$
as $R\to\infty$. 
In case $\Re ( \pi_1 - \pi_2)  =0$ the limit of $l_2 ( \E, \sqrt{R} \theta  )$ as $R \to \infty$ exists and is finite. 
\end{prop}

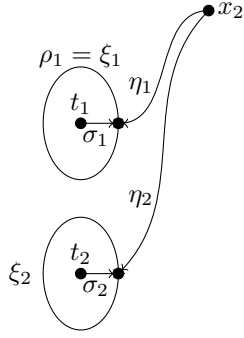
\begin{figure}
\centering
\begin{tikzpicture}[baseline = 5cm]
 \draw (-4.2,1) circle [x radius = 0.5cm, y radius = 0.75cm];
 \draw (-4.2,-1) circle [x radius = 0.5cm, y radius = 0.75cm];
 \filldraw [black] (-2.5,2.5) circle (2pt) node [anchor = west] {$x_2$};
 \filldraw [black] (-3.7,1) circle (2pt);
 \filldraw [black] (-3.7,-1) circle (2pt);
 \filldraw [black] (-4.2,1) circle (2pt)node [anchor = south] {$t_1$};
 \filldraw [black] (-4.2,-1) circle (2pt)node [anchor = south] {$t_2$};
 \draw[->>] (-2.5,2.5) to [out = 180, in = 0] (-3.7,1); 
 \draw[->>](-2.5,2.5) to [out = 225, in = 45] (-3.7,-1); 
 \draw (-4.2,1.9) node {$\rho_1 = \xi_1$};
 \draw (-5,-1) node {$\xi_2$};
 \draw (-3.4,1.5) node {$\eta_1$};
 \draw (-3.4,0) node {$\eta_2$};
 \draw[->>] (-4.2,-1) -- (-3.7,-1); 
 \draw[->>] (-4.2,1) -- (-3.7,1); 
 \draw (-4,-1.2) node {$\sigma_2$};
 \draw (-4,0.8) node {$\sigma_1$};
\end{tikzpicture}
\caption{Paths $\sigma_1, \sigma_2$.}
\label{fig2}
\end{figure}

\begin{proof}
 As mentioned above, it is sufficient to find the asymptotic behaviour of the diagonal entries of $\RH ( \nabla_{\sqrt{R}}) [ \zeta_2 ]$. 
 By Theorem~\ref{thm:Moc}, parallel transport map of $\nabla_{\sqrt{R}}$ along $\eta_j$ with respect to a diagonalizing frame of the Higgs field is approximated 
 by the matrix 
 \begin{equation}\label{eq:parallel_transport_eta}
  P_j ( \E, \sqrt{R} \theta  )  = 
  \begin{pmatrix}
   e^{\int_{\eta_j} \frac 12 B_{\det (\E )} + B_{\L_{\E }} + \sqrt{R} \Re Z_+ (q, z) } & 0 \\
   0 & e^{\int_{\eta_j} \frac 12 B_{\det (\E )} - B_{\L_{\E }}  + \sqrt{R} \Re Z_- (q, z)} 
  \end{pmatrix}.
 \end{equation}
 Parallel transport $P_j ( \E, \sqrt{R} \theta  )$ identifies the unit length diagonalizing frame 
 $$
  (\vec{e}_+ (\eta_j(0)), \vec{e}_- (\eta_j(0)))
 $$
 of $\theta_{\sqrt{R}}$ at $\eta_j(0)$ with a diagonalizing frame 
 \begin{equation}\label{eq:non_unit_diagonalizing_frame}
    (\vec{e}_+ (\eta_j(1)), \vec{e}_- (\eta_j(1)))
 \end{equation}
 at $\eta_j(1)$. This latter, however, is not of unit length; instead, the lengths of its vectors are given for $i \in \{ \pm \}$ by 
 $$
  \vert \vec{e}_i (\eta_j(1)) \vert = e^{\sqrt{R} \int_{\eta_j} \Re Z_i (q, z)}. 
 $$
 So, by suitably picking the phases of the frame vectors $e_+, e_-$, 
 we may assume that the matrix expressing the basis elements 
 of~\eqref{eq:non_unit_diagonalizing_frame} with respect 
 to~\eqref{eq:diagonalizing_frame_singular} is given by 
 $$
  Q_j ( \E, \sqrt{R} \theta  )  = 
  \begin{pmatrix}
   e^{\int_{\eta_j} \sqrt{R} \Re Z_+ (q, z) } & 0 \\
   0 & e^{\int_{\eta_j} \sqrt{R} \Re Z_- (q, z)} 
  \end{pmatrix}.
 $$
 It follows that the monodromy matrix of $\nabla_{\sqrt{R}}$ along the loop $\eta_j \ast \xi_j \ast \eta_j^{-1}$ with respect to the frame~\eqref{eq:non_unit_diagonalizing_frame} is equal to 
 $$   
 \Ad_{P_j ( \E, \sqrt{R} \theta  )^{-1}} \circ \Ad_{Q_j ( \E, \sqrt{R} \theta  )^{-1}} 
  ( \RH ( \nabla_{\sqrt{R}}) [ \xi_j ] ).
 $$
 In view of Proposition~\ref{prop:asymptotic_RH} we find 
 \begin{align}
   a ( \eta_j \ast \xi_j \ast \eta_j^{-1} , (\E, \sqrt{R} \theta )  ) e^{(8 - \upsilon_j) \sqrt{R r_j}} \upsilon_j \sqrt[4]{R r_j} & \to \sqrt{\pi} \notag \\
   b ( \eta_j \ast \xi_j \ast \eta_j^{-1} , (\E, \sqrt{R} \theta )  ) e^{ \sqrt{R r_j} \upsilon_j + 4 \int_{\eta_j} \sqrt{R} \Re Z_+ (q, z)} & \to \sqrt{-1} e^{-2 \int_{\eta_j}  B_{\L_{\E }}} 
   \label{eq:off_diagonal1} \\
   c ( \eta_j \ast \xi_j \ast \eta_j^{-1} , (\E, \sqrt{R} \theta )  ) e^{-\sqrt{R r_j} \upsilon_j - 4 \int_{\eta_j} \sqrt{R} \Re Z_+ (q, z)} & \to \sqrt{-1} e^{2 \int_{\eta_j} B_{\L_{\E }}}
   \label{eq:off_diagonal2} \\
   d ( \eta_j \ast \xi_j \ast \eta_j^{-1} , (\E, \sqrt{R} \theta )  ) e^{(8 - \upsilon_j) \sqrt{R r_j}} \upsilon_j \sqrt[4]{R r_j} & \to \sqrt{\pi} \notag
 \end{align} 
 with 
 \begin{align*}
     \upsilon_j & = 8 \cos \left( \frac{\arg ( \tau_j )}2 \right),\\
     r_j & = |\tau_j| r_0. 
 \end{align*}
 Given that $B_{\L_{\E }}$ is a $\sqrt{-1}\R$-valued $1$-form, the above limiting values are 
 of length $1$ and $\sqrt{\pi}$ respectively; the phases appearing 
 in~\eqref{eq:off_diagonal1} and~\eqref{eq:off_diagonal2} will play a fundamental role 
 in Section~\ref{sec:proof}.
 
 We now make the observation that by~\eqref{eq:Z_tilde2} and~\eqref{eq:Hopf} for 
 $0 < r_0\ll 1$ we have 
 \begin{align*}
  \int_{\sigma_j} Z_+ (q, z) & = \pm \int_{\sigma_j} \sqrt{\frac{(a z - b)}{\prod_{k=0}^4 (z - t_k)}} \mbox{d} z \\
  & \approx \pm \sqrt{\frac{a t_j - b}{\prod_{0 \leq k \leq 4, k\neq j} (t_j - t_k )}} \int_{t_j}^{t_j + r_0} \frac{\mbox{d} z}{\sqrt{z-t_j}} \\
  & = \pm 2 \sqrt{\tau_j r_0} 
 \end{align*}
 where we have used~\eqref{eq:tauj} in the last line. 
 Notice that we have a freedom of sign in choosing both $Z_+ (q, z)$ and $\sqrt{\tau_j}$. 
 We require that the sign of $Z_+$ is chosen so that the precise form of the above equality be 
 \begin{equation}\label{eq:sign_choice}
    \int_{\sigma_j} Z_+ (q, z) = - 2 \sqrt{\tau_j r_0}. 
 \end{equation}
 We infer that the integral appearing in the exponent of the off-diagonal 
 terms~\eqref{eq:off_diagonal1}--\eqref{eq:off_diagonal2} is of the form 
 \begin{align*}
  \sqrt{R r_j} \upsilon_j + 4 \int_{\eta_j} \sqrt{R} \Re Z_+ (q, z) & = 8 \sqrt{R r_0} \Re \sqrt{\tau_j} + 4 \sqrt{R} \int_{\eta_j}  \Re Z_+ (q, z) \\
  & = 4  \sqrt{R} \left( - \int_{\sigma_j} \Re Z_+ (q, z) + \int_{\eta_j} \Re Z_+ (q, z) \right) \\
  & = 4 \sqrt{R} \int_{x_2}^{t_j} \Re Z_+ (q, z) \\
  & = 4 \sqrt{R} \Re \pi_j, 
 \end{align*}
 which allows us to recast the 
 limits~\eqref{eq:off_diagonal1}--\eqref{eq:off_diagonal2} as 
  \begin{align*}
  b ( \eta_j \ast \xi_j \ast \eta_j^{-1} , (\E, \sqrt{R} \theta )  ) e^{4 \sqrt{R} \Re \pi_j} & 
  \to \sqrt{-1} e^{-2 \int_{\eta_j}  B_{\L_{\E }}}\\
  c ( \eta_j \ast \xi_j \ast \eta_j^{-1} , (\E, \sqrt{R} \theta )  ) e^{-4 \sqrt{R} \Re \pi_j} & 
  \to \sqrt{-1} e^{2 \int_{\eta_j}   B_{\L_{\E }}}.
  \end{align*}

 Now, we have 
 \begin{align}
  & \RH ( \nabla_{\sqrt{R}}) [ \zeta_2 ] \label{eq:product_RH} \\
  & = \RH ( \nabla_{\sqrt{R}}) [ \eta_1 \ast \xi_1 \ast \eta_1^{-1} ] \RH ( \nabla_{\sqrt{R}}) [ \eta_2 \ast \xi_2 \ast \eta_2^{-1} ] .\notag
 \end{align}
 By definition, $l_2 ( \E, \sqrt{R} \theta  )$ is the trace of this matrix, hence we need to compute the diagonal entries of the product~\eqref{eq:product_RH}. 
 
 Its entry $a ( \zeta_2, (\E, \sqrt{R} \theta )  )$ of index $(1,1)$ is the sum 
 $$
  a ( \eta_1 \ast \xi_1 \ast \eta_1^{-1} , (\E, \sqrt{R} \theta )  ) a ( \eta_2 \ast \xi_2 \ast \eta_2^{-1} , (\E, \sqrt{R} \theta )  ) + b ( \eta_1 \ast \xi_1 \ast \eta_1^{-1} , (\E, \sqrt{R} \theta )  ) c ( \eta_2 \ast \xi_2 \ast \eta_2^{-1} , (\E, \sqrt{R} \theta )  ).
 $$
 According to~\eqref{eq:RH_xi} the leading order term of the asymptotic expansion of its first term as $R \to \infty$ is given by 
 \begin{equation}\label{eq:1_1_entry_first_term}
    \pi \frac{e^{- (8 - \upsilon_1) \sqrt{R r_1} - (8 - \upsilon_2) \sqrt{R r_2} }}{\upsilon_1 \upsilon_2 \sqrt[4]{r_1 r_2 R^2}}.
 \end{equation}
 The leading order term of the asymptotic expansion of the second term of $a ( \zeta_2, (\E, \sqrt{R} \theta )  )$ is 
  \begin{equation}
  - \exp \left( 2 \int_{\eta_2 - \eta_1}  B_{\L_{\E }} + 4 \sqrt{R} \Re ( \pi_2 - \pi_1) \right). \label{eq:1_1_entry_second_term}
 \end{equation}
 We again emphasize that this formula gives the polar decomposition of the corresponding 
 term, as $\int_{\eta_2 - \eta_1}  B_{\L_{\E }}$ is purely imaginary and $\pi_j\in\R$. 
 
 The terms of the $(2,2)$-entry $d ( \zeta_2, (\E, \sqrt{R} \theta )  )$ of~\eqref{eq:product_RH} are similar to those of  $a ( \zeta_2, (\E, \sqrt{R} \theta )  )$, up to exchanging the subscripts $j=1$ and $j=2$ of 
 $\upsilon_j, r_j, \eta_j$.  Namely, the term coming from the product of diagonal entries 
 of the factors has leading order term in its asymptotic expansion given 
 by~\eqref{eq:1_1_entry_first_term}, and the leading order term of its other term is 
 \begin{equation}\label{eq:2_2_entry_second_term}
  - \exp \left( 2 \int_{\eta_1 - \eta_2} B_{\L_{\E }} + 4 \sqrt{R} \Re ( \pi_1 - \pi_2) \right).
 \end{equation}
 Notice that the product of~\eqref{eq:1_1_entry_second_term} 
 and~\eqref{eq:2_2_entry_second_term} equals $1$. 
  
 Now, we observe that by~\eqref{eq:upsilon} the coefficient of $\sqrt{R}$ in 
 the exponent of~\eqref{eq:1_1_entry_first_term} is never positive (as it has already 
 been pointed out in the proof of Proposition~\ref{prop:asymptotic_RH}). 
 On the other hand, at least one of the coefficients of $\sqrt{R}$ in the exponent of~\eqref{eq:1_1_entry_second_term} and in the exponent of~\eqref{eq:2_2_entry_second_term} is non-negative. 
 In the extreme case where the coefficients of $\sqrt{R}$ in the exponent of all terms~\eqref{eq:1_1_entry_first_term},~\eqref{eq:1_1_entry_second_term} and~\eqref{eq:2_2_entry_second_term} vanish, 
 then the $\sqrt{R}$ in the denominator of~\eqref{eq:1_1_entry_first_term} guarantees that it is negligible compared to the sum of the other two terms. 
 To sum up, this shows that in the trace, the leading-order term may not 
 be~\eqref{eq:1_1_entry_first_term}, rather it is equal to~\eqref{eq:1_1_entry_second_term} 
 or~\eqref{eq:2_2_entry_second_term} according as $\Re ( \pi_1 - \pi_2) < 0$ or 
 $\Re ( \pi_1 - \pi_2) > 0$, and to the sum of these terms if $\Re ( \pi_1 - \pi_2) = 0$. 
 In any case, the term~\eqref{eq:1_1_entry_first_term} converges to $0$ as $R\to\infty$, 
 and if $\Re ( \pi_1 - \pi_2) > 0$ (respectively, $\Re ( \pi_1 - \pi_2) < 0$) then the 
 same limit holds for the term~\eqref{eq:1_1_entry_second_term} 
 (respectively,~\eqref{eq:2_2_entry_second_term}). 
 Finally, we conclude using that for a ray $\C$ of the form $t e^{\sqrt{-1}\phi_0}$
 with fixed $- \frac{\pi}{2} < \phi_0 < \frac{\pi}{2}$ and variable $t>0$ we have 
 $$
  \lim_{t\to\infty } 2 \cosh ( t e^{\sqrt{-1}\phi_0} ) e^{-t e^{\sqrt{-1}\phi_0} } = 1.
 $$
\end{proof}

Recall Hopf co-ordinates~\eqref{eq:Hopf_coordinates} on $S^3_1$, $\varphi$ being the co-ordinate along the Hopf fibers. 
\begin{prop}\label{prop:critial_angle2}
 Fix $ q \in S^3_1 $ and assume $\pi_1 ( q) \neq \pi_2 ( q)$. Then there exists a unique 
 $\varphi_2 \in [0, 2\pi)$ such that for every $\L_{\E }\in \operatorname{Jac}(X_{q})$ 
 the co-ordinate $l_2 (\E, e^{\sqrt{-1} \varphi_2} \sqrt{R} \theta )$ remains bounded as $R \to \infty$. 
\end{prop}

\begin{proof}
 According to Proposition~\eqref{prop:monodromy_zeta2}, 
 $l_2 (\E, e^{\sqrt{-1} \varphi_2} \sqrt{R} \theta  )$ is bounded as $R \to \infty$ if and only 
 if the equation 
 $$
  \Re ( \pi_1 (e^{\sqrt{-1} \varphi}  q) - \pi_2 (e^{\sqrt{-1} \varphi}  q)) = 0
 $$
 holds for the variable $\varphi \in [0, 2\pi)$. 
 This quantity is the horizontal projection of 
 $$
  \int_{t_2}^{t_1}  Z_+ ( e^{\sqrt{-1} \varphi}  q, z), 
 $$
 where the contour of integration is $\sigma_1 * \eta_1^{-1} * \eta_2 * \sigma_2^{-1}$. 
 Now, taking into account the definition~\eqref{eq:Z_tilde2}, we have 
 $$
   Z_+ ( e^{\sqrt{-1} \varphi}  q, z) =  e^{\sqrt{-1} \varphi/2}  Z_+ (q, z). 
 $$
 Clearly, there exists a unique value $\varphi_2 \in [0, 2\pi)$ satisfying the property that the non-zero complex number $\pi_1 ( q) - \pi_2 ( q)$ 
 multiplied by the unit length complex number $e^{\sqrt{-1} \varphi_2/2}$ has horizontal projection equal to $0$. 
\end{proof}
The results of this subsection have been stated for the case of $l_2 (  \E, \sqrt{R} \theta  )$. 
We now proceed to stating the results analogous to Propositions~\ref{prop:monodromy_zeta2} and~\ref{prop:critial_angle2} for the case of $l_3 (  \E, \sqrt{R} \theta  )$, whose proofs are straightforward modifications of the ones proven so far. 
For $j\in \{ 0, 4 \}$ introduce 
$$
  \pi_j = \pi_j ( q) = \int_{x_4}^{t_j}  Z_+ (q, z) 
$$
along some path in $S_3$. 
\begin{prop}\label{prop:monodromy_zeta3}
Fix $ q \in S^3_1 $. 
\begin{enumerate}
 \item 
  In case $\Re ( \pi_4 ( q) - \pi_0 ( q)) \neq 0$ we have the limit 
  $$
  l_3 (  \E, \sqrt{R} \theta  ) 2 \cosh 
  \left( 2 \int_{\eta_4 - \eta_0}  B_{\L_{\E }} + 4 \sqrt{R} \Re ( \pi_4 - \pi_0)  \right) 
  \to - 1 
  $$
  as $R \to \infty$. 
 \item In case $\Re ( \pi_4 ( q) - \pi_0 ( q))  = 0$ the limit of $l_3 (  \E, \sqrt{R} \theta  )$ as 
 $R \to \infty$ exists and is finite. 
 \item \label{prop:monodromy_zeta3_3}
 If $\pi_4 ( q) \neq \pi_0 ( q)$, then there exists a unique $\varphi_3 \in [0, 2\pi)$ such that for every $\L_{\E }\in \operatorname{Jac}(X_{q})$ the co-ordinate 
 $l_3 (\E, e^{\sqrt{-1} \varphi_3} \sqrt{R} \theta )$ remains bounded as $R \to \infty$. 
\end{enumerate}
\end{prop}

\subsection{Asymptotic behaviour of complex twist co-ordinates}\label{sec:asymptotic_twist}

For $i \in \{ 2, 3 \}$ we let 
$$
    [p_i (\E, \sqrt{R} \theta ) : q_i (\E, \sqrt{R} \theta ) ] 
$$
stand for the complex twist co-ordinates $[p_i : q_i]$ introduced in 
Section~\ref{sec:twist} of the local system $\RH \circ \psi (\E, \sqrt{R} \theta )$. 
In this section we will determine the asymptotic behaviour of these co-ordinates 
as $R \to \infty$. For this purpose, we first determine the asymptotic behaviour 
of the quantities 
$$
  u_i, w_i, A_i, R_i, R'_{i-1}, T_i, U_i, h_i, \psi_i, P_i, Q_i
$$ 
introduced in Section~\ref{sec:twist}. For ease of notation, we will often omit 
to indicate the dependence of the co-ordinates on $(\E, \sqrt{R} \theta )$. 

Before continuing, we add an important tweak to our construction. 
Namely, so far we have assumed that the radii (with respect to Euclidean 
distance) of the circles $\xi_j$ were all equal to each other. 
We now lift this constraint and from now on we merely assume the radii to be 
\emph{pairwise} equal to each other: the radii of $\xi_1, \xi_2$ are equal 
to $s_2 > 0$, however the radius of $\xi_3$ is some possibly different 
real number $s_3 > 0$, while the radii of $\xi_0, \xi_4$ are equal to yet another 
real number $s_4 > 0$. 
From this point on, these numbers take over the role that used to be that of $r_0$ 
up to this point. We will asjust their values according to our needs in 
Section~\ref{sec:Geometry_period}.
The results proven up to now all remain valid, because their assertions 
and proofs did not depend on the actual value of $r_0$. 

In what follows, for $2\times 2$ matrices $A, B$ with non-vanishing entries 
depending on a parameter $R\in\R$ we write $A \approx B$ whenever the limit of each entry of $A$ divided by the corresponding entry of $B$ converges to $1$ as $R \to \infty$. 
Similarly, for two scalar quantities $a,b$ depending on $R\in\R$ we write $a \approx b$ to express that $\frac ab \to 1$ as $R \to \infty$. 
We say $a$ is negligible compared to $b$ if $\frac ab \to 0$ as $R \to \infty$. 

Recall from Section~\ref{ssec:RH} that $c_j^{\pm} = \pm \sqrt{-1}$. 
We start by recording some asymptotic behaviours as $R \to \infty$ ensuing from Sections~\ref{sec:twist} and~\ref{sec:asymptotic_length} in the case $\Re ( \pi_1 - \pi_2) \neq 0$ 
\begin{align}
 l_1 & = c_1^+ + c_1^- = 0 \notag \\
 l_2 & \approx - 2 \cosh \left( 4 \sqrt{R} \Re ( \pi_2 - \pi_1) + 2 \int_{\eta_2 - \eta_1}  B_{\L_{\E }} \right)  \\
 u_2 & = \frac{l_1 - c_2^- l_2}{c_2^+ - c_2^-} \approx  - \cosh \left( 4 \sqrt{R} \Re ( \pi_2 - \pi_1) + 2 \int_{\eta_2 - \eta_1}  B_{\L_{\E }} \right)  \\
 U_2 & \approx \begin{pmatrix}
          1 & 0 \\
         - \cosh \left(  4 \sqrt{R} \Re ( \pi_2 - \pi_1) + 2 \int_{\eta_2 - \eta_1}  B_{\L_{\E }} \right) & 1
         \end{pmatrix}
         \label{eq:U2}
         \\
 A_2 & = \begin{pmatrix}
          c_2^+ & 0 \\
          0 & c_2^-
         \end{pmatrix} 
         = 
         \sqrt{-1} 
         \begin{pmatrix}
          1 & 0 \\
          0 & -1
         \end{pmatrix} 
         \\
 R_2 & \approx \begin{pmatrix}
          - \cosh \left(  4 \sqrt{R} \Re ( \pi_2 - \pi_1) + 2 \int_{\eta_2 - \eta_1}  B_{\L_{\E }} \right) & 1 \\
          * & - \cosh \left(  4 \sqrt{R} \Re ( \pi_2 - \pi_1) + 2 \int_{\eta_2 - \eta_1}  B_{\L_{\E }} \right)
         \end{pmatrix}
         \\
 R'_1 & \approx \begin{pmatrix}
          - \sqrt{-1} \cosh \left(  4 \sqrt{R} \Re ( \pi_2 - \pi_1) + 2 \int_{\eta_2 - \eta_1}  B_{\L_{\E }} \right) & \sqrt{-1} \\
          * & \sqrt{-1} \cosh \left(  4 \sqrt{R} \Re ( \pi_2 - \pi_1) + 2 \int_{\eta_2 - \eta_1}  B_{\L_{\E }} \right)
         \end{pmatrix}
\end{align}
where the entries marked by $*$ may be determined by the (known) determinant of the matrices, but we refrain from spelling them out as they will be irrelevant for our purposes.

\begin{prop}\label{prop:asymptotic_h_i}
For $2\leq i \leq 4$ we have the asymptotic behaviours 
 $$
 h_i^{-1} \approx 
 \begin{pmatrix}
    v_i & e^{-8\sqrt{R s_i} \Re \sqrt{\tau_i}} w_i \\
    - e^{8\sqrt{R s_i} \Re \sqrt{\tau_i}} v_i & w_i
  \end{pmatrix}
$$
and
$$
    h_i \approx \frac 1{2 v_i w_i} 
    \begin{pmatrix}
    w_i & - e^{-8\sqrt{R s_i} \Re \sqrt{\tau_i}} w_i \\
    e^{8\sqrt{R s_i} \Re \sqrt{\tau_i}} v_i & v_i
  \end{pmatrix}
$$
for some $v_i , w_i \in \C$ and the same choice of $\sqrt{\tau_i}$ as in~\eqref{eq:RH_xi}. 
\end{prop}

\begin{proof}
 As the monodromy matrix of $V_i (l_{i-1}, l_i )$ around $\xi_i$ is the diagonal matrix $A_i$, in order to determine $h_i^{-1}$ we need to find the eigenvectors of~\eqref{eq:RH_xi} for $j = i$, taking into account the relation 
 $$
    \sqrt{r_i} \upsilon_i = 8 \Re \sqrt{\tau_i s_i},
 $$
 see~\eqref{eq:upsilon},\eqref{eq:rj_r0}. 
 As we have shown in Proposition~\ref{prop:asymptotic_RH}, the diagonal entries 
 of~\eqref{eq:RH_xi} converge to $0$ and its determinant converges to $1$. 
 A direct computation then shows that $h_i^{-1}$ is of the desired form.  
 We conclude by taking matrix inverse. 
\end{proof}

\begin{prop}\label{prop:asymptotic_h2}
Assume $\Re (\sqrt{\tau_1} - \sqrt{\tau_2}) \neq 0$. 
Then the large scale behaviour as $R \to \infty$ of the isomorphism $h_2^{-1}$ identifying $V_i (l_{i-1}, l_i )$ to the restriction of $V$ to $S_i$ written with respect to the unit diagonalizing 
frame of the local system $V$ is  
$$
 h_2^{-1} \approx 
 \begin{pmatrix}
    \sinh \left( 8 \sqrt{R s_2} \Re (\sqrt{\tau_1} - \sqrt{\tau_2}) \right) & -1 \\
    - e^{8\sqrt{R s_2}  \Re \sqrt{\tau_2}} \sinh \left( 8 \sqrt{R s_2} \Re (\sqrt{\tau_1} - \sqrt{\tau_2}) \right) & - e^{ 8 \sqrt{R s_2} \Re \sqrt{\tau_2}} 
  \end{pmatrix}
$$
\end{prop}

\begin{proof}
By Proposition~\ref{prop:asymptotic_h_i}, we just need to find the values 
of  $v_2,w_2$; for this purpose, we will use the monodromy around the loop $\xi_1$.
Indeed, it is required that the $(1,2)$-entry of 
$h_2 \RH ( \nabla_{\sqrt{R}}) [ \xi_1 ] h_2^{-1}$ be equal to $\sqrt{-1}$. 
After elementary algebra, this is equivalent to the relation 
$$
  \frac{w_2}{v_2} = - \frac{\exp( 8 \sqrt{R s_2} \Re \sqrt{\tau_2})}{\sinh \left( 8 \sqrt{R s_2} \Re (\sqrt{\tau_1} - \sqrt{\tau_2}) \right)}   , 
$$
up to super-exponential error terms in the numerator and denominator. 
The pairs of solutions $[v_2 : w_2 ]\in\CP1$ of this equation are defined 
up to simultaneous rescaling by some element of $\C^{\times}$, and such a 
rescaling  does not affect the complex twist co-ordinates~\eqref{eq:twist}. 
Hence, we may set 
\begin{align*}
 v_2 & = \sinh \left( 8 \sqrt{R s_2} \Re (\sqrt{\tau_1} - \sqrt{\tau_2}) \right) \\
 w_2 & = - \exp( 8 \sqrt{R s_2} \Re \sqrt{\tau_2}) 
\end{align*}
and with these values we find the desired result. 
\end{proof}

\begin{prop}\label{prop:twist}
 Fix  $q \in S^3_1$ such that $\Re (\pi_2 - \pi_1) \neq 0$. 
 Then, the complex twist co-ordinate $[p_2 : q_2]$ associated to $(\E, \sqrt{R} \theta )$ converges to $[0:1]$ as $R\to \infty$ if the conditions  
 \begin{align*}
 \int_{\psi_2} \Re Z_+ & < 2 \Re ( 2 \sqrt{s_2 \tau_2} - \sqrt{s_2 \tau_1} - \sqrt{s_3 \tau_3})  \\
 \vert \Re ( \pi_1 - \pi_2) \vert & = 2 \sqrt{s_2} \Re (\sqrt{\tau_1} - \sqrt{\tau_2})
\end{align*}
hold for one choice of a square root $Z_+$ of $Q$ and 
$$
    \int_{\eta_2 - \eta_1} B_{\L_{\E }} \equiv 0 \pmod{\pi\sqrt{-1}}. 
$$
Specifically, we then have 
$$
    \frac{p_2}{q_2} \approx \exp \left( -8 \sqrt{R} \Re ( 2 \sqrt{s_2 \tau_2} - \sqrt{s_2 \tau_1} - 
  \sqrt{s_3 \tau_3}) + 2 \int_{\psi_2} \left( B_{\L_{\E }} + 2 \sqrt{R} \Re Z_+ \right) \right).
$$
Similarly, $[p_2 : q_2] \to [0:1]$ as $R\to \infty$ if 
\begin{align*}
 \int_{\psi_2} \Re Z_+ & < 2 \Re ( 2 \sqrt{s_2 \tau_2} - \sqrt{s_2 \tau_1} - \sqrt{s_3 \tau_3})  \\
 \vert \Re ( \pi_1 - \pi_2) \vert & = - 2 \sqrt{s_2} \Re (\sqrt{\tau_1} - \sqrt{\tau_2})
\end{align*}
hold for one choice of a square root $Z_+$ of $Q$ and 
$$
    \int_{\eta_2 - \eta_1} B_{\L_{\E }} \equiv \frac{\pi\sqrt{-1}}2 \pmod{\pi\sqrt{-1}}. 
$$
On the other hand, under the condition  
$$
    \vert \Re ( \pi_1 - \pi_2) \vert \neq 2 \sqrt{s_2} \Re (\sqrt{\tau_1} - \sqrt{\tau_2})
$$
$[p_2 : q_2]$ converges to $[1:0]$. 
\end{prop}

\begin{proof} 
 Our task is to compute 
\begin{equation}\label{eq:P2}
   P_2 = h_3 \circ \psi_2 \circ h_2^{-1}, 
\end{equation}
where $\psi_2$ stands for parallel transport map along the path $\psi_2$. 
Since constant factors multiplying $Q_i$ (or $P_i$) have no influence on the definition~\eqref{eq:twist} of the complex twist co-ordinates, from now on we will ignore constant factors; 
said differently, the formulas of the rest of this section hold in $\mbox{PGl}(2,\C)$. 
In particular, the exact value of the constanst $v_3,w_3$ appearing in Proposition~\ref{prop:asymptotic_h_i} will not be relevant, the essential information is that 
their ratio is well-defined. 
Just as in the proof of Proposition~\ref{prop:monodromy_zeta2}, with respect to 
suitable diagonalizing bases we have 
\begin{equation*}
  \psi_2 \approx \begin{pmatrix}
                  e^{\int_{\psi_2} \left( \frac 12 B_{\det (\E )} + B_{\L_{\E }} + \sqrt{R} \Re Z_+ (q, z)\right)} & 0 \\
                  0 & e^{\int_{\psi_2} \left( \frac 12 B_{\det (\E )} - B_{\L_{\E }} - \sqrt{R} \Re Z_+ (q, z)\right)}
                 \end{pmatrix} .
\end{equation*}
We take the frame at $\psi_2 (0) = x_2$ to consist of unit vectors, 
and then the above matrix expresses the action of parallel 
transport with respect to a basis at $\psi_2 (1) = x_3$ whose first and second 
vectors are respectively of length 
$$
  \exp \left(  \pm \sqrt{R} \int_{x_2}^{x_3} \Re Z_+ (q, z) \right). 
$$
It follows that the action of parallel transport along $\psi_2$, written in unit-length diagonalizing bases both at $x_2, x_3$, is described by the matrix 
$$
  \begin{pmatrix}
   e^{\int_{\psi_2} \left(\frac 12 B_{\det (\E )} + B_{\L_{\E }} + 2 \sqrt{R} \Re Z_+ (q, z)\right)}  & 0 \\
   0 & e^{\int_{\psi_2} \left(\frac 12 B_{\det (\E )} - B_{\L_{\E }} - 2 \sqrt{R} \Re Z_+ (q, z)\right)}
  \end{pmatrix}  .
$$
For ease of notation, from now on we will drop the argument of $Z_+$ and the term 
$\frac 12 B_{\det (\E )}$ in the argument of the exponential 
(remember that we work in $\mbox{PGl}(2,\C)$). 

Now, by Proposition~\ref{prop:asymptotic_h2} the growth orders of the entries of 
$\psi_2 h_2^{-1}$ are given by  
\begin{equation*}
    \begin{pmatrix}
    \sinh \left( 8 \sqrt{R s_2} \Re (\sqrt{\tau_1} - \sqrt{\tau_2}) \right) 
    e^{\int_{\psi_2} \left( B_{\L_{\E }} + 2 \sqrt{R} \Re Z_+ \right)}
    & - e^{\int_{\psi_2} \left( B_{\L_{\E }} + 2 \sqrt{R} \Re Z_+ \right)} \\
    - e^{8\sqrt{R s_2} \Re \sqrt{\tau_2} - \int_{\psi_2} \left( B_{\L_{\E }} + 2 \sqrt{R} \Re Z_+ \right)} \sinh \left( 8 \sqrt{R s_2} \Re (\sqrt{\tau_1} - \sqrt{\tau_2}) \right) \\
    & - e^{8 \sqrt{R s_2} \Re \sqrt{\tau_2} -\int_{\psi_2} \left( B_{\L_{\E }} + 2 \sqrt{R} \Re Z_+ \right)}
    \end{pmatrix}.
\end{equation*}
Using Proposition~\ref{prop:asymptotic_h_i}, the $(1,1)$-entry of~\eqref{eq:P2} is then of the form 
\begin{align}
  & \frac 1{2 v_3} e^{\int_{\psi_2} \left( B_{\L_{\E }} + 2 \sqrt{R} \Re Z_+ \right)} 
  \sinh \left( 8 \sqrt{R s_2} \Re (\sqrt{\tau_1} - \sqrt{\tau_2}) \right) \label{eq:P2(1,1)_1} \\
  & + \frac 1{2 v_3} e^{8\sqrt{R} \Re ( \sqrt{s_2 \tau_2} - \sqrt{s_3 \tau_3} )  -\int_{\psi_2} \left( B_{\L_{\E }} + 2 \sqrt{R} \Re Z_+ \right)} 
  \sinh \left( 8 \sqrt{R s_2} \Re (\sqrt{\tau_1} - \sqrt{\tau_2}) \right), \label{eq:P2(1,1)_2}
\end{align}
and its $(1,2)$-entry is 
\begin{equation}
  - \frac 1{2 v_3} e^{\int_{\psi_2} \left( B_{\L_{\E }} + 2 \sqrt{R} \Re Z_+ \right)}  
  + \frac 1{2 v_3}  e^{8 \sqrt{R} \Re ( \sqrt{ s_2 \tau_2} - \sqrt{s_3 \tau_3} ) - 
  \int_{\psi_2} \left( B_{\L_{\E }} + 2 \sqrt{R} \Re Z_+ \right)}. \label{eq:P2(1,2)}
\end{equation}

We now turn to computing the ratio of the entries of the first row of the matrix 
\begin{equation}\label{eq:Q2}
 Q_2 = A_3^{-\frac 12} U_3 P_2 U_2^{-1}. 
\end{equation}
Since $A_3^{-\frac 12} U_3$ is lower triangular, left multiplication by this matrix does not affect the quotient of the entries in the first row, so we may ignore this factor. 
On the other hand, if $\Re ( \pi_1 - \pi_2) \neq 0$ then we have by~\eqref{eq:U2} 
$$
  U_2^{-1} =  \begin{pmatrix}
          1 & 0 \\
          - u_2 & 1
         \end{pmatrix}
         \approx 
         \begin{pmatrix}
          1 & 0 \\
          \cosh \left( 4 \sqrt{R} \Re ( \pi_2 - \pi_1) + 2 \int_{\eta_2 - \eta_1}  B_{\L_{\E }} \right) & 1
         \end{pmatrix}. 
$$
The $(1,1)$-entry of~\eqref{eq:Q2} is of the form 
\begin{align}
  p_2 = & \frac 1{2 v_3} e^{\int_{\psi_2} \left( B_{\L_{\E }} + 2 \sqrt{R} \Re Z_+ \right)} 
  \sinh \left( 8 \sqrt{R s_2} \Re (\sqrt{\tau_1} - \sqrt{\tau_2}) \right) \label{eq:Q2(1,1)_1} \\
  & + \frac 1{2 v_3} e^{8\sqrt{R} \Re ( \sqrt{ s_2 \tau_2} - \sqrt{ s_3 \tau_3} )  - \int_{\psi_2} \left( B_{\L_{\E }} + 2 \sqrt{R} \Re Z_+ \right)} 
  \sinh \left( 8 \sqrt{R s_2} \Re (\sqrt{\tau_1} - \sqrt{\tau_2}) \right) \label{eq:Q2(1,1)_2}\\
  & + \frac 1{2 v_3} 
  \cosh \left( 4 \sqrt{R} \Re ( \pi_2 - \pi_1) + 2 \int_{\eta_2 - \eta_1}  B_{\L_{\E }} \right)
  e^{\int_{\psi_2} \left( B_{\L_{\E }} + 2 \sqrt{R} \Re Z_+ \right)} \label{eq:Q2(1,1)_3} \\
  & - \frac 1{2 v_3} 
  \cosh \left( 4 \sqrt{R} \Re ( \pi_2 - \pi_1) + 2 \int_{\eta_2 - \eta_1}  B_{\L_{\E }} \right)
  e^{8 \sqrt{R} \Re ( \sqrt{s_2 \tau_2} - \sqrt{s_3 \tau_3} ) - 
  \int_{\psi_2} \left( B_{\L_{\E }} + 2 \sqrt{R} \Re Z_+ \right)}, \label{eq:Q2(1,1)_4} 
\end{align}
and its $(1,2)$-entry $q_2$ agrees with~\eqref{eq:P2(1,2)}. 
We now analyze cases depending on the possible relationships of the coefficients of 
$\sqrt{R}$ in the arguments. Fix some $0 < \varepsilon < \frac{\pi}2$ and let 
$c\in\C$ satisfy 
$$
    \operatorname{Arg} (\pm c) \in \left[ -\frac{\pi}2 + \varepsilon 
    , \frac{\pi}2 - \varepsilon \right].
$$
We will then use that as $| \Re c | \to \infty$ we have 
\begin{align*}
     2 \cosh (c) & \approx e^{\vert \Re c \vert} (\cos (\Im c) + \sqrt{-1} \sgn(\Re c) 
     \sin (\Im c)) \\
     2 \sinh (c) & \approx e^{\vert \Re c \vert} (\sgn(\Re c) \cos (\Im c) + \sqrt{-1} 
     \sin (\Im c)), 
\end{align*}
where $\approx$ is used to denote expressions having the same asymptotic expansion 
as $R\to \infty$. 

\begin{enumerate}
 \item \label{case:1} If the inequalities 
\begin{align*}
 \int_{\psi_2} \Re Z_+ & < 2 \Re ( \sqrt{s_2 \tau_2}- \sqrt{s_3 \tau_3} ) \\
 2 \sqrt{s_2} | \Re (\sqrt{\tau_1} - \sqrt{\tau_2})| & < \vert \Re ( \pi_1 - \pi_2) \vert
\end{align*}
hold then the second term of~\eqref{eq:P2(1,2)} dominates its first term in absolute value, and~\eqref{eq:Q2(1,1)_4} dominates~\eqref{eq:Q2(1,1)_1},~\eqref{eq:Q2(1,1)_2} and~\eqref{eq:Q2(1,1)_3}. In this case, we have 
\begin{align*}
 [p_2 : q_2] & \approx \left[ 
 \cosh \left( 4 \sqrt{R} \Re ( \pi_2 - \pi_1) + 2 \int_{\eta_2 - \eta_1}  B_{\L_{\E }} \right) : - 1 \right] \\ 
  & \approx \left[ e^{4 \sqrt{R} | \Re ( \pi_2 - \pi_1)|} 
  \left( \cos \left( \frac 2{\sqrt{-1}} \int_{\eta_2 - \eta_1}  B_{\L_{\E }} \right) 
  + \sqrt{-1} \sgn (\Re ( \pi_2 - \pi_1)) \sin \left( \frac 2{\sqrt{-1}} \int_{\eta_2 - \eta_1}  B_{\L_{\E }} \right)
  \right) : - 2 \right] \\ 
 & \to [1 : 0] . 
\end{align*}
\item \label{case:2}
In case the inequalities 
\begin{align*}
 \int_{\psi_2} \Re Z_+ & < 2 \Re ( \sqrt{s_2 \tau_2}- \sqrt{s_3 \tau_3} ) \\
  \vert \Re ( \pi_1 - \pi_2) \vert & < 2 \sqrt{s_2} | \Re (\sqrt{\tau_1} - \sqrt{\tau_2})| ,
\end{align*}
hold then the second term of~\eqref{eq:P2(1,2)} dominates its first term in absolute value, 
and the dominant term of $p_2$ is~\eqref{eq:Q2(1,1)_2}. In this case, we get 
$$
    [p_2 : q_2] \approx \left[ 
    \sinh \left( 8 \sqrt{R s_2} \Re (\sqrt{\tau_1} - \sqrt{\tau_2}) \right) : 1 \right] 
    \to [1 : 0] . 
$$
\item \label{case:3} In the case 
\begin{align*}
 2 \Re ( \sqrt{s_2 \tau_2}- \sqrt{s_3 \tau_3} ) & < \int_{\psi_2} \Re Z_+ \\
 2 \sqrt{s_2} | \Re (\sqrt{\tau_1} - \sqrt{\tau_2})| & < \vert \Re ( \pi_1 - \pi_2) \vert
\end{align*}
the first term of~\eqref{eq:P2(1,2)} dominates its second term in absolute value, 
and the dominant term of $p_2$ is~\eqref{eq:Q2(1,1)_3}. In this case, we get 
\begin{align*}
 [p_2 : q_2] & \approx - u_2 \\
 & \approx \left[ 
 \cosh \left( 4 \sqrt{R} \Re ( \pi_2 - \pi_1) + 2 \int_{\eta_2 - \eta_1}  B_{\L_{\E }} \right)
 : - 1 \right] \\ 
 & \approx \left[ e^{4 \sqrt{R} | \Re ( \pi_2 - \pi_1)|} 
  \left( \cos \left( \frac 2{\sqrt{-1}} \int_{\eta_2 - \eta_1}  B_{\L_{\E }} \right) 
  + \sqrt{-1} \sgn (\Re ( \pi_2 - \pi_1)) \sin \left( \frac 2{\sqrt{-1}} \int_{\eta_2 - \eta_1}  B_{\L_{\E }} \right)
  \right) : - 2 \right] \\ 
 & \to [1 : 0] . 
\end{align*}
\item \label{case:4} In the case 
\begin{align*}
 2 \Re ( \sqrt{s_2 \tau_2}- \sqrt{s_3 \tau_3} ) & < \int_{\psi_2} \Re Z_+ \\
 \vert \Re ( \pi_1 - \pi_2) \vert & < 2 \sqrt{s_2} | \Re (\sqrt{\tau_1} - \sqrt{\tau_2})|
\end{align*}
the first term of~\eqref{eq:P2(1,2)} dominates its second term in absolute value, 
and the dominant term of $p_2$ is~\eqref{eq:Q2(1,1)_1}, so we have 
$$
 [p_2 : q_2] \approx \left[ \sinh \left( 8 \sqrt{R s_2} \Re (\sqrt{\tau_1} - \sqrt{\tau_2}) \right) : 1 \right] \to [1 : 0] . 
$$
\item \label{case:5} In the case 
\begin{align}
 \int_{\psi_2} \Re Z_+ & < 2 \Re ( \sqrt{s_2 \tau_2}- \sqrt{s_3 \tau_3} )
 \label{eq:inequality1} \\
 \vert \Re ( \pi_1 - \pi_2) \vert & = 2 \sqrt{s_2} \Re (\sqrt{\tau_1} - \sqrt{\tau_2}) \label{eq:inequality2} \\
 \int_{\eta_2 - \eta_1} B_{\L_{\E }} & \equiv 0 \pmod{\pi\sqrt{-1}} 
\end{align}
the dominant term of $q_2$ is the second term in~\eqref{eq:P2(1,2)}, 
and the terms~\eqref{eq:Q2(1,1)_2} and~\eqref{eq:Q2(1,1)_4} cancel each other 
up to terms of order 
$$
    o( e^{8\sqrt{R} \Re ( \sqrt{s_2 \tau_2}- \sqrt{s_3 \tau_3} ) - 2\int_{\psi_2} \sqrt{R} \Re Z_+}  )
$$
(see Proposition~\ref{prop:monodromy_zeta2}), 
while the terms~\eqref{eq:Q2(1,1)_1} and~\eqref{eq:Q2(1,1)_3} are of the same order of growth, 
and add up to the double of each of them. We then have 
$$
 \frac{p_2}{q_2} \approx  \frac
 {\exp \left( 4\sqrt{R} \vert \Re ( \pi_1 - \pi_2) \vert + \int_{\psi_2} \left( B_{\L_{\E }} + 2 \sqrt{R} \Re Z_+ \right)\right)}{\exp\left( 8\sqrt{R} \Re ( \sqrt{s_2 \tau_2}- \sqrt{s_3 \tau_3} ) - \int_{\psi_2} \left( B_{\L_{\E }} + 2 \sqrt{R} \Re Z_+ \right)\right)} + o (1).   
$$
Now, assuming the inequality 
$$
    2 \Re ( \sqrt{s_2 \tau_2}- \sqrt{s_3 \tau_3} ) - \int_{\psi_2} \Re Z_+ > \vert \Re ( \pi_1 - \pi_2) \vert = 2 \sqrt{s_2} \Re (\sqrt{\tau_1} - \sqrt{\tau_2}),
$$
or equivalently, 
\begin{equation}\label{eq:inequality3}
  \int_{\psi_2} \Re Z_+ <  2 \Re ( 2 \sqrt{s_2 \tau_2} - \sqrt{s_2 \tau_1} - 
  \sqrt{s_3 \tau_3}) , 
\end{equation}
then we get 
\begin{align*}
     [p_2 : q_2] & \approx \left[ 
     \exp \left( -8 \sqrt{R} \Re ( 2 \sqrt{s_2 \tau_2} - \sqrt{s_2 \tau_1} - 
  \sqrt{s_3 \tau_3}) + 2 \int_{\psi_2} \left( B_{\L_{\E }} + 2 \sqrt{R} \Re Z_+ \right) \right) : 1 \right] \\
     & \to [0:1]
\end{align*}
Notice that in view of the assumption  $\Re ( \pi_1 - \pi_2) \neq 0$
and~\eqref{eq:inequality2}, the inequality~\eqref{eq:inequality3} is stronger than~\eqref{eq:inequality1}. 
\item \label{case:6} 
In the case 
\begin{align*}
 \int_{\psi_2} \Re Z_+ & < 2 \Re ( \sqrt{s_2 \tau_2}- \sqrt{s_3 \tau_3} )\\
 \vert \Re ( \pi_1 - \pi_2) \vert & = 2 \sqrt{s_2} \Re (\sqrt{\tau_1} - \sqrt{\tau_2})\\
 \int_{\eta_2 - \eta_1} B_{\L_{\E }} & \not{\!\!\equiv} 0 \pmod{\pi\sqrt{-1}} 
\end{align*}
the cancellation of case~\eqref{case:5} does not occur and we have 
$$
    [p_2 : q_2] \approx - u_2 \to [1:0]. 
$$
\item \label{case:7} In the case 
\begin{align*}
 \int_{\psi_2} \Re Z_+ & < 2 \Re ( \sqrt{s_2 \tau_2}- \sqrt{s_3 \tau_3} )
 \\
 \vert \Re ( \pi_1 - \pi_2) \vert & = - 2 \sqrt{s_2} \Re (\sqrt{\tau_1} - \sqrt{\tau_2}) 
 \\
 \int_{\eta_2 - \eta_1} B_{\L_{\E }} & \equiv \frac{\pi\sqrt{-1}}2 \pmod{\pi\sqrt{-1}} 
\end{align*}
the analysis is similar to case~\eqref{case:5} up to sign differences. 
Namely, the dominant term of $q_2$ is the second term in~\eqref{eq:P2(1,2)}, 
but now we have 
$$
    \exp \left( 2 \int_{\eta_2 - \eta_1} B_{\L_{\E }} \right) = -1.
$$ 
The terms~\eqref{eq:Q2(1,1)_2} and~\eqref{eq:Q2(1,1)_4} again cancel each other 
and~\eqref{eq:Q2(1,1)_1},~\eqref{eq:Q2(1,1)_3} are asymptotically equal to each other,
hence the behaviour of $[p_2 : q_2]$ agrees (up to a global sign) with the one given in
case~\eqref{case:5}. 
Namely, under condition~\eqref{eq:inequality3} we get $[p_2 : q_2]\to [0:1]$ 
with the same exponential decay rate as in case~\eqref{case:5}. 
\item \label{case:8} In the case 
\begin{align*}
 \int_{\psi_2} \Re Z_+ & < 2 \Re ( \sqrt{s_2 \tau_2}- \sqrt{s_3 \tau_3} ) \\
 \vert \Re ( \pi_1 - \pi_2) \vert & = -2 \sqrt{s_2} \Re (\sqrt{\tau_1} - \sqrt{\tau_2})\\
 \int_{\eta_2 - \eta_1} B_{\L_{\E }}  & \not{\!\!\equiv} \frac{\pi\sqrt{-1}}2  \pmod{\pi\sqrt{-1}} 
\end{align*}
the dominant term of $q_2$ is the second term of~\eqref{eq:P2(1,2)}, 
while the terms~\eqref{eq:Q2(1,1)_2} and~\eqref{eq:Q2(1,1)_4} are of the same order 
of growth and dominate~\eqref{eq:Q2(1,1)_1} and~\eqref{eq:Q2(1,1)_3}.
Fixing an arbitrary $\varepsilon > 0$, for all $q$ such that 
$$ 
   \left\vert 1 + \exp \left( 2 \int_{\eta_2 - \eta_1} B_{\L_{\E }} \right) \right\vert
   > \varepsilon ,
$$
we get 
$$
 \frac{p_2}{q_2} \approx \left[ 
 \exp \left( 4 \sqrt{R} | \Re ( \pi_2 - \pi_1)|\right)  
 \left( - 1 - \exp \left( 2 \int_{\eta_2 - \eta_1}  B_{\L_{\E }} \right) \right) : 1 \right] 
 \to [1 : 0] . 
$$
\end{enumerate}
When 
\begin{align*}
  2 \Re ( \sqrt{s_2 \tau_2}- \sqrt{s_3 \tau_3} ) & < \int_{\psi_2} \Re Z_+ \\
 \vert \Re ( \pi_1 - \pi_2) \vert & =\pm 2\sqrt{s_2} \Re (\sqrt{\tau_1} - \sqrt{\tau_2}), 
\end{align*}
then the analysis is similar to cases~\eqref{case:5}--\eqref{case:8}, 
depending on whether or not 
$$ 
\int_{\eta_2 - \eta_1} B_{\L_{\E }}  \equiv 0 \pmod{\pi\sqrt{-1}} 
$$
(in case the second condition holds with a $+$ sign), and on whether or not 
$$
\int_{\eta_2 - \eta_1} B_{\L_{\E }}  \equiv \frac{\pi\sqrt{-1}}2  \pmod{\pi\sqrt{-1}} 
$$
(in case the second condition holds with a $-$ sign). 
\end{proof}

\begin{remark}\label{rem:change_sign}
 Notice that some of the inequalities appearing in the above cases depend on the choice of square root $Z_+$, see~\eqref{eq:Z_tilde2}. 
 Changing sign of $Z_+$  results in the following transformations: 
 \begin{align*}
  \pi_j & \mapsto - \pi_j \\
  \vert & \Re ( \pi_1 - \pi_2) \vert \mapsto  \vert \Re ( \pi_2 - \pi_1) \vert =  \vert \Re ( \pi_1 - \pi_2) \vert \\
  \tau_j & \mapsto - \tau_j \\
  \int_{\psi_2} \Re Z_+ & \mapsto - \int_{\psi_2} \Re Z_+ .
 \end{align*}
 It is immediate to see for instance that these transformations pairwise interchange cases~\eqref{case:1} with~\eqref{case:3} and~\eqref{case:2} with~\eqref{case:4}. 
 Importantly however, the asymptotic behaviours we have found in these 
 corresponding cases agree with each other, because $\cosh$ is an even function 
 and $\sinh$ is odd. 
 As a matter of fact, the cases that we have not completely worked out in the 
 above argument arise via these changes of signs from the 
 cases~\eqref{case:5},~\eqref{case:6},~\eqref{case:7} and~\eqref{case:8} 
 in a similar way than~\eqref{case:3} arises from~\eqref{case:1}. 
 In particular, doing the computations in each of the missig cases, we would 
 find that the asymptotic formulas for $[p_2:q_2]$ agree with the ones that we 
 have already found. For us the case~\eqref{case:5} will be of particular interest.
\end{remark}

We will also need an analogous statement to Proposition~\ref{prop:twist} 
for the complex twist co-ordinate $[p_3 : q_3]$. 
\begin{prop}\label{prop:twist3}
 Fix  $q \in S^3_1$ such that $\Re (\pi_4 - \pi_0) \neq 0$. 
 Then, the complex twist co-ordinate $[p_3 : q_3]$ associated to $(\E, \sqrt{R} \theta )$ 
 converges to $[0:1]$ as $R\to \infty$ if the conditions  
 \begin{align}
 \int_{\psi_3} \Re Z_+ & > 2 \Re ( \sqrt{s_3 \tau_3} + \sqrt{s_4 \tau_0} -
 2 \sqrt{s_4 \tau_4}  ) \label{eq:inequality4} \\
 | \Re (\pi_0 - \pi_4)| & = 2 \sqrt{s_4} \Re (\sqrt{\tau_0} - \sqrt{\tau_4}) \notag \\
 \int_{\eta_0 - \eta_4} B_{\L_{\E }} & \equiv 0 \pmod{\pi\sqrt{-1}} \notag
\end{align}
are met; specifically, we then have 
$$
    \frac{p_3}{q_3} \approx \exp 
    \left( - 2 \int_{\psi_3} ( B_{\L_{\E }} + 2 \sqrt{R} \Re Z_+) 
    +  8 \sqrt{R} \Re ( \sqrt{s_3 \tau_3} + \sqrt{s_4 \tau_0} - 2 \sqrt{s_4 \tau_4} )  \right).
$$
Meanwhile, $[p_3 : q_3]\to [1:0]$ if 
$$
    | \Re (\pi_0 - \pi_4)| \neq 2 \sqrt{s_4} | \Re (\sqrt{\tau_0} - \sqrt{\tau_4})|.
$$
\end{prop}

\begin{proof}
 This is similar to Proposition~\ref{prop:twist}, yet we will give details for the 
 sake of completeness. 
 We start by determining $[v_4 : w_4]$ along the lines of 
 Proposition~\ref{prop:asymptotic_h2}. Namely, from the assumption that 
 the $(1,2)$-entry of the monodromy of $V_4$ around $\rho_4 = \xi_0$ is equal 
 to $1$, knowing that the radius of this circle is $s_4$, we find 
 \begin{align*}
 v_4 & = \sinh \left( 8 \sqrt{R s_4} \Re (\sqrt{\tau_0} - \sqrt{\tau_4}) \right) \\
 w_4 & = \sqrt{-1} \exp( 8 \sqrt{R s_4} \Re \sqrt{\tau_4}) .
\end{align*}
The $(1,2)$-entry of $P_3 = h_4 \circ \psi_3 \circ h_3^{-1}$ is 
\begin{equation}\label{eq:P3(1,2)_1} 
 \frac{w_3}2 \frac{e^{\int_{\psi_3} ( B_{\L_{\E }} + 2 \sqrt{R} \Re Z_+) - 8 \sqrt{R s_3} \Re \sqrt{\tau_3}}}
 {\sinh \left( 8 \sqrt{R s_4} \Re (\sqrt{\tau_0} - \sqrt{\tau_4}) \right)} - \frac{w_3}2 \frac{e^{- 8\sqrt{R s_4} \Re \sqrt{\tau_4}  - 
 \int_{\psi_3} ( B_{\L_{\E }} + 2 \sqrt{R} \Re Z_+)}}
  {\sinh \left( 8 \sqrt{R s_4} \Re (\sqrt{\tau_0} - \sqrt{\tau_4}) \right)} 
\end{equation}
Furthermore, the $(2,2)$-entry of $P_3 = h_4 \circ \psi_3 \circ h_3^{-1}$ is equal to 
\begin{equation}\label{eq:P3(2,2)}
    - \sqrt{-1} \frac{w_3}2 \left( e^{- 8\sqrt{R} \Re \sqrt{s_3 \tau_3} 
    + \int_{\psi_3} ( B_{\L_{\E }} + 2 \sqrt{R} \Re Z_+)} 
    + e^{- 8\sqrt{R} \Re \sqrt{s_4 \tau_4}  - 
    \int_{\psi_3} ( B_{\L_{\E }} + 2 \sqrt{R} \Re Z_+)} \right).
\end{equation}

Given that $U_3$ is lower triangular with $(2,2)$-entry equal to $1$, the 
second column of $Q_3 = A_4^{-\frac 12} U_3 P_3 U_3^{-1}$ is the same as 
that of $A_4^{-\frac 12} U_3 P_3$. Therefore, in computing the second 
column of $Q_3$ we may ignore the effect of multiplication by $U_3^{-1}$ from 
the right. Moreover, taking into account $c_4^{\pm} = \pm \sqrt{-1}$ 
we have $l_4 = 0$ and $u_4 = \frac{l_3}{2 \sqrt{-1}}$. We also notice the 
relation 
\begin{equation}\label{eq:c4}
     \sqrt{-1} \sqrt{c_4^-} = \sqrt{c_4^+}. 
\end{equation}
Combining these, we find 
\begin{equation*}
    A_4^{-\frac 12} U_3 = \begin{pmatrix}
                           \sqrt{c_4^-} & 0 \\
                           \sqrt{c_4^+} \frac{l_3}{2 \sqrt{-1}} & \sqrt{c_4^+}
                        \end{pmatrix}
                        = \begin{pmatrix}
                           \sqrt{c_4^-} & 0 \\
                           \sqrt{c_4^-} \frac{l_3}{2} & \sqrt{c_4^+}
                        \end{pmatrix}
\end{equation*}
We deduce that the $(1,2)$-entry of $A_4^{-\frac 12} U_3 P_3$ reads as 
\begin{align}\label{eq:Q3(1,2)_1} 
 q_ 3 = & \pm \sqrt{c_4^-} w_3 e^{- 8 \sqrt{R s_4} | \Re (\sqrt{\tau_0} - \sqrt{\tau_4})| - 8 \sqrt{R s_3} \Re \sqrt{\tau_3} + 
  \int_{\psi_3} ( B_{\L_{\E }} + 2 \sqrt{R} \Re Z_+) } 
  \\ \label{eq:Q3(1,2)_2} 
  & \mp \sqrt{c_4^-} w_3 e^{- 8 \sqrt{R s_4} | \Re (\sqrt{\tau_0} - \sqrt{\tau_4})| - 
  8\sqrt{R s_4} \Re \sqrt{\tau_4}  - \int_{\psi_3} ( B_{\L_{\E }} + 2 \sqrt{R} \Re Z_+)}. 
\end{align}
The $(2,2)$-entry of $A_4^{-\frac 12} U_3 P_3$ is 
\begin{align}\notag
  p_3 + l_3 q_3 = &  \frac{\sqrt{c_4^+} w_3}{2 \sqrt{-1}} e^{- 8 \sqrt{R s_4} | \Re (\sqrt{\tau_0} - \sqrt{\tau_4})| + 
  \int_{\psi_3} ( B_{\L_{\E }} + 2 \sqrt{R} \Re Z_+) - 8 \sqrt{R s_3} \Re \sqrt{\tau_3} 
  + 4 \sqrt{R} \Re | \pi_4 - \pi_0 |} \\ \notag 
   & + \frac{\sqrt{c_4^+} w_3}{2 \sqrt{-1}}
  e^{- 8 \sqrt{R s_4} | \Re (\sqrt{\tau_0} - \sqrt{\tau_4})| - 8\sqrt{R s_4} 
  \Re \sqrt{\tau_4}  - \int_{\psi_3} ( B_{\L_{\E }} + 2 \sqrt{R} \Re Z_+)
  + 4 \sqrt{R} \Re | \pi_4 - \pi_0 |} \\ \notag 
  & -\sqrt{-1} \frac{\sqrt{c_4^+} w_3}2 
  e^{-8\sqrt{R s_3} \Re \sqrt{\tau_3} + \int_{\psi_3} ( B_{\L_{\E }} + 2 \sqrt{R} \Re Z_+)}\\
     & -\sqrt{-1} \frac{\sqrt{c_4^+} w_3}2 
  e^{-8\sqrt{R s_4} \Re \sqrt{\tau_4}  - \int_{\psi_3} ( B_{\L_{\E }} + 2 \sqrt{R} \Re Z_+)}.  \label{eq:Q3(2,2)_3}
\end{align}
We now consider the expression 
\begin{equation}\label{eq:p3t3q3}
     p_3 = ( p_3 + l_3 q_3 ) - l_3 q_3. 
\end{equation}
Just as in the proof of Proposition~\ref{prop:twist}, in case 
$$
    | \Re (\pi_0 - \pi_4)| \neq 2 \sqrt{s_4}| \Re (\sqrt{\tau_0} - \sqrt{\tau_4})| 
$$
or 
$$
    \int_{\eta_0 - \eta_4} B_{\L_{\E }} \not{\!\!\equiv} 0 \quad \mbox{or} \quad 
    \frac{\pi\sqrt{-1}}2  \pmod{\pi\sqrt{-1}} 
$$
neither of the terms~\eqref{eq:Q3(1,2)_1} and~\eqref{eq:Q3(1,2)_2} 
cancels any term of~\eqref{eq:Q3(2,2)_3}, and (depending on certain inequalities 
similar to those in the proof of Proposition~\ref{prop:twist}) we get 
$$
    [p_3 : q_3 ] \approx \left[ \cosh \left( 4\sqrt{R} \Re ( \pi_0 - \pi_4) 
    + 2 \int_{\eta_0 - \eta_4} B_{\L_{\E }} \right) : 1 \right] 
    \to [1 : 0] 
$$
or 
$$
    [p_3 : q_3 ] \approx \left[ e^{8 \sqrt{R s_4} | \Re (\sqrt{\tau_0} - \sqrt{\tau_4})|}
    : 1 \right] \to [1 : 0] 
$$
as $R\to\infty$. 

For the sake of concreteness let us now assume 
\begin{equation}\label{eq:tau0tau4positive}
   0 < | \Re (\pi_0 - \pi_4)| = 2 \sqrt{s_4} \Re (\sqrt{\tau_0} - \sqrt{\tau_4}).
\end{equation}
Then the signs in~\eqref{eq:Q3(1,2)_1} and~\eqref{eq:Q3(1,2)_2} are $+$ 
and $-$ respectively; the analysis in the case of opposite sign is completely analogous. 
In this case, we immediately see from~\eqref{eq:Q3(2,2)_3} that 
\begin{align}\label{eq:p3t3q3_expanded1}
 p_3 + l_3 q_3 = & -\sqrt{-1} \sqrt{c_4^+} w_3 
  e^{-8\sqrt{R s_3} \Re \sqrt{\tau_3} + \int_{\psi_3} ( B_{\L_{\E }} + 2 \sqrt{R} \Re Z_+)}\\
   & -\sqrt{-1} \sqrt{c_4^+} w_3 e^{- 8\sqrt{R s_4} \Re \sqrt{\tau_4}  - 
    \int_{\psi_3} ( B_{\L_{\E }} + 2 \sqrt{R} \Re Z_+)}.\label{eq:p3t3q3_expanded2}
\end{align}
Now, in view of~\eqref{eq:c4}, if 
$$
    \int_{\eta_0 - \eta_4} B_{\L_{\E }} \equiv 0 \pmod{\pi\sqrt{-1}} 
$$
then on the right-hand side of~\eqref{eq:p3t3q3} $l_3$ 
times the term~\eqref{eq:Q3(1,2)_2} cancels~\eqref{eq:p3t3q3_expanded2}, 
and $l_3$ times the term~\eqref{eq:Q3(1,2)_2} is equal 
to~\eqref{eq:p3t3q3_expanded1}. In sum, we then find 
\begin{equation}\label{eq:p3}
  p_3 =  2 \frac{\sqrt{c_4^+} w_3}{\sqrt{-1}} 
  e^{- 8\sqrt{R s_4} \Re \sqrt{\tau_4}  - 
  \int_{\psi_3} ( B_{\L_{\E }} + 2 \sqrt{R} \Re Z_+) }. 
\end{equation}
There are two cases to consider: first, when the leading-order term 
of $q_3$ is~\eqref{eq:Q3(1,2)_1} and second, when it is~\eqref{eq:Q3(1,2)_2}. 
The first of these possibilities holds if 
$$
    - 8 \Re \sqrt{s_3 \tau_3} + 2 \int_{\psi_3} \Re Z_+ 
$$
is larger than 
$$
    - 8 \Re \sqrt{s_4 \tau_4}  - 2 \int_{\psi_3} \Re Z_+ ,
$$
while the second possibility holds for the reverse relation. 
These relations may be equivalently expressed as 
\begin{equation}\label{eq:inequality_case1}
     \int_{\psi_3} \Re Z_+ > 2 \Re (\sqrt{s_3 \tau_3} - \sqrt{s_4 \tau_4} )
\end{equation}
and
$$
    \int_{\psi_3} \Re Z_+ < 2 \Re (\sqrt{s_3 \tau_3} - \sqrt{s_4 \tau_4} ).
$$
Let us assume~\eqref{eq:inequality_case1}; in this case 
using~\eqref{eq:tau0tau4positive} we have 
$$
    \frac{p_3}{q_3} \approx 2 \exp \left( 8 \sqrt{R} 
    \Re ( \sqrt{s_3 \tau_3} + \sqrt{s_4 \tau_0} - 2 \sqrt{s_4 \tau_4} ) 
    - \int_{\psi_3} ( B_{\L_{\E }} + 2 \sqrt{R} \Re Z_+)
    \right)  .  
$$
Now, the expression on the right-hand side converges to $0$ under the 
assumption~\eqref{eq:inequality4}. 
We finish by observing that because of the assumption~\eqref{eq:tau0tau4positive} 
this latter inequality implies~\eqref{eq:inequality_case1}. 
\end{proof}

\section{Proof of Theorem~\ref{thm:main}}\label{sec:proof}

From this point on the choices in~\eqref{eq:divisor} will be in effect. 
Choices of continuous parameters do not affect our results on homotopy types 
of maps, so this assumption is only made to make our arguments more concrete. 

\subsection{Geometry of period integrals}\label{sec:Geometry_period}

For arbitrary $s_2 > 0$ let $U_2 = U_2 (s_2) \subset S^3_1$ consist of the quadratic differentials $q$ satisfying the property 
\begin{equation*}
     0 \neq \pi_1 (q) - \pi_2 (q) \neq 
     \pm 2 \sqrt{s_2} (\sqrt{\tau_1} (q) - \sqrt{\tau_2}(q)).  
\end{equation*}
Then, $U_2$ is an everywhere dense analytic open subset. 
By definitions~\eqref{eq:tauj},~\eqref{eq:pij} the quantities $\pi_1(q), \pi_2 (q), \int_{\psi_2} \Re Z_+$ 
are defined for every $q\in U_2$ (up to an ambiguity of sign), and we have  
\begin{equation}\label{eq:action_Hopf_circle}
   \pi_j ( e^{\sqrt{-1} \varphi} q) = e^{\sqrt{-1} \varphi/2} \pi_j (q) , 
   \quad \tau_j ( e^{\sqrt{-1} \varphi} q) = e^{\sqrt{-1} \varphi/2} \tau_j (q).
\end{equation}
This shows in particular that $U_2$ consists of entire orbits of the Hopf map $t$, 
namely $q$ satisfies the defining conditions of $U_2$ if and only if the same 
holds for $e^{\sqrt{-1} \varphi} q$ for every $\varphi$. 
\begin{remark}
\begin{enumerate}
 \item 
To define $\pi_j$ one needs to make coherent choices of square roots of $Q$ 
over the points 
$$
     t(q) \in \mbox{Im} \psi_2 \cup \mbox{Im} \eta_1 \cup \mbox{Im} \eta_2 \cup \mbox{Im} \sigma_1 \cup \mbox{Im} \sigma_2; 
$$
for more details see the proof of Proposition~\ref{prop:twist_t1}.
\item 
The half-weight transformation~\eqref{eq:action_Hopf_circle} means in particular 
that at the end-points of the natural domain $[0,2\pi ]$ for $\varphi$
in the Hopf fibration, the quantities $\pi_j, \tau_j$ change sign. 
This amounts to changing the choice of square root $Z_+$ of $Q$. 
As we have already explained in Remark~\ref{rem:change_sign}, this change of 
choice does not affect the asymptotic behaviour that we have found for the 
complex twist Fenchel--Nielsen co-ordinates. 
\end{enumerate}
\end{remark}
An easy geometric argument using~\eqref{eq:action_Hopf_circle} shows that for every 
$q \in U_2$ there exists a unique value of $\varphi^* \in [0, 2\pi )$ fulfilling 
\begin{equation}\label{eq:phi*}
    \Re ( \pi_1 (e^{\sqrt{-1} \varphi^*} q) - \pi_2 (e^{\sqrt{-1} \varphi^*} q)) = 
    2 \sqrt{s_2} \Re (\sqrt{\tau_1} (e^{\sqrt{-1} \varphi^*} q) 
    - \sqrt{\tau_2} (e^{\sqrt{-1} \varphi^*} q)). 
\end{equation}
This specifies a unique point $e^{\sqrt{-1} \varphi^*} q$ in the Hopf fiber of $q$. 
Applying Proposition~\ref{prop:twist} to $e^{\sqrt{-1} \varphi^*} q$, the 
complex twist co-ordinate $[p_2\colon q_2]$ converges to $[0\colon 1]$ in case 
the further condtion~\eqref{eq:inequality3} of the proposition is fulfilled. 
Now, by the implicit function theorem, the point $e^{\sqrt{-1} \varphi^*} q$ 
depends smoothly on $q\in U_2$. Let 
\begin{align}\label{eq:section_S2}
 S_2\colon t( U_2 ) & \to S^3_1 \\
 [a:b] & \mapsto e^{\sqrt{-1} \varphi^* (q)} q \notag
\end{align}
be the corresponding smooth section of the Hopf fibration, where $q\in t^{-1} ([a:b])$ is an arbitrary point in the fiber of the 
Hopf fibration over $[a:b]$. 

\begin{prop}\label{prop:twist_t1}
Let $q = S_2 (t_1)$. Then $q$ belongs to $U_2 (s_2)$ for every $s_2 > 0$, 
and we have  $\Re ( \pi_1 (q) - \pi_2 (q) ) \neq 0$. 
Moreover, there exist distinct points $x_2,x_3 \in \CP1 \setminus D$ and 
$$
\rho = \rho (q, t_0, \ldots , t_4, x_2, x_3)>0
$$ 
such that for every $0 < s_2, s_3 < \rho$ the inequality~\eqref{eq:inequality3} holds. 
In particular, we then have $[p_2 : q_2]\to [0:1]$ as $R\to \infty$. 
\end{prop}

\begin{proof}
 We recall our choices of logarithmic points~\eqref{eq:divisor}. 
 In this case the quantities to study become elliptic integrals of the first kind with eccentricity $k$ in Legendre canonical form.
 
 By assumption we need to consider homogeneous polynomials $q$ of degree $6$ such that $t(q) = t_1 = 0$, i.e. of the form 
 $$
  q (z,w) = a z^2 (z^2 - 1) \left( z^2 - \frac 1{k^2} \right)
 $$
 for some coefficients $a \in S^1$. 
 The parameter  $a$ can be identified 
 with $e^{\sqrt{-1} \varphi}$ where $\varphi$ is the parameter of the Hopf fiber. 
 Using this form and~\eqref{eq:Hopf} we get that the corresponding quadratic differential reads as 
  \begin{equation}\label{eq:quadratic_differential}
  Q (z) = \frac{a z^2 (z^2 - 1) \left( z^2 - \frac 1{k^2} \right) \mbox{d} z^{\otimes 2}}{z^2 (z^2 - 1)^2 \left( z^2 - \frac 1{k^2} \right)^2} 
  = \frac{a \mbox{d} z^{\otimes 2}}{(z^2 - 1) \left( z^2 - \frac 1{k^2} \right)} .
  \end{equation}
 The square-root of $Q$ is then given by 
 $$
  Z_+ = \sqrt{a} \frac{\mbox{d} z}{\sqrt{(z^2 - 1) \left( z^2 - \frac 1{k^2} \right)}}  .
 $$
 Here, we need to be precise about the determination of the square roots: 
 for $a=1$ we choose $\sqrt{a} = 1$ and in the denominator we choose 
 \begin{equation}\label{eq:sign_Z}
 \sqrt{(z^2 - 1) \left( z^2 - \frac 1{k^2} \right)} \in 
 \begin{cases}
  \R_- & \mbox{if } z < - \frac 1k \\
  \sqrt{-1} \R_- & \mbox{if } - \frac 1k < z < -1 \\
  \R_+ & \mbox{if } -1 < z < 1 \\
  \sqrt{-1} \R_+ & \mbox{if } 1 < z < \frac 1k \\
  \R_- & \mbox{if }  \frac 1k < z 
 \end{cases}.
 \end{equation}
 By definition~\eqref{eq:tauj} we find 
 $$
  - \tau_3 = \tau_2 = \frac a{2\left( 1 - \frac 1{k^2} \right)} < 0, \quad \tau_1 = 0. 
 $$
 For simplicity, throughout the proof we will thus omit to spell out terms containing 
 $\tau_1$. 
 From now on until further mention we set $a = 1$. 
 According to the choice of determination of $Z_+$~\eqref{eq:sign_Z} and 
 the choice~\eqref{eq:sign_choice} of square roots of $\tau_j$ we have 
 \begin{equation}\label{eq:sign_tau}
   \sqrt{\tau_3} \in \R_- , \quad \sqrt{\tau_2} \in \sqrt{-1} \R_- .
 \end{equation}
 Moreover, we have 
 \begin{align}
    \pi_1 - \pi_2 & = \int_{t_2}^{t_1} \frac{\mbox{d} z}{\sqrt{(z^2 - 1) \left( z^2 - \frac 1{k^2} \right)}} \notag \\ 
    & = - k K(k) \label{eq:complete_elliptic_integral}
 \end{align}
 where 
 $$
  K(k) = \int_0^1 \frac{\mbox{d} t}{\sqrt{(1 - t^2) ( 1 - k^2 t^2 )}} > 0 
 $$
 is the complete elliptic integral of the first kind with eccentricity $k$. 
 In particular, as $\pi_1 - \pi_2 \in \R_-$ and 
 $- \sqrt{\tau_2} \in \sqrt{-1} \R$, we trivially have 
 $q \in U_2 (s_2 )$ for every choice of $s_2 > 0$. 
 Now, as $\Re (\pi_1 - \pi_2 ) = \pi_1 - \pi_2 \neq 0$, we get the first statement. 
 
 \begin{tikzpicture}
  \draw [thick, ->] (-5,0)--(5,0) node[right]{$x$};
  \fill (-4,0) circle(2pt) node[below]{$t_0 = -\frac 1k$};
  \fill (-2,0) circle(2pt) node[below]{$t_3 = -1$};
  \fill (0,0) circle(2pt) node[below]{$t_1= 0$};
  \fill (2,0) circle(2pt) node[below]{$t_2 = 1$};
  \fill (4,0) circle(2pt) node[below]{$t_4 =\frac 1k$};
  \draw (-2.3,0) circle(2pt) node[above]{$x_2 = - 1 - \varepsilon$};
  \draw (1.7,0) circle(2pt) node[above]{$x_3 = 1 -  \varepsilon$};
 \end{tikzpicture}
 
 Let us now pick 
 $$
  x_2 = - 1 - \varepsilon, \quad  x_3 = 1 - \varepsilon
 $$
 for some $0 < \varepsilon \ll 1$.  
 The integrand being an even function, decomposing the path $\psi_2$ as 
 $$
 \left[ - 1 - \varepsilon, 1 - \varepsilon \right] = \left[ - 1 - \varepsilon, -1  \right] \cup [-1, 0] \cup \left[ 0, 1 - \varepsilon \right] 
 $$
 we see that   
 \begin{align}
  \int_{\psi_2} \frac{\mbox{d} z}{\sqrt{(z^2 - 1) \left( z^2 - \frac 1{k^2} \right)}} & =  
  \int_{- 1 - \varepsilon}^{1 - \varepsilon} \frac{\mbox{d} z}{\sqrt{(z^2 - 1) \left( z^2 - \frac 1{k^2} \right)}}  \notag \\
  & = \int_{- 1 - \varepsilon}^{-1} \frac{\mbox{d} z}{\sqrt{(z^2 - 1) \left( z^2 - \frac 1{k^2} \right)}} + k [ K(k) +  F(1 - \varepsilon; k) ]  \label{eq:integral}
 \end{align}
 where 
 $$
  F(x; k) = \int_0^x \frac{\mbox{d} t}{\sqrt{(1 - t^2) ( 1 - k^2 t^2 )}} 
 $$
 is the incomplete elliptic integral of the first kind with eccentricity $k$. 
 By~\eqref{eq:sign_Z}, the first integral of~\eqref{eq:integral} belongs to 
 $\sqrt{-1} \R_-$.  
 For fixed $\varepsilon > 0$ let us choose $\rho > 0$ such that for all 
 $0 < s_2, s_3 < \rho$ we have 
 $$
  |4 \sqrt{s_2 \tau_2}| < \left\vert \int_{- 1 - \varepsilon}^{-1} \frac{\mbox{d} z}{\sqrt{(z^2 - 1) \left( z^2 - \frac 1{k^2} \right)}} \right\vert 
 $$
 and 
 $$
  |2 \sqrt{s_3 \tau_3}| < k [ K(k) +  F(1 - \varepsilon; k) ] .
 $$
 We then find that the complex number 
 \begin{equation}\label{eq:4th_quadrant}
  \int_{- 1 - \varepsilon}^{1 - \varepsilon} \frac{\mbox{d} z}{\sqrt{(z^2 - 1) \left( z^2 - \frac 1{k^2} \right)}}  
  - 2 ( 2 \sqrt{s_2 \tau_2}   - \sqrt{s_3 \tau_3}  )
 \end{equation}
 belongs to the fourth quadrant, i.e. 
 \begin{align}\label{eq:4th_quadrant_Re}
 \Re \left( \int_{\psi_2}  Z_+ (q)
    - 2 ( 2 \sqrt{s_2 \tau_2} (q) - \sqrt{s_3 \tau_3} (q) ) \right) & > 0 , \\
    \label{eq:4th_quadrant_Im}
  \Im \left( \int_{\psi_2}  Z_+ (q)
    - 2 ( 2 \sqrt{s_2 \tau_2} (q) - \sqrt{s_3 \tau_3} (q) ) \right) & < 0 .
 \end{align}

 
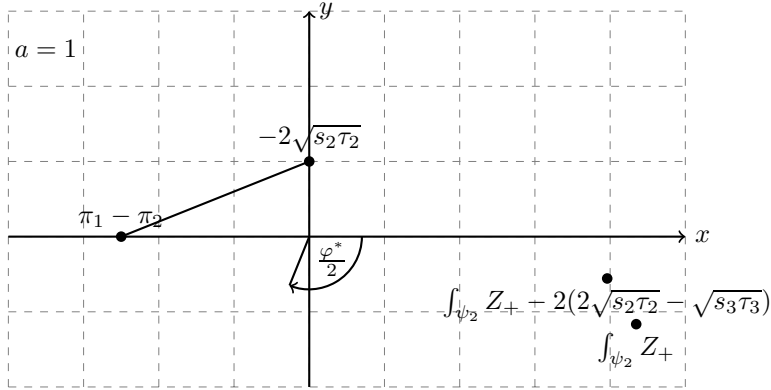
\begin{tikzpicture}[set style = {{help lines}+=[dashed]}]
  \draw[style = help lines] (-4,-2) grid (5,3);
  \draw [thick, ->] (0,-2)--(0,3) node[right]{$y$};
  \draw [thick, ->] (-4,0)--(5,0) node[right]{$x$};
  \fill (-15:4.5) circle(2pt) node[below]{$\int_{\psi_2} Z_+$};
  \fill (-8:4) circle(2pt) node[below]{$\int_{\psi_2} Z_+ - 
  2 ( 2 \sqrt{s_2 \tau_2} - \sqrt{s_3 \tau_3}  )$};
  \fill (180:2.5) circle(2pt) node[above]{$\pi_1 - \pi_2$};
  \fill (90:1) circle(2pt) node[above]{$- 2 \sqrt{s_2 \tau_2}$};
  \draw [thick] (180:2.5)--(90:1);
  \draw [thick, ->] (.7,0) arc (0:-112:.7);
  \draw [thick] (0,0)--(-112:.7);
  \draw (.3,-.3) node {$\frac{\varphi^*}2$};
  \draw (-3.5,2.5) node {$a = 1$};
 \end{tikzpicture}

 Now let us lift the assumption $a=1$, and let $a = e^{\sqrt{-1} \varphi} \in S^1$ vary. 
 By the action~\eqref{eq:action_Hopf_circle} and definition~\eqref{eq:phi*}, 
 $\varphi^*$ is the unique angle satisfying the property that when one simultaneously 
 rotates the vectors~\eqref{eq:complete_elliptic_integral} 
 and $- 2 \sqrt{s_2 \tau_2}$ by angle $\varphi^*/2$ then they get carried to vectors with equal real parts. As~\eqref{eq:complete_elliptic_integral} 
 is a negative real number and $- 2 \sqrt{s_2 \tau_2}$ is purely imaginary with positive imaginary part, the angle that achieves this 
 rotation satisfies $\varphi^*/2 \in (- \pi, - \pi /2)$. 
 (Indeed, in the limit $s_2 \to 0$ we have $\varphi^*/2 \to - \pi /2$.)

 
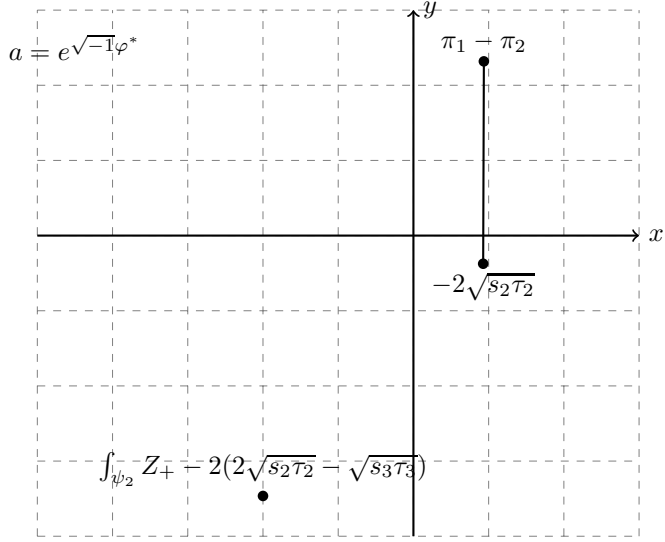
\begin{tikzpicture}[set style = {{help lines}+=[dashed]}]
  \draw[style = help lines] (-5,-4) grid (3,3);
  \draw [thick, ->] (0,-4)--(0,3) node[right]{$y$};
  \draw [thick, ->] (-5,0)--(3,0) node[right]{$x$};
  \fill (-120:4) circle(2pt) node[above]{$\int_{\psi_2} Z_+ - 
  2 ( 2 \sqrt{s_2 \tau_2} - \sqrt{s_3 \tau_3}  )$};
  \fill (68:2.5) circle(2pt) node[above]{$\pi_1 - \pi_2$};
  \fill (-22:1) circle(2pt) node[below]{$- 2 \sqrt{s_2 \tau_2}$};
  \draw [thick] (68:2.5)--(-22:1);
  \draw (-4.5,2.5) node {$a = e^{\sqrt{-1} \varphi^*}$};
 \end{tikzpicture}
 
 The second assertion of the proposition is equivalent to stating that this 
 rotation brings the complex number~\eqref{eq:4th_quadrant} to one in either 
 the second or third quadrant. Now, this assertion follows 
 from~\eqref{eq:4th_quadrant_Re}--\eqref{eq:4th_quadrant_Im} 
 by a straightforward geometric inspection (see the above figure). 
 Proposition~\ref{prop:twist} then implies the third assertion. 
\end{proof}

Similarly to $U_2$ we now define the dense open set 
$$
    U_3 = U_3 (s_4 )  \subset S^3_1
$$ 
consisting of the quadratic differentials $q$ satisfying the conditions 
\begin{equation*}
     0 \neq \pi_0 (q) - \pi_4 (q) \neq 
     \pm 2 \sqrt{s_4} (\sqrt{\tau_0} (q) - \sqrt{\tau_4}(q)).  
\end{equation*}
The action of the Hopf circle co-ordinate is again given by~\eqref{eq:action_Hopf_circle}, 
and analogously to~\eqref{eq:section_S2}, we now define 
\begin{equation}\label{eq:section_S4}
 S_3\colon t(U_3) \to S^3_1
\end{equation}
to be the unique section fulfilling for any $[a:b]\in t(U_3) \subset \CP1$ 
the condition 
$$
    | \Re ( \pi_0 ( S_3 ( [a:b] ) ) - \pi_4 ( S_3 ( [a:b] ) ) ) | = 
    2 \sqrt{s_4}  \Re (\sqrt{\tau_0} ( S_3 ( [a:b] ) ) - \sqrt{\tau_4}( S_3 ( [a:b] ) )). 
$$

We now state the analog of Proposition~\ref{prop:twist_t1} for 
the complex twist co-ordinates $[p_3 : q_3]$. We fix base points 
$x_2, x_3$ provided in the proof of Proposition~\ref{prop:twist_t1}. 

\begin{prop}\label{prop:twist_t1_bis}
Let $q' = S_3 (t_1)$. Then the quadratic differential $q'$ belongs to $U_3 (s_4)$
for every $s_4 > 0$ and we have  $\Re ( \pi_0 (q) - \pi_4 (q) ) \neq 0$. 
Moreover, there exists 
$$
    x_4 \in \CP1 \setminus (D\cup \{ x_2, x_3 \})
$$ 
and 
$$
    \rho' = \rho' (q, t_0, \ldots , t_4, x_4)>0
$$ 
such that for every $0 < s_3, s_4 < \rho'$ inequality~\eqref{eq:inequality4} holds. 
In particular, we then have $[p_3 : q_3]\to [0:1]$ as $R\to \infty$. 
\end{prop}

\begin{proof}
 Using the results of the proof of Proposition~\ref{prop:twist_t1}, we again have 
 \begin{equation}\label{eq:quadratic_differential'}
     q' (z,w) = a' z^2 (z^2 - 1) \left( z^2 - \frac 1{k^2} \right)
 \end{equation}
 for some $a' \in S^1$. 
 For the quadratic differential given by this formula with the choice $a' = 1$ we find 
 $$
    - \tau_4 = \tau_0 = \frac{k}{2\left( 1 - \frac 1{k^2} \right)} = k \tau_2, 
 $$
 so that 
 $$
    \tau_2 < \tau_0 < 0 < \tau_4 < \tau_3 . 
 $$
 Furthemore, we have 
 \begin{align*}
      \pi_0 - \pi_4 & = \int_{t_4}^{t_0} 
      \frac{\mbox{d} z}{\sqrt{(z^2 - 1) \left( z^2 - \frac 1{k^2} \right)}} \\ 
      & = - \int_{-\frac 1k}^{-1} - \int_{-1}^1 - \int_1^{\frac 1k} \\
    & = - 2 k K(k) 
 \end{align*}
 because the first and third integrals in the second line cancel each other, 
 the integrand being an odd function on the underlying intervals~\eqref{eq:sign_Z}. 
 To prove the first assertion, we merely realize that by the choices of 
 signs~\eqref{eq:sign_tau} the argument of the complex number 
 $\sqrt{\tau_0}  - \sqrt{\tau_4}$ is $\frac{3\pi}{4}$ while that of 
 $\pi_0 - \pi_4$ is $\pi$, so the origin is not incident to the line connecting 
 these points. 
 
 
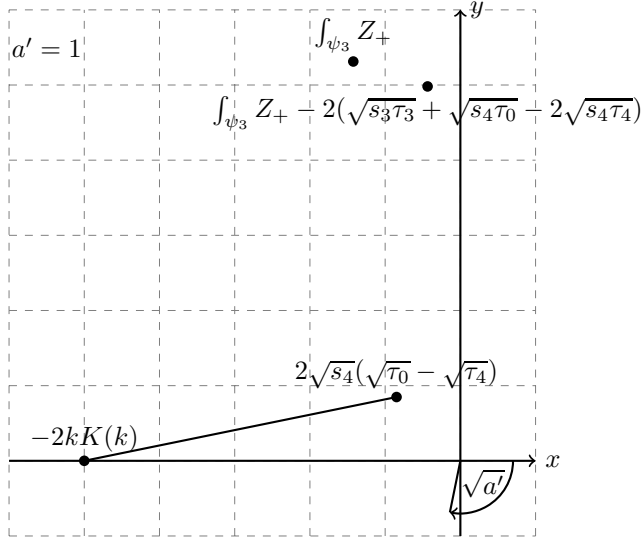
\begin{tikzpicture}[set style = {{help lines}+=[dashed]}]
  \draw[style = help lines] (-6,-1) grid (1,6);
  \draw [thick, ->] (0,-1)--(0,6) node[right]{$y$};
  \draw [thick, ->] (-6,0)--(1,0) node[right]{$x$};
  \fill (180:5) circle(2pt) node[above]{$- 2 k K(k) $};
  \fill (135:1.2) circle(2pt) node[above]
  {$2 \sqrt{s_4} (\sqrt{\tau_0} - \sqrt{\tau_4})$};
  \fill (105:5.5) circle(2pt) node[above]
  {$\int_{\psi_3} Z_+$};
  \fill (95:5) circle(2pt) node[below] 
  {$\int_{\psi_3} Z_+ - 2 ( \sqrt{s_3 \tau_3} + \sqrt{s_4 \tau_0} - 
    2 \sqrt{s_4 \tau_4} )$};
  \draw [thick] (180:5)--(135:1.2);
  \draw [thick, ->] (.7,0) arc (0:-101.5:.7);
  \draw [thick] (0,0)--(-101.5:.7);
  \draw (.3,-.3) node {$\sqrt{a'}$};
  \draw (-5.5,5.5) node {$a' = 1$};
 \end{tikzpicture}

 In order to prove the second assertion, we choose $x_4 > \frac 1k$ so that 
 $$
    \left\vert \int_{\frac 1k}^{x_4} 
    \frac{\mbox{d} z}{\sqrt{(z^2 - 1) \left( z^2 - \frac 1{k^2} \right)}} \right\vert  
    = 2 
    \int_{1-\varepsilon}^1 \frac{\mbox{d} z}{\sqrt{(z^2 - 1) \left( z^2 - \frac 1{k^2} \right)}}.
 $$
 Such a choice is clearly possible for fixed $\varepsilon \ll 1$, and we see that 
 \begin{equation}\label{eq:x4}
  x_4 \to \frac 1k = t_4 \qquad \mbox{as} \quad  \varepsilon \to 0.
 \end{equation}
 Note that by~\eqref{eq:sign_Z} the integral within the absolute value signs 
 on the left-hand side is negative, while the one on the right-hand side is positive. 
 We then see that the real part of 
 \begin{equation*}
      \int_{x_3}^{x_4} \frac{\mbox{d} z}{\sqrt{(z^2 - 1) \left( z^2 - \frac 1{k^2} \right)}} = \int_{1-\varepsilon}^1 + \int_1^{\frac 1k} + \int_{\frac 1k}^{x_4} 
 \end{equation*}
 is negative, and its imaginary part is positive. 
  
 Now, there exists $\rho' > 0$ such that for every $0 < s_3, s_4 < \rho'$ we have 
 \begin{align*}
  2 | \sqrt{s_4 \tau_4} | & < \min \left( \frac{kK(k)}2, \frac 14
   \left\vert \int_1^{\frac 1k} 
   \frac{\mbox{d} z}{\sqrt{(z^2 - 1) \left( z^2 - \frac 1{k^2} \right)}} \right\vert  \right)  , \\
   2 | \sqrt{s_3 \tau_3} | & < \min \left( \frac{kK(k)}2, \frac 14
   \left\vert \int_1^{\frac 1k} 
   \frac{\mbox{d} z}{\sqrt{(z^2 - 1) \left( z^2 - \frac 1{k^2} \right)}} \right\vert  \right) .
 \end{align*}
 The choice of $a'\in S^1$ giving rise to $q' = S_3 (t_1)$ is singled out 
 by the criterion that image under multiplication by $\sqrt{a'}$ of the line 
 segment connecting $\pi_0 - \pi_4$ and $\sqrt{\tau_0}  - \sqrt{\tau_4}$ be parallel 
 to the imaginary axis, the images of these points having positive real part. 
 The same geometric consideration as in the proof of Proposition~\ref{prop:twist_t1} 
 then concludes the proof.
 
 
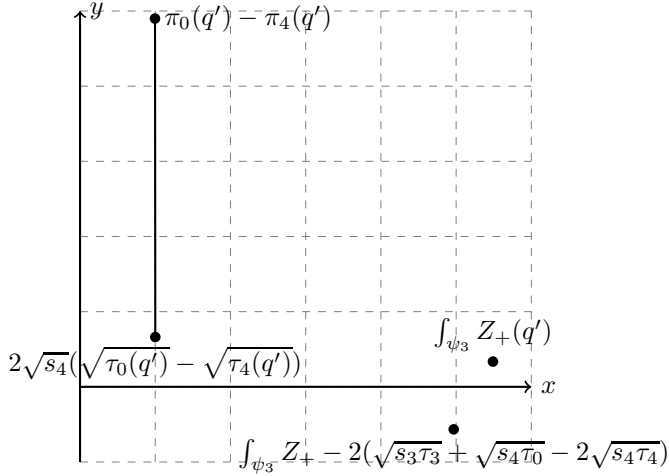
\begin{tikzpicture}[set style = {{help lines}+=[dashed]}]
  \draw[style = help lines] (0,-1) grid (6,5);
  \draw [thick, ->] (0,-1)--(0,5) node[right]{$y$};
  \draw [thick, ->] (0,0)--(6,0) node[right]{$x$};
  \fill (78.5:5) circle(2pt) node[right]{$\pi_0(q') - \pi_4(q')$};
  \fill (33.5:1.2) circle(2pt) node[below]
  {$2 \sqrt{s_4} (\sqrt{\tau_0 (q')} - \sqrt{\tau_4(q')})$};
  \draw [thick] (78.5:5)--(33.5:1.2);
  \fill (3.5:5.5) circle(2pt) node[above]{$\int_{\psi_3} Z_+ (q' )$};
  \fill (-6.5:5) circle(2pt) node[below] 
  {$\int_{\psi_3} Z_+ - 2 ( \sqrt{s_3 \tau_3} + \sqrt{s_4 \tau_0} - 
    2 \sqrt{s_4 \tau_4} )$};
 \end{tikzpicture}
 
\end{proof}

We now choose $x_2, x_3, x_4$ as in Propositions~\ref{prop:twist_t1} 
and~\ref{prop:twist_t1_bis}, and set $\rho'' = \min ( \rho, \rho')$ 
where $\rho, \rho' > 0$ are the scalars provided by the Propositions. 

\begin{prop}\label{prop:p2p30}
 There exist $0< s_2, s_3, s_4 < \rho''$ such that $S_2 (t_1) = S_3 (t_1)$. 
 Let us denote this point by $q^*$. 
 Then, for the choice $q^*\in S^3_1$, we have $[p_2 : q_2]\to [0:1]$ 
 and $[p_3 : q_3]\to [0:1]$ as $R\to \infty$. 
\end{prop}

\begin{proof}
 Making use of the notations of Subsection~\ref{subsec:ramification}, the Hopf 
 co-ordinates $\theta, \phi$ of $S_2 (t_1)$ and $S_3 (t_1)$ agree by construction, 
 because $q,q'$ both map to $t_1$ by the Hopf map. 
 In order to achieve $S_2 (t_1) = S_3 (t_1)$, all we need is to make sure that the 
 Hopf circle co-ordinate $\varphi$ of the point $q$ in Proposition~\ref{prop:twist_t1} 
 agrees with the one of the point $q'$ in Proposition~\ref{prop:twist_t1_bis}; 
 the common point that they define will then be $q^*\in S^3_1$. 
 Equivalently, the condition may be rephrased by requiring that equal
 rotations bring the indicated straight line segment on the figure corresponding to the case $a=1$ in the proof of
 Proposition~\ref{prop:twist_t1} and the one corresponding to the case $a'=1$ in 
 the proof of Proposition~\ref{prop:twist_t1_bis} into vertical position. 
 Still equivalently, the named segments need to be parallel to each other. 
 
 In order to find $0 < s_2, s_4 < \rho''$ fulfilling this condition, fix any 
 $0 < s_2, s_4 < \rho''$. If the line segments are parallel then we are done. 
 Otherwise there are two cases: if the slope of the line connecting 
 $- k K(k) $ to $-2 \sqrt{s_2 \tau_2}$ is greater than that of the line connecting 
 $- 2 k K(k) $ to $2 \sqrt{s_4} (\sqrt{\tau_0} - \sqrt{\tau_4})$, then we need to 
 decrease $s_2$ (keeping $s_4$ fixed) until they become parallel. 
 If the relation between the slopes is the opposite then we need to decrease $s_4$ 
 (keeping $s_2$ fixed) until they become parallel. 
 In both cases a simple application of the mean value theorem shows that there exists 
 a unique value of $s_2$ or $s_4$ with the required property. 
 
 Propositions~\ref{prop:twist_t1} and~\ref{prop:twist_t1_bis} now show the second 
 statement. 
\end{proof}

\subsection{Proof of Theorem~\ref{thm:main}}
We are in position to prove our main result Theorem~\ref{thm:main}. 

Let us first reformulate a few of our results obtained thus far. 
According to Subsection~\ref{subsec:weight_filtration}, the highest graded piece 
$\Gr^W_{8} H^4 (\Mod_{\Betti} , \C )$ of the MHS on the cohomology of the 
character variety is spanned by $0$-cycles in the union $\tilde{D}^4$ of 
quadruple intersections of the compactifying divisor components; 
clearly, $\tilde{D}^4$ is a finite union of points in $\overline{\Mod}_{\Betti}$. 
It will turn out that these cycles also govern $\Gr^W_{2k} H^k$ for 
all $0 \leq k \leq 4$. 
Let us denote by $D_1, D_2, D_3, D_4$ the divisor components of 
$\overline{\Mod}_{\Betti}\setminus {\Mod}_{\Betti}$ given in order by the equations 
$$
    l_2 = \infty, \quad [p_2\colon q_2] = [0\colon 1], \quad 
    l_3 = \infty, \quad [p_3\colon q_3] = [0\colon 1]. 
$$
Let us denote by 
$$
    Q^* = D_1 \cap D_2 \cap D_3 \cap D_4 \in \tilde{D}^4 \subset \overline{\Mod}_{\Betti}
$$
their intersection point and fix a punctured neighbourhood $U(Q^* )$ of $Q^*$ in 
$\Mod_{\Betti}$. It follows from the choices made in the proofs of 
Propositions~\ref{prop:twist_t1} and~\ref{prop:twist_t1_bis}, 
and in particular from~\eqref{eq:x4} that 
$$
    \lim_{\varepsilon \to 0} x_2 = t_3, \qquad \lim_{\varepsilon \to 0} x_3 = t_2, 
    \qquad \lim_{\varepsilon \to 0} x_4 = t_4.
$$
We may assume by homotopy that 
$$
    x_2 = t_3, \qquad x_3 = t_2, \qquad x_4 = t_4. 
$$
It follows from Proposition~\ref{prop:monodromy_zeta2} that as $\L_{\E}$ ranges 
over $H^{-1} (Rq)$ for any $q\in U_2 \cap U_3$ along a path so that 
$$
    \int_{\eta_2 - \eta_1} B_{\L_{\E }}
$$
starts at $0$ and ends at $2\pi\sqrt{-1}$, the point $\RH \circ \psi (\E , \theta  )$ 
winds around $D_1$ once (either in positive or in negative direction). 
Similarly, from Proposition~\ref{prop:monodromy_zeta3} we see that if 
$$
    \int_{\eta_0 - \eta_4} B_{\L_{\E }} 
$$
varies from $0$ to $2\pi\sqrt{-1}$,  $\RH \circ \psi (\E , \theta  )$ 
winds around $D_3$ once. 
We make choices as in Proposition~\ref{prop:p2p30}, in particular we study the 
Hitchin fiber over $Rq^*$ for $R\gg 0$. 
There exists a real $2$-parameter family of $(\E , \theta ) \in H^{-1} (Rq^*)$ 
such that 
\begin{align}
    \int_{\eta_2 - \eta_1} B_{\L_{\E }} & \equiv 0 \pmod{\pi\sqrt{-1}}, 
    \label{eq:congruence_1} \\
    \int_{\eta_0 - \eta_4} B_{\L_{\E }} & \equiv 0 \pmod{\pi\sqrt{-1}}.
    \label{eq:congruence_2}
\end{align}
It follows from 
Propositions~\ref{prop:monodromy_zeta2},~\ref{prop:monodromy_zeta3},~\ref{prop:twist},~\ref{prop:twist3} 
that if $R$ is chosen sufficiently large, then for any $(\E , \theta ) \in H^{-1} (Rq^*)$ 
satisfying~\eqref{eq:congruence_1}--\eqref{eq:congruence_2} the Fenchel--Nielsen 
co-ordinates 
$$
    l_2 (\E , \theta ) , \quad [p_2(\E , \theta )\colon q_2(\E , \theta )], 
    \quad l_3(\E , \theta ), \quad [p_3(\E , \theta )\colon q_3(\E , \theta )]
$$
of $\RH \circ \psi (\E , \theta  )$ belong to $U(Q^* )$. 
Fix $R\gg 0$ so that this holds. Moreover, the same results also imply that the 
phase factors of the Fenchel--Nielsen co-ordinates of $\RH \circ \psi (\E , \theta  )$ 
defining $D_1, D_2, D_3, D_4$ are in this order given by the following expressions: 
\begin{align*}
  - \exp & \left( 2 \left\vert \int_{t_1}^{t_2} B_{\L_{(\E , \theta  )}}\right\vert \right), \\
  \exp & \left( - 2 \int_{t_3}^{t_2} B_{\L_{(\E , \theta  )}} \right), \\
  - \exp & \left( 2 \left\vert \int_{t_0}^{t_4} B_{\L_{(\E , \theta  )}}\right\vert \right), \\
  \exp &\left( -2 \int_{t_2}^{t_4} B_{\L_{(\E , \theta  )}} \right), 
\end{align*}
the respective contours being $\eta_2 - \eta_1, \psi_2, \eta_4 - \eta_0,  \psi_3$. 
Notice that there exists an involution-equivariant base $A_1, A_2, B_1, B_2$ of 
$H_1(X_{q^*}, \Z)$ that project down to the above contours by $p_{q^*}$. 
It follows (possibly up to switching the orientation of some of the generators 
$A_i, B_i$) that the above quantities can be rewritten as 
\begin{align}
  - \exp & \left( \int_{A_1} B_{\L_{(\E , \theta  )}} \right), \label{eq:phase1} \\
  \exp & \left( \int_{B_1} B_{\L_{(\E , \theta  )}} \right), \label{eq:phase2}\\
  - \exp & \left( \int_{A_2} B_{\L_{(\E , \theta  )}} \right), \label{eq:phase3}\\
  \exp &\left( \int_{B_2} B_{\L_{(\E , \theta  )}} \right). \label{eq:phase4}
\end{align}
We infer from Subsection~\ref{subsec:spectral_curve} that the image of 
$H^{-1} (R q^* )$ under $\RH \circ \psi$ is homotopic to a torus $T^4$ generating 
$H_4 ( U(Q^* ), \Z )$. 
Recall from Subsection~\ref{subsec:perverse_filtration} that we have 
$$
    \Gr_P^{-k-2} H^k (\Mod_{\Dol}, \Q ) \cong 
    \operatorname{Im} ( H^k (\Mod_{\Dol}, \Q ) \to H^k (H^{-1} (Y_{-2}), \Q) )
$$
where $H^{-1} (Y_{-2})$ is the generic Hitchin fiber. We choose the affine flag 
so that $Y_{-2} = \{ Rq^* \}$. 
For every $0 \leq k \leq 4$ and any subset $I\subset \{ 1,2,3,4 \}$ 
with $|I| = 4-k$ one may define a $k$-dimensional subtorus $T_I^k$ in 
$H^{-1} (Y_{-2})$ by fixing the phases corresponding to the divisor components 
$D_i$ with $i\in I$. 
Let us assume that $T_I^k$ defines a non-trivial homology class in 
$H_k (\Mod_{\Dol}, \Z )$. 
Such classes are precisely the ones that generate $\Gr_P^{-k-2} H^k (\Mod_{\Dol}, \Q )$. 
It is then easy to see that image of $T_I^k$ under $\RH \circ \psi$ is homotopic to 
a normal torus at the generic point of the intersection 
$$
    \bigcap_{j\in \{ 1,2,3,4 \} \setminus I} D_j 
$$
of the remaining $k$ divisor components. 
According to the conventions of Subsection~\ref{subsec:weight_filtration}, 
$\RH \circ \psi (T_I^k)$ then defines a class in $W_{-2k} H_k ( U(Q^* ), \Z )$ 
(that is non-trivial by assumption), and the dual cohomology class gives a 
non-trivial class in $W_{2k} H^k ( U(Q^* ), \Z )$. Since the map 
$$
    H^k (\Mod_{\Betti}, \C ) \to H^k ( U(Q^* ), \C ) 
$$
preserves $W$ strictly, this finishes the proof.

\end{document}